\documentclass{amsart}

\usepackage{amsfonts}
\usepackage{amssymb}
\usepackage{tikz}
\usetikzlibrary{calc}
\usepackage{caption}
\usepackage{enumitem}
\makeatletter
\newcommand{\mylabel}[2]{#2\def\@currentlabel{#2}\label{#1}}
\makeatother
\usepackage{graphicx}
\usepackage{bm}
\usepackage[
	pdftitle={},
	pdfauthor={},
	ocgcolorlinks,
	linkcolor=linkblue,
	citecolor=linkred,
	urlcolor=linkred]
{hyperref}
\usepackage{color}
\definecolor{linkred}{rgb}{0.75,0,0}
\definecolor{linkblue}{rgb}{0,0,0.75}

\theoremstyle{plain}
\newtheorem{theorem}{Theorem}
\newtheorem{proposition}{Proposition}[section]
\newtheorem{lemma}[proposition]{Lemma}
\newtheorem{thm}[proposition]{Theorem}

\newtheorem{corollary}[theorem]{Corollary}
\newtheorem{cor}[proposition]{Corollary}

\newcommand{\bt}{\begin{theorem}}
\newcommand{\et}{\end{theorem}}
\newcommand{\floor}[1]{\lfloor\hspace{-.4mm} #1 \hspace{-.4mm}\rfloor}

\theoremstyle{definition}
\newtheorem{definition}[proposition]{Definition}
\newtheorem{remark}[proposition]{Remark}
\newcommand{\beq}{\begin{equation}}
\newcommand{\eeq}{\end{equation}}
\newcommand{\bl}{\begin{lemma}}
\newcommand{\el}{\end{lemma}}

\newcommand{\cal}{\mathcal}
\newcommand{\Res}{\mathop{\,\rm Res\,}}

\newcommand{\ca}{\mathcal{A}}

\newcommand{\cc}{\mathcal{C}}
\newcommand{\cd}{\mathcal{D}}

\newcommand{\cl}{\mathcal{L}}
\newcommand{\ce}{\mathcal{E}}
\newcommand{\co}{\mathcal{O}}
\newcommand{\cp}{\mathcal{P}}

\newcommand{\ct}{\mathcal{T}}
\newcommand{\cu}{\mathcal{U}}
\newcommand{\cv}{\mathcal{V}}
\newcommand{\cf}{{\cal F}}
\newcommand{\ba}{\mathbb{A}}
\newcommand{\bc}{\mathbb{C}}
\newcommand{\bh}{\mathbb{H}}

\newcommand{\bn}{\mathbb{N}}
\newcommand{\bp}{\mathbb{P}}
\newcommand{\bq}{\mathbb{Q}}
\newcommand{\br}{\mathbb{R}}
\newcommand{\bz}{\mathbb{Z}}
\newcommand{\tr}{\text{tr}\hspace{.5mm}}

\newcommand*{\Cdot}{\raisebox{-0.5ex}{\scalebox{1.8}{$\cdot$}}}
\newcommand{\LL}{\boldsymbol{L}}

\newcommand{\sgn}{\mathop{\mathrm{sgn}}}

\newcommand{\modm}{\cal M}
\newcommand{\un}{1\!\!1}

\allowdisplaybreaks

\begin{document}
	
\title[Enumerative geometry via super Riemann surfaces]{Enumerative geometry via the moduli space of super Riemann surfaces}
\author{Paul Norbury}
\address{School of Mathematics and Statistics, University of Melbourne, VIC 3010, Australia}
\email{\href{mailto:norbury@unimelb.edu.au}{norbury@unimelb.edu.au}}
\thanks{}
\date{\today}

\begin{abstract}
In this paper we relate volumes of moduli spaces of super Riemann surfaces to integrals over the moduli space of stable Riemann surfaces $\overline{\cal M}_{g,n}$.  This allows us to prove via algebraic geometry a recursion between the volumes of moduli spaces of super hyperbolic surfaces previously proven via super geometry techniques by Stanford and Witten.  The recursion between the volumes of moduli spaces of super hyperbolic surfaces is proven to be equivalent to the property that a generating function for the intersection numbers of a natural collection of cohomology classes $\Theta_{g,n}$ with tautological classes on $\overline{\cal M}_{g,n}$ is a KdV tau function.  This is analogous to Mirzakhani's proof of the Kontsevich-Witten theorem, which relates a generating function for the intersection numbers of tautological classes on $\overline{\cal M}_{g,n}$ to KdV, using volumes of moduli spaces of hyperbolic surfaces.
\end{abstract}

\maketitle

\tableofcontents

\section{Introduction}  \label{sec:intro}

Mumford initiated a systematic approach to calculating intersection numbers of tautological classes on the moduli space of stable Riemann surfaces $\overline{\modm}_{g,n}$ in \cite{MumTow}.  Witten conjectured a recursive structure on a collection of these intersection numbers \cite{WitTwo} and Kontsevich proved the conjecture in \cite{KonInt}, now known as the Kontsevich-Witten theorem.  Other proofs followed in \cite{KLaAlg,MirWei,OPaGro}.  The proof by Mirzakhani \cite{MirWei} deduced the Kontsevich-Witten theorem by proving recursion relations between Weil-Petersson volumes of moduli spaces of hyperbolic surfaces, defined using the top power of the Weil-Petersson symplectic form $\omega^{WP}$.  Wolpert had proven earlier in \cite{WolHom,WolWei} that $\omega^{WP}$ extends from the non-compact moduli space of hyperbolic surfaces to the compact moduli space of stable curves, and related it to a tautological cohomology class, $\kappa_1\in H^2(\overline{\modm}_{g,n},\bq)$, which was studied by Mumford in \cite{MumTow}.  This enabled Mirzakhani to relate volume integrals over $\modm_{g,n}$ to cohomological calculations over $\overline{\modm}_{g,n}$.  

Stanford and Witten \cite{SWiJTG} proved recursion relations between volumes of moduli spaces of super hyperbolic surfaces using methods analogous to those of Mirzakhani.  In this paper we prove these recursion relations, given by \eqref{volrec} below, via algebro-geometric methods.  We achieve this by expressing volumes of moduli spaces of super hyperbolic surfaces in terms of cohomology classes over the moduli space of stable curves, analogous to Wolpert's results.  The volumes are expressed in terms of classes $\Theta_{g,n}\in H^*(\overline{\modm}_{g,n},\bq)$ previously studied by the author \cite{NorNew}.

Super Riemann surfaces have been studied over the last thirty years \cite{CRaSup,FKPReg,LRoMod,RSVGeo,SWiJTG,WitNot}.  Underlying any super Riemann surface is a Riemann surface equipped with a spin structure.  
The moduli space of super Riemann surfaces can be defined algebraically, complex analytically and using hyperbolic geometry, building on the same approaches to the moduli space of Riemann surfaces.  The last of these approaches, used in the work of Stanford and Witten \cite{SWiJTG}, regards a super Riemann surface as a super hyperbolic surface, which is a quotient of super hyperbolic space $\widehat{\bh}$ defined in \ref{suphyp}.  In this paper we consider Riemann surfaces of finite type $\Sigma=\overline{\Sigma}-\{p_1,...,p_n\}$ where $\overline{\Sigma}$ is a compact curve containing distinct, labeled points $p_i$ that define a divisor $D=\{p_1,...,p_n\}\subset\overline{\Sigma}$.   A boundary component of $\Sigma$ is defined to be a punctured open disk embedded in $\Sigma$ which is a deleted disk neighbourhood in $\overline{\Sigma}$ of any given $p_i\in\overline{\Sigma}$.

A Riemann surface equipped with a spin structure, or spin surface, has a well-defined square root bundle, $T_\Sigma^\frac12$, of the tangent bundle, so that $T_\Sigma^\frac12\otimes_\bc T_\Sigma^\frac12\cong T_\Sigma$, which is also a real subbundle of the rank two bundle of spinors $T_\Sigma^\frac12\otimes_\br\bc\cong S_\Sigma$.  It is a flat $SL(2,\br)$-bundle, and the flat structure defines the sheaf of locally constant sections of $T_\Sigma^\frac12$ with sheaf cohomology $H^1_{dR}(\Sigma,T_\Sigma^\frac12)$.  We  require that the trace of the holonomy of the flat $SL(2,\br)$-bundle around any boundary component is negative, known as a {\em Neveu-Schwarz} boundary component, although we will occasionally also need to consider general boundary components---see Definition~\ref{NS}.   The deformation theory of a super Riemann surface with underlying spin surface $\Sigma$ defines a natural bundle
\[E_{g,n}\to\modm_{g,n,\vec{o}}^{\text{spin}},\qquad E_{g,n}|_\Sigma=H^1_{dR}(\Sigma,T_\Sigma^\frac12)\]
over the moduli space of smooth genus $g$ spin Riemann surfaces $\Sigma=\overline{\Sigma}-\{p_1,...,p_n\}$ with only Neveu-Schwarz boundary components.  The moduli spaces of spin curves, or Riemann surfaces, $\modm_{g,n,\vec{o}}^{\text{spin}}$ and spin hyperbolic surfaces $\modm_{g,n,\vec{o}}^{\text{spin}}(L_1,...,L_n)$, together with the natural diffeomorphisms between them, are defined in Definitions~\ref{def:modspin}, \ref{modspace} and \eqref{modiffeo}.  The vector $\vec{o}=(0,...,0)\in\{0,1\}^n$ in the subscript denotes the condition that all boundary components are Neveu-Schwarz.  More generally, vectors $\vec{\epsilon}\in\{0,1\}^n$ denote different connected components of the moduli space, defined in Definition~\ref{modspace}.  The bundle $E_{g,n}$ can be defined over each of these connected components however we will not consider that case in this paper.   

The sheaf of smooth sections of the exterior algebra of the dual bundle $E_{g,n}^\vee$ defines the moduli space of super Riemann surfaces as a smooth supermanifold.\footnote{Donagi and Witten proved in \cite{DWiSup} that the moduli space of super Riemann surfaces as a complex supermanifold cannot be represented as the sheaf of holomorphic sections of an exterior algebra of a bundle over the moduli space of Riemann surfaces.}      
The group $H^1_{dR}(\Sigma,T_\Sigma^\frac12)$ can be calculated via the cohomology of the twisted de Rham complex defined by the flat connection that defines the flat bundle $T_\Sigma^\frac12$.

The volume of the moduli space of super hyperbolic surfaces is defined via integration of a top power of a super symplectic form.  It reduces via a rather general super integration argument, \cite{SWiJTG}, to integration of the Euler form of a canonical connection on $E_{g,n}^\vee$ combined with the Weil-Petersson symplectic form over the moduli space $\modm_{g,n,\vec{o}}^{\text{spin}}(L_1,...,L_n)$ of spin hyperbolic surfaces with Neveu-Schwarz geodesic boundary components of lengths $L_1,...,L_n$.  For the purposes of this paper, we take this to be the definition of the volume of the moduli space of super hyperbolic surfaces.
\begin{equation}  \label{supvol}
\widehat{V}^{WP}_{g,n}(L_1,...,L_n):=\int_{\modm_{g,n,\vec{o}}^{\text{spin}}(L_1,...,L_n)}e(E_{g,n}^\vee)\exp\omega^{WP}
\end{equation}
where $e(E_{g,n}^\vee)$ is a differential form given by the Euler form of the bundle $E_{g,n}^\vee$ with respect to a canonical connection on $E_{g,n}^\vee$ defined in Section~\ref{sec:eulerform} using the hyperbolic metric.

One key result of this paper 
is a construction of a natural extension of the bundle $E_{g,n}$ to the moduli space $\overline{\modm}_{g,n,\vec{o}}^{\text{spin}}$, of genus $g$ stable spin curves with $n$ Neveu-Schwarz labeled points, on which the natural Euler form $e(E_{g,n}^\vee)$ extends to represent the Euler class of a bundle.  
The extension of the bundle $E_{g,n}$ and its Euler form to a compactification is a crucial ingredient for enumerative methods such as the calculation of volumes via intersection theory of cohomology classes, and in particular leads to the recursion in Theorem~\ref{main} below.  

A stable spin curve is a stable orbifold curve with $n$ labeled points $(\cc,D)$, equipped with a spin structure $\theta$ which is an orbifold line bundle satisfying 
\[\theta^2=\omega_{\cc}^{\text{log}}=\omega_\cc(D).\]  
The points of $D=\{p_1,...,p_n\}$, and nodal points of $\cc$ are orbifold points with isotropy group $\bz_2$---see Section~\ref{sec:theta}.  The bundle $\theta$ defines a representation $\bz_2\to\bz_2$ at each point $p_i$ and each nodal point, and we require that it is the unique non-trivial representation at each point $p_i$, which is known as a Neveu-Schwarz point, and any representation at nodal points.   There is a map from $\cc$ to its underlying coarse curve 
which forgets the orbifold structure.  
When $\cc$ is smooth, $\cc-D=\Sigma$ is a Riemann surface and there is an isomorphism of vector bundles $\theta^\vee|_\Sigma\cong T_\Sigma^\frac12$, where as usual $(\cdot)^\vee$ denotes the dual bundle.  Using a theorem of Simpson \cite{SimHig,SimHar} applied to the rank two spinor bundle equipped with a natural Higgs field we prove in Section~\ref{sec:higgs} a canonical isomorphism when $\cc$ is smooth and the spin structure has only Neveu-Schwarz boundary components/labeled points:
\begin{equation}  \label{canis}
H^1_{dR}(\Sigma,T_\Sigma^\frac12)\cong H^1(\cc,\theta^\vee)^\vee.
\end{equation} 
The isomorphism \eqref{canis} is non-trivial even in the case $D=\varnothing$ where $\theta^\vee\cong T_\Sigma^\frac12$ as vector bundles.  The left hand side of \eqref{canis} uses the sheaf of locally constant sections while the right hand side  uses the sheaf of locally holomorphic sections, and we take the sheaf cohomology in both cases.   The difference between the sheaf structures on each side of \eqref{canis} is demonstrated most clearly in the non-compact case, where the sheaf of locally holomorphic sections of $\theta^\vee|_\Sigma$ is trivial, whereas the sheaf of locally constant sections of $T_\Sigma^\frac12$ is non-trivial, detected by $H^1_{dR}(\Sigma,T_\Sigma^\frac12)\neq 0$. 
The push-forward of $\theta^\vee$ from $\cc$ to $\overline{\Sigma}$ is $T_{\overline{\Sigma}}^\frac12(-D)$, since the non-trivial representation induced by $\theta^\vee$ at each point of $D$ forces the local sections to vanish on $D$, and $T_{\overline{\Sigma}}^\frac12(-D)$ embeds in a parabolic bundle, as described in \ref{noncompact}.  In particular, we can express \eqref{canis} in terms of the coarse curve $(\cc,D)\to(\overline{\Sigma},D)$ via
$H^1(\cc,\theta^\vee)\cong H^1(\overline{\Sigma},T_{\overline{\Sigma}}^\frac12(-D))$.   One particularly satisfying aspect of applying Simpson's parabolic Higgs bundles techniques to the pair $(\overline{\Sigma},D)$ is that it naturally gives rise to the orbifold curve $(\cc,D)\to(\overline{\Sigma},D)$.  Parabolic bundles over the coarse curve $\overline{\Sigma}$ correspond to the push-forward of bundles over $\cc$, \cite{BodRep,FStSei}. 

The cohomology groups $H^1(\cc,\theta^\vee)$ are well-defined on any stable spin curve $(\cc,\theta)$ and $\dim H^1(\cc,\theta^\vee)$ is locally constant on $\overline{\modm}_{g,n,\vec{o}}^{\text{spin}}$, hence the bundle $E_{g,n}\to\modm_{g,n,\vec{o}}^{\text{spin}}$ is the restriction of a bundle $\widehat{E}_{g,n}\to\overline{\modm}_{g,n,\vec{o}}^{\text{spin}}$ with fibres $H^1(\cc,\theta^\vee)$.   The sheaf of smooth sections of the exterior algebra of $\widehat{E}_{g,n}^\vee$ gives  the compactification of the moduli space of super Riemann surfaces studied by Witten in \cite[Section 6]{WitNot}.

Under the forgetful map $p:\overline{\modm}_{g,n,\vec{o}}^{\text{spin}}\to\overline{\modm}_{g,n}$, define the push-forward classes 
\[\Theta_{g,n}:=(-1)^n2^{g-1+n}p_*c_{2g-2+n}(\widehat{E}_{g,n})\in H^{4g-4+2n}(\overline{\modm}_{g,n})\] 
for $g\geq 0$, $n\geq 0$ and $2g-2+n>0$.  These classes are shown in \cite{NorNew} to pull back naturally under the gluing maps 
\[
\overline{\modm}_{g-1,n+2}\stackrel{\phi_{\text{irr}}}{\longrightarrow}\overline{\modm}_{g,n},\qquad\overline{\modm}_{h,|I|+1}\times\overline{\modm}_{g-h,|J|+1}\stackrel{\phi_{h,I}}{\longrightarrow}\overline{\modm}_{g,n},\quad I\sqcup J=\{1,...,n\}
\] 
and the forgetful map 
$
\overline{\modm}_{g,n+1}\stackrel{\pi}{\longrightarrow}\overline{\modm}_{g,n}
$ as follows.
\begin{equation}  \label{glue}
\phi_{\text{irr}}^*\Theta_{g,n}=\Theta_{g-1,n+2},\quad \phi_{h,I}^*\Theta_{g,n}=\Theta_{h,|I|+1}\otimes \Theta_{g-h,|J|+1},  
\end{equation}
\begin{equation} \label{forget}
\Theta_{g,n+1}=\psi_{n+1}\cdot\pi^*\Theta_{g,n}
\end{equation}
where $\psi_{n+1}\in H^2(\overline{\modm}_{g,n+1},\bq)$ is a tautological class, defined in \eqref{psiclass} in Section~\ref{sec:theta}.
Properties \eqref{glue}, \eqref{forget} and a single calculation $\int_{\overline{\modm}_{1,1}}\Theta_{1,1}=\frac18$ are enough to uniquely determine the intersection numbers
\[ \int_{\overline{\modm}_{g,n}}\hspace{-2mm}\Theta_{g,n}\prod_{i=1}^n\psi_i^{m_i}\prod_{j=1}^N\kappa_j^{\ell_j}\]
via a reduction argument---see \eqref{kappaclass} for the definition of $\kappa_j$ and Section~\ref{sec:theta} for further details. In particular, we restrict to the case of only $\kappa_1$ classes.

Wolpert \cite{WolHom,WolWei} proved that $\omega^{WP}$ extends from $\modm_{g,n}$ to a current $\tilde{\omega}^{WP}$ defined on $\overline{\modm}_{g,n}$, with cohomology class $[\tilde{\omega}^{WP}]=2\pi^2\kappa_1\in H^2(\overline{\modm}_{g,n},\br)$.   More generally, over the moduli space $\modm_{g,n}(L_1,...,L_n)$ of hyperbolic surfaces with geodesic boundary components of lengths $L_1,...,L_n$, Mirzakhani \cite{MirWei} proved that the extension of the Weil-Petersson form to a natural compactification of $\modm_{g,n}(L_1,...,L_n)$ by nodal surfaces, which is homeomorphic to $\overline{\modm}_{g,n}$, has cohomology class 
$[\tilde{\omega}^{WP}]=2\pi^2\kappa_1+\frac12\sum_{i=1}^n L_i^2\psi_i.$ 
In particular, the Weil-Petersson volumes coincide with intersection numbers:
\[V^{WP}_{g,n}(L_1,...,L_n)=\int_{\modm_{g,n}(L_1,...,L_n)}\hspace{-2mm}\exp\omega^{WP}=\int_{\overline{\modm}_{g,n}}\exp(2\pi^2\kappa_1+\frac12\sum_{i=1}^n L_i^2\psi_i).\]
This relationship between the integral of a measure over a non-compact moduli space on the left hand side and the evaluation of cohomology classes defined over a compactification of the moduli space via algebraic geometry on the right hand side proves to be powerful.  In this paper we produce an analogous relationship involving super volumes.
Define the polynomials
\begin{equation}  \label{voltheta} 
V^{\Theta}_{g,n}(L_1,...,L_n):=\int_{\overline{\modm}_{g,n}}\hspace{-2mm}\Theta_{g,n}\exp\left\{2\pi^2\kappa_1+\frac12\sum_{i=1}^n L_i^2\psi_i\right\}.
\end{equation}
\begin{theorem}  \label{volequal}
\[
\widehat{V}^{WP}_{g,n}(L_1,...,L_n)=2^{1-g-n}V^{\Theta}_{g,n}(L_1,...,L_n).
\]
\end{theorem}
Theorem~\ref{volequal} proves that the total measure of the non-compact moduli space of smooth spin hyperbolic surfaces can be calculated using the intersection of cohomology classes on the moduli space of stable surfaces (with spin structures forgotten).
The normalisation factor $2^{1-g-n}$ is not important---without it the properties \eqref{glue} and \eqref{forget} would be less elegant.  
The proof of Theorem~\ref{volequal} requires an extension of $E_{g,n}$ and its natural Euler form to $\overline{\modm}_{g,n,\vec{o}}^{\text{spin}}$, proven in Section~\ref{sec:hypspin}, combined with Wolpert's extension of $\omega^{WP}$ to $\overline{\modm}_{g,n}$ which naturally lifts to $\overline{\modm}_{g,n,\vec{o}}^{\text{spin}}$. 
The polynomial $V^{\Theta}_{g,n}(L_1,...,L_n)$ is of degree $2g-2$ and its top degree terms store the intersection numbers $\int_{\overline{\modm}_{g,n}}\hspace{-2mm}\Theta_{g,n}\prod_{i=1}^n\psi_i^{m_i}$ involving only $\psi_i$ classes with $\Theta_{g,n}$.  

The following theorem gives recursion relations satisfied by the polynomials $V^{\Theta}_{g,n}(L_1,...,L_n)$ hence also by $2^{g-1+n}\widehat{V}^{WP}_{g,n}(L_1,...,L_n)$.    Introduce the kernel
\begin{equation}  \label{kerH} 
H(x,y)=\frac{1}{4\pi}\left(\frac{1}{\cosh\frac{x-y}{4}}-\frac{1}{\cosh\frac{x+y}{4}}\right)
\end{equation}
and the associated kernels
\begin{equation}  \label{kerDR} 
D(x,y,z)=H(y+z,x),\quad R(x,y,z)=\frac12 H(z,x+y)+\frac12 H(z,x-y).
\end{equation}
Let $L_A=\{L_i\mid i\in A\}$ for any set of positive integers $A$, and write any symmetric polynomial of the $|A|$ variables by $P(L_A)$ or allow more variables via, say $P(k,L_A)$.
\begin{theorem}  \label{main}
$V^{\Theta}_{g,n}$ is uniquely determined by $V^{\Theta}_{1,1}(L_1)=\frac18$ and the recursion
\begin{align}   \label{volrec}
L_1V^{\Theta}_{g,n}(L_1,L_K)=&\frac12\int_0^\infty\int_0^\infty xyD(L_1,x,y)P_{g,n+1}(x,y,L_K)dxdy\\
&\quad+
\sum_{j=2}^n\int_0^\infty xR(L_1,L_j,x)V^{\Theta}_{g,n-1}(x,L_{K\backslash\{j\}})dx
\nonumber
\end{align}
where $K=\{2,...,n\}$ and
\[\qquad\quad P_{g,n+1}(x,y,L_K)=V^{\Theta}_{g-1,n+1}(x,y,L_K)+\hspace{-3mm}\mathop{\sum_{g_1+g_2=g}}_{I \sqcup J = K}\hspace{-2mm}V^{\Theta}_{g_1,|I|+1}(x,L_I)V^{\Theta}_{g_2,|J|+1}(y,L_J).
\]
\end{theorem}
Theorems~\ref{volequal} and \ref{main} allow only Neveu-Schwarz boundary behaviour.
In \cite{ANoSup,NorSup}, we consider the more general situation of Neveu-Schwarz boundary together with Ramond punctures and prove a recursion between the corresponding  volumes.  The recursion is essentially the same as \eqref{volrec} although the right hand side involves extra, unstable, terms.

The recursion relations \eqref{volrec} are equivalent to recursion relations between intersections numbers over $\overline{\modm}_{g,n}$ involving the classes $\Theta_{g,n}$ and the tautological classes $\kappa_1,\psi_i$.  Furthermore, the recursion relations restrict to the top degree terms of $V^{\Theta}_{g,n}$ producing recursion relations between the numbers $\int_{\overline{\modm}_{g,n}}\hspace{-2mm}\Theta_{g,n}\prod_{i=1}^n\psi_i^{m_i}$.

Theorems~\ref{volequal} and \ref{main} combine to produce a recursion between the volumes of moduli spaces of super hyperbolic surfaces which coincides with a recursion by Stanford and Witten in \cite{SWiJTG}.   Stanford and Witten worked over the moduli space of smooth super hyperbolic surfaces, avoiding the need for a compactification and intersection theory, using techniques analogous to those of Mirzakhani applied to the super setting.  There are still some rigorous steps missing from the proof in \cite{SWiJTG}.  Nevertheless, Theorem~\ref{main} shows that the recursion between volumes of moduli spaces of super hyperbolic surfaces is rigorous.

Theorem~\ref{main} enables one to calculate $V^\Theta_{g,n}$ for $n>0$ whereas the definition \eqref{voltheta} makes sense also for $n=0$ and $g>1$.  The $n=0$ case can be calculated from the $n=1$ polynomial as follows.  For $g>1$,
\[V^\Theta_{g,0}=\frac{1}{2g-2}V^\Theta_{g,1}(2\pi i).\]
Note that the polynomial $V^\Theta_{g,n}(L_1,...,L_n)$ allows any complex argument, although to make sense of them as volumes, we require $L_i\geq 0$.  The formula for $V^\Theta_{g,0}$ is a special case of the following more general relation which is proven in \ref{cone}
\begin{equation}  \label{dilatonvol}
V^\Theta_{g,n+1}(2\pi i, L_1,...,L_n)=(2g-2+n)V^\Theta_{g,n}(L_1,...,L_n).
\end{equation}

The recursion \eqref{volrec} resembles the recursion between volumes of moduli spaces of hyperbolic surfaces---see \eqref{mirzvolrec}---by Mirzakhani \cite{MirSim}.   Moreover, Stanford and Witten \cite{SWiJTG} adapted Mirzakhani's proof to produce their proof of \eqref{volrec}.  Mirzakhani used the recursion between volumes to give a new proof that a generating function for intersection numbers of $\psi$ classes on $\overline{\modm}_{g,n}$ is annihilated by a collection of Virasoro operators. Together with the initial conditions, this is equivalent to the following famous theorem conjectured by Witten and proven by Kontsevich.  
\begin{theorem}[Kontsevich-Witten 1992, \cite{KonInt,WitTwo}]   \label{KW}
\[ Z^{\text{KW}}(\hbar,t_0,t_1,...)=\exp\sum_{g,n,\vec{k}}\frac{\hbar^{g-1}}{n!}\int_{\overline{\modm}_{g,n}} \prod_{i=1}^n\psi_i^{k_i}t_{k_i}
\] 
is a tau function of the KdV hierarchy.
\end{theorem}
Similar to Mirzakhani's proof of Theorem~\ref{KW}, the recursion \eqref{volrec} can be used to produce another set of Virasoro operators that annihilate a generating function for intersection numbers of $\psi$ classes and the classes $\Theta_{g,n}$.  This, together with its converse, is summarised in the following theorem.
Assemble the intersection numbers involving $\Theta_{g,n}$ and $\psi$ classes in the following generating function:
\begin{equation}  \label{ztheta}
Z^\Theta(\hbar,t_0,t_1,...)=\exp\sum_{g,n,\vec{k}}\frac{\hbar^{g-1}}{n!}\int_{\overline{\modm}_{g,n}}
\Theta_{g,n}\cdot\prod_{i=1}^n\psi_i^{k_i}t_{k_i}.
\end{equation} 
\begin{theorem}   \label{thetatau}
The recursion \eqref{volrec} is equivalent to Virasoro constraints satisfied by $Z^{\Theta}(\hbar,t_0,t_1,...)$.   These Virasoro constraints, given explicitly by \eqref{bgwvir}, are a consequence of the equality:
\begin{equation}  \label{taufn}
Z^{\Theta}(\hbar,t_0,t_1,...)=Z^{\text{BGW}}(\hbar,t_0,t_1,...)
\end{equation} 
where $Z^{\text{BGW}}$ 
is the Br\'ezin-Gross-Witten tau function of the KdV hierarchy.
\end{theorem}
The Virasoro constraints in Theorem~\ref{thetatau} produce recursion relations between the numbers  $\int_{\overline{\modm}_{g,n}}\hspace{-2mm}\Theta_{g,n}\prod_{i=1}^n\psi_i^{m_i}$ and the proof of the theorem uses the fact that the intersection numbers $ \int_{\overline{\modm}_{g,n}}\hspace{-2mm}\Theta_{g,n}\prod_{i=1}^n\psi_i^{m_i}\kappa_1^{\ell}$ are uniquely determined by the intersection numbers involving only the $\psi$ classes.
The Br\'ezin-Gross-Witten tau function of the KdV hierarchy which comes from a $U(n)$ matrix model \cite{BGrExt,GWiPos} is uniquely determined by the initial condition
\[\partial^2_{t_0}\log Z^{\text{BGW}}|_{\{t_{k>0}=0\}}=\frac{1}{8(1-t_0)^2}.\]  
This initial condition is also satisfied by $Z^{\Theta}(\hbar,t_0,t_1,...)$ due to $V^{\Theta}_{1,1}(L_1)=\frac18$.  The equality \eqref{taufn} was conjectured in \cite{NorNew} and proven in \cite{CGGRel}.  
The function $Z^{\Theta}(\hbar,t_0,t_1,...)$ is a specialisation of a more general tau function of the KdV hierarchy involving all of the classes $\kappa_j$, $j=1,2,...$ which is analogous to the higher Weil-Petersson volumes. This appears as Theorem~\ref{higherWPtrans} in Section~\ref{sec:kdv}.  

Eynard and Orantin \cite{EOrWei} proved that Mirzakhani's volume recursion, given by \eqref{mirzvolrec} in Section~\ref{sec:mirz}, can be neatly expressed in terms of topological recursion, defined in Section~\ref{sec:TR}, applied to the spectral curve
\[ x=\frac12z^2,\quad y=\frac{\sin(2\pi z)}{2\pi}.
\]
The following theorem describes a similar spectral curve on which topological recursion is equivalent to the recursion \eqref{volrec} in Theorem~\ref{main}.  Essentially the spectral curve efficiently encodes the kernels $D(x,y,z)$ and $R(x,y,z)$ defined in \eqref{kerDR}.  Let 
\[\cl\{V^{\Theta}_{g,n}(L_1,...,L_n)\}=\int_0^\infty...\int_0^\infty V^{\Theta}_{g,n}(L_1,...,L_n)\prod_{i=1}^n \exp(-z_iL_i)dL_i\] denote the Laplace transform.
\begin{theorem}   \label{spectral}
Topological recursion applied to the spectral curve
\[ x=\frac12z^2,\quad y=\frac{\cos(2\pi z)}{z}\]
produces correlators
\[\omega_{g,n}=\frac{\partial}{\partial z_1}...\frac{\partial}{\partial z_n}\cl\{V^{\Theta}_{g,n}(L_1,...,L_n)\}dz_1...dz_n.\]
\end{theorem}
The proof of Theorem~\ref{spectral} uses the algebro-geometric definition $V^{\Theta}_{g,n}(L_1,...,L_n)$ in \eqref{voltheta}   together with deep relations between topological recursion and Givental type factorisations of partition functions.   A more direct, but not yet rigorous, proof due to Stanford and Witten \cite{SWiJTG} uses the differential geometric definition \eqref{supvol} of $\widehat{V}^{WP}_{g,n}(L_1,...,L_n)$.  They produced a matrix model related to super JT gravity which gives rise to the spectral curve in Theorem~\ref{spectral}, and loop equations which coincide with topological recursion.

Theorem~\ref{main} is a consequence of Theorems~\ref{thetatau} and \ref{spectral} which essentially follows a converse to Mirzakhani's proof of Theorem~\ref{KW}.  The converse argument uses an elegant relationship between translations of $Z^{\Theta}(\hbar,t_0,t_1,...)$ and the introduction of $\kappa$ classes to the integrands, analogous to the result of Manin and Zograf \cite{MZoInv} for the Kontsevich-Witten tau function $Z^{\text{KW}}(\hbar,t_0,t_1,...)$.  It is achieved via topological recursion applied to the spectral curve given in Theorem~\ref{spectral}.


\vspace{.5cm}

{\em Outline:} In Section~\ref{sec:theta} we define the classes $\Theta_{g,n}$ required for the definition of the polynomials $V^\Theta_{g,n}$.    In Section~\ref{sec:hypspin} spin structures on hyperbolic surfaces are studied from a gauge theoretic viewpoint which brings in Higgs bundles techniques to achieve a number of goals:  it relates the sheaf cohomologies arising from a flat structure and a holomorphic structure on a bundle; it relates hyperbolic metrics on a non-compact Riemann surface $\Sigma=\overline{\Sigma}-D$ to bundles on the compact pair $(\overline{\Sigma},D)$;  it naturally produces bundles on the orbifold curve $(\cc,D)\to(\overline{\Sigma},D)$ which makes a connection with the construction of $\Theta_{g,n}$ in Section~\ref{sec:theta}.  The proof in Section~\ref{sec:hypspin} of the isomorphism \eqref{canis}  does not directly follow from Simpson's theorem.  Instead, we embed the rank one bundle on the right hand side of \eqref{canis} inside a rank two bundle to which Simpson's theorem is applied.  The main outcome of Section~\ref{sec:hypspin} is the proof that the bundle $E_{g,n}\to\modm_{g,n,\vec{o}}^{\text{spin}}$ naturally extends to $\overline{\modm}_{g,n,\vec{o}}^{\text{spin}}$, and the proof that the natural Euler form on $E_{g,n}$ also extends, which is given in Theorem~\ref{eulerfc}.  Together these lead to the proof of Theorem~\ref{volequal}. In  Section~\ref{sec:mirz} we recall details of Mirzakhani's techniques and the recursion of Stanford and Witten between volumes of moduli spaces of super hyperbolic surfaces analogous to Mirzakhani's recursions between volumes of moduli spaces of hyperbolic surfaces.  Section~\ref{sec:kdv} follows Mirzakhani's methods to show that the top degree terms in the recursion \eqref{volrec} are equivalent to a collection of Virasoro operators annihilating $Z^\Theta$, which is necessary for the proof of Theorem~\ref{main}.   Section~\ref{sec:TR} contains the final details of the proof of Theorem~\ref{main} as a consequence of Theorems~\ref{thetatau} and \ref{spectral}.  The main technique used in the proof of Theorem~\ref{main}, via Theorem~\ref{spectral}, is topological recursion which conveniently encodes the Givental factorisation \cite{GivGro} of partition functions of CohFTs, defined in \eqref{partfun}, into a complex curve equipped with extra structure, known as a spectral curve.  The appearance of topological recursion is extremely natural in this case, since it turns out to be directly related to the Laplace transform of the recursion \eqref{volrec}, which is stated in Theorem~\ref{thetaspec}. \\
\\

{\em Acknowledgements.}   I would like to express my deep gratitude to Edward Witten for his numerous patient 
explanations of many aspects of this paper.  I would also like to thank Quentin Guignard, Ran Tessler and Anton Zeitlin for useful conversations, and the Max Planck Insitute for Mathematics, Bonn, and LMU, Munich where part of this work was carried out.  This work was partially supported under the  Australian
Research Council {\sl Discovery Projects} funding scheme project number DP180103891.

\section{The classes \texorpdfstring{$\Theta_{g,n}\in H^*(\overline{\modm}_{g,n},\bq)$}.}  \label{sec:theta}

Let $\overline{\modm}_{g,n}$ be the moduli space of genus $g$ stable curves---curves with only nodal singularities and finite automorphism group---with $n$ labeled points disjoint from nodes.  
In this section we define the cohomology classes $\Theta_{g,n}\in H^*(\overline{\modm}_{g,n},\bq)$ via a construction over the moduli space of stable twisted spin curves $\overline{\modm}_{g,n}^{\text{spin}}$.  The class $\Theta_{g,n}$ will be defined as a multiple of the push-forward of the top Chern class of a natural bundle, given in Definition~\ref{obsbun} below, over a component of $\overline{\modm}_{g,n}^{\text{spin}}$.
The volume polynomials $V^{\Theta}_{g,n}(L_1,...,L_n)$ defined in \eqref{voltheta} and the partition function $Z^{\Theta}(\hbar,t_0,t_1,...)$ defined in \ref{taufn}, Theorem~\ref{main} will be shown to depend only on the characterisation \eqref{glue}, \eqref{forget} of $\Theta_{g,n}$  and $\int_{\overline{\modm}_{1,1}}\Theta_{1,1}=\frac18$.  In other words, $V^{\Theta}_{g,n}(L_1,...,L_n)$ and $Z^{\Theta}(\hbar,t_0,t_1,...)$ can be characterised purely in terms of $\overline{\modm}_{g,n}$ without reference to $\overline{\modm}_{g,n}^{\text{spin}}$.   

The following definition which uses twisted, or orbifold, curves is taken from \cite{AJaMod}.
\begin{definition}  \label{def:modspin}
The moduli space of spin curves is defined by
\[\modm_{g,n}^{\text{spin}}=\{(\cc,\theta,p_1,...,p_n,\phi)\mid \phi:\theta^2\stackrel{\cong}{\longrightarrow}\omega_{\cc}^{\text{log}}\}
\]
where $\theta$ is a line bundle over a twisted curve $\cc$ with group $\bz_2$, each labeled point $p_i$ has isotropy group $\bz_2$ and all other points have trivial isotropy group.
\end{definition}
There is a natural compactification of $\modm_{g,n}^{\text{spin}}$ by twisted, stable, spin curves.
\begin{definition}  \label{def:modspincomp}
The moduli space of stable spin curves is defined by
\[\overline{\modm}_{g,n}^{\text{spin}}=\{(\cc,\theta,p_1,...,p_n,\phi)\mid \phi:\theta^2\stackrel{\cong}{\longrightarrow}\omega_{\cc}^{\text{log}}\}
\]
where $\theta$ is a  line bundle over a stable, twisted curve $\cc$ with group $\bz_2$, each nodal point and labeled point $p_i$ has isotropy group $\bz_2$, and all other points have trivial isotropy group.
\end{definition}
A stable twisted curve $\cc$ is equipped with a map which forgets the orbifold structure $\rho:\cc\to C$ where $C$ is a stable curve known as the coarse curve of $\cc$.    The map $\rho$ induces a map
\[ p:\overline{\modm}_{g,n}^{\text{spin}}\to\overline{\modm}_{g,n}.
\]
In fact, the map $p$ is a composition of $\rho$ with the  $2^{2g}$ to 1 map to the moduli space of twisted curves $\overline{\modm}_{g,n}^{\text{spin}}\to\overline{\modm}^{(2)}_{g,n}$, where the latter moduli space is defined as above without the spin structure, and consists of twisted curves $\{(\cc,p_1,...,p_n)\}$.  There are $2^{2g+n-1}$ choices of $(\theta,\phi)$ for each twisted curve $\cc$ in $\overline{\modm}^{(2)}_{g,n}$, and after fixing representation data at each $p_i$, described below, there are $2^{2g}$ different spin structures.  See \cite{FJRQua} for further details.  

The bundles $\omega_{\cc}^{\text{log}}$ and $\theta$ are line bundles over $\cc$, i.e. locally equivariant bundles over the local charts such that at each nodal point there is an equivariant isomorphism of fibres.  On each fibre over an orbifold point $p$ the equivariant isomorphism associates a representation of $\bz_2$ which is either trivial or the unique non-trivial representation.  The equivariant isomorphism at nodes guarantees that the representations agree on each local irreducible component at the node, known as the balanced condition.  The representation associated to $\omega_{\cc}^{\text{log}}$ at $p_i$ and nodal points is trivial since locally $dz/z\stackrel{z\mapsto-z}{\longrightarrow}dz/z$.  The representations associated to $\theta$ at each $p_i$ define a vector $\vec{\epsilon}\in\{0,1\}^n$, where $0$, respectively $1$, in $\{0,1\}$ corresponds to the unique non-trival, respectively trivial, representation $\bz_2\to\bz_2$.  The assignment of $0$ to the non-trivial representation looks more natural when viewed cohomologically via an associated quadratic form defined in \ref{qf}.   As described in the introduction, a labeled point $p_i$ is known as a Neveu-Schwarz point when the associated representation is non-trivial, and a Ramond point otherwise.  The representations at labeled points define a decomposition into connected components 
\begin{equation}  \label{modspin}
\overline{\modm}_{g,n}^{\text{spin}}=\bigsqcup_{\vec{\epsilon}\in\{0,1\}^n}\overline{\modm}_{g,n,\vec{\epsilon}}^{\text{spin}}
\end{equation}
and an analogous decomposition $\modm_{g,n}^{\text{spin}}=\bigsqcup_{\vec{\epsilon}\in\{0,1\}^n}\modm_{g,n,\vec{\epsilon}}^{\text{spin}}$ of the moduli space of smooth curves.  We will see a decomposition of the character variety analogous  to \eqref{modspin} in Definition~\ref{modspace}.  

The construction of the classes $\Theta_{g,n}$ use only the component with non-trivial representations at labeled points, or Neveu-Schwarz points, denoted 
\[\overline{\modm}_{g,n,\vec{o}}^{\text{spin}}\subset\overline{\modm}_{g,n}^{\text{spin}},\qquad\vec{o}=\{0,...,0\}\in\{0,1\}^n.\]  
Nevertheless, other components arise in lower strata of the compactification since at nodal points, both types---trivial and non-trivial representations can occur.

We have  $\deg\omega_{\cc}^{\text{log}}=2g-2+n$ and $\deg\theta=g-1+\frac12n$ which may be a half-integer since the orbifold points allows for such a possibility.   In particular $\deg\theta^{\vee}=1-g-\frac12n<0$, and for any irreducible component $\deg\theta^{\vee}|_{\cc'}<0$ since $\cc'$ is stable so its log canonical bundle has negative degree.  Thus $H^0(\cc,\theta^{\vee})=0$ so $H^1(\cc,\theta^{\vee})$ has constant dimension and defines a vector bundle $\widehat{E}_{g,n}\to\overline{\modm}_{g,n,\vec{o}}^{\text{spin}}$.  By the Riemann-Roch theorem $H^1(\cc,\theta^{\vee})\cong\bc^{2g-2+n}$.  More formally,  denote by $\ce$ the universal spin structure defined over the universal curve  $\cu_{g,n}^{\text{spin}}\stackrel{\pi}{\longrightarrow}\overline{\modm}_{g,n,\vec{o}}^{\text{spin}}$.  
\begin{definition}  \label{obsbun}
Define the bundle $\widehat{E}_{g,n}:=-R\pi_*\ce^\vee\hspace{-1mm}\to\overline{\modm}_{g,n,\vec{o}}^{\text{spin}}$ with fibre $H^1(\cc,\theta^{\vee})$.   
\end{definition} 

\begin{definition}  \label{theta}
$\Theta_{g,n}:=(-1)^n2^{g-1+n}p_*c_{2g-2+n}(\widehat{E}_{g,n})\in H^{4g-4+2n}(\overline{\modm}_{g,n},\bq).$
\end{definition} 

Define
\begin{equation}   \label{psiclass}
\psi_i=c_1(L_i)\in H^{2}(\overline{\mathcal{M}}_{g,n},\mathbb{Q})
\end{equation} 
to be the first Chern class of the line bundle $L_i\to\overline{\mathcal{M}}_{g,n}$ with fibre $T_{p_i}^*C$ above $[(C,p_1,...,p_n)]$.  Using the forgetful map $\overline{\modm}_{g,n+1}\stackrel{\pi}{\longrightarrow}\overline{\modm}_{g,n}$, define 
\begin{equation}   \label{kappaclass}
\kappa_m:=\pi_*\psi_{n+1}^{m+1}\in H^{2m}(\overline{\mathcal{M}}_{g,n},\mathbb{Q}).
\end{equation} 

It is proven in \cite{NorNew} that $\Theta_{g,n}$ satisfies the pull-back properties \eqref{glue} and \eqref{forget} and $\int_{\overline{\modm}_{1,1}}\Theta_{1,1}=\frac18$.  These properties uniquely determine the intersection numbers of $\Theta_{g,n}$ with $\psi$ classes and $\kappa$ classes as shown in the following proposition.  A consequence is that the polynomial $V^{\Theta}_{g,n}(L_1,...,L_n)$ and the partition function  $Z^{\Theta}(\hbar,t_0,t_1,...)$, can be characterised purely in terms of $\overline{\modm}_{g,n}$ without reference to $\overline{\modm}_{g,n}^{\text{spin}}$. 
\begin{proposition}[\cite{NorNew}]  \label{th:unique}
For any collection $\Theta_{g,n}\in H^{4g-4+2n}(\overline{\modm}_{g,n})$ satisfying the pull-back properties \eqref{glue} and \eqref{forget}, the intersection numbers
\begin{equation}  \label{corr}
\int_{\overline{\modm}_{g,n}}\hspace{-2mm}\Theta_{g,n}\prod_{i=1}^n\psi_i^{m_i}\prod_{j=1}^N\kappa_{\ell_j}
\end{equation}
are uniquely determined from the initial condition $\Theta_{1,1}=\lambda\psi_1$ for $\lambda\in\bc$. 
\end{proposition}
\begin{proof}[Sketch of proof]
For $n>0$, since $\psi_n\psi_k=\psi_n\pi^*\psi_k$ for $k<n$ and\\ $\Theta_{g,n}=\psi_n\cdot\pi^*\Theta_{g,n-1}$ then
\[\Theta_{g,n}\psi_k=\Theta_{g,n}\pi^*\psi_k,\quad k<n.
\]
When there are no $\kappa$ classes. 
\[
\int_{\overline{\modm}_{g,n}}\hspace{-2mm}\Theta_{g,n}\prod_{i=1}^n\psi_i^{m_i}=
\int_{\overline{\modm}_{g,n}}\hspace{-2mm}\pi^*\Big(\Theta_{g,n-1}\prod_{i=1}^{n-1}\psi_i^{m_i}\Big)\psi_n^{m_n+1}
=\int_{\overline{\modm}_{g,n-1}}\hspace{-5mm}\Theta_{g,n-1}\prod_{i=1}^{n-1}\psi_i^{m_i}\kappa_{m_n}
\]
so we have reduced an intersection number over $\overline{\modm}_{g,n}$ to an intersection number over $\overline{\modm}_{g,n-1}$.
In the presence of $\kappa$ classes, replace $\kappa_{\ell_j}$ by $\kappa_{\ell_j}=\pi^*\kappa_{\ell_j}+\psi_n^{\ell_j}$ and repeat the push-forward as above on all summands.  By induction, we see that
\[\int_{\overline{\modm}_{g,n}}\hspace{-2mm}\Theta_{g,n}\prod_{i=1}^n\psi_i^{m_i}\prod_{j=1}^N\kappa_{\ell_j}=\int_{\overline{\modm}_{g}}\Theta_{g}\cdot p(\kappa_1,\kappa_2,...,\kappa_{3g-3})\]
i.e. the intersection number \eqref{corr} reduces to an intersection number over $\overline{\modm}_{g}$ of $\Theta_g$ times a polynomial in the $\kappa$ classes.  Since $\deg\Theta_g=2g-2$ we may assume the polynomial $p$ consists only of terms of homogeneous degree $g-1$.   Any homogeneous degree $g-1$ monomial in the $\kappa$ classes is equal in cohomology to the sum of boundary terms, \cite{LooTau,PPZRel}.  By \eqref{glue} the pull-back of $\Theta_g$ to these boundary terms is $\Theta_{g',n'}$ for $g'<g$ so we have expressed \eqref{corr} as a sum of integrals of $\theta_{g',n'}$ against $\psi$ and $\kappa$ classes.  By induction, one can reduce to the integral $\int_{\overline{\modm}_{1,1}}\Theta_{1,1}=\frac{\lambda}{24}$ and the proposition is proven.
\end{proof}

\subsubsection{Cohomological field theories}   \label{sec:cohft}
The classes $\Theta_{g,n}$ pair with any cohomological field theory, such as Gromov-Witten invariants, to give rise to new invariants. Recall that a {\em cohomological field theory} is a pair $(V,\eta)$ composed of a finite-dimensional complex vector space $V$ equipped with a nondegenerate, bilinear, symmetric form $\eta$ which we call a metric (although it is not positive-definite) and for $n\geq 0$ a sequence of $S_n$-equivariant maps. 
\[ \Omega_{g,n}:V^{\otimes n}\to H^*(\overline{\modm}_{g,n},\bc)\]
that satisfy pull-back properties with respect to the gluing maps defined in the introduction, that generalise \eqref{glue}.
\begin{align}
\phi_{\text{irr}}^*\Omega_{g,n}(v_1\otimes...\otimes v_n)&=\Omega_{g-1,n+2}(v_1\otimes...\otimes v_n\otimes\Delta) 
  \label{glue1}
\\
\phi_{h,I}^*\Omega_{g,n}(v_1\otimes...\otimes v_n)&=\Omega_{h,|I|+1}\otimes \Omega_{g-h,|J|+1}\big(\bigotimes_{i\in I}v_i\otimes\Delta\otimes\bigotimes_{j\in J}v_j\big) \label{glue2}
\end{align}
where $\Delta\in V\otimes V$ is dual to the metric $\eta\in V^*\otimes V^*$. 

There exists a vector $\un\in V$ satisfying
\begin{equation}  \label{nondeg}
\Omega_{0,3}(v_1\otimes v_2\otimes \un)=\eta(v_1,v_2)
\end{equation}
which is essentially a non-degeneracy condition.  A CohFT defines a product $\cdot$ on $V$ using the non-degeneracy of $\eta$ by
\begin{equation}  \label{prod} 
\eta(v_1\Cdot v_2,v_3)=\Omega_{0,3}(v_1,v_2,v_3).
\end{equation}
and $\un$ is a unit for the product.  Such CohFTs were classified by Teleman \cite{TelStr}.  We will also consider sequences of $S_n$-equivariant maps $\Omega_{g,n}$ that satisfy \eqref{glue1} and \eqref{glue2}, but do not satisfy \eqref{nondeg} which we call a CohFT without unit.  

The CohFT is said to have {\em flat unit} if 
\begin{equation}   \label{cohforget}
\Omega_{g,n+1}(\un\otimes v_1\otimes...\otimes v_n)=\pi^*\Omega_{g,n}(v_1\otimes...\otimes v_n)
\end{equation}
for $2g-2+n>0$.  A CohFT without unit may still possess a distinguished element $\un$ which, in place of \eqref{cohforget}, may satisfy the following:
\begin{equation}   \label{cohforget1}
\Omega_{g,n+1}(\un\otimes v_1\otimes...\otimes v_n)=\psi_{n+1}\pi^*\Omega_{g,n}(v_1\otimes...\otimes v_n).
\end{equation}

The product \eqref{prod} is {\em semisimple} if it is diagonal $V\cong\bc\oplus\bc\oplus...\oplus\bc$, i.e. there is a canonical basis $\{ u_1,...,u_N\}\subset V$ such that $u_i\Cdot u_j=\delta_{ij}u_i$.  The metric is then necessarily diagonal with respect to the same basis, $\eta(u_i,u_j)=\delta_{ij}\eta_i$ for some $\eta_i\in\bc \setminus \{0\}$, $i=1,...,N$.

For a one-dimensional CohFT, i.e. $\dim V=1$, identify $\Omega_{g,n}$ with the image $\Omega_{g,n}(\un^{\otimes n})$, so we write $\Omega_{g,n}\in H^*(\overline{\modm}_{g,n},\bc)$.  An example of a one-dimensional CohFT is
\[\Omega_{g,n}=\exp(2\pi^2\kappa_1).
\]
The classes $\Theta_{g,n}$ define a one-dimensional CohFT without unit.

The partition function of a CohFT $\Omega=\{\Omega_{g,n}\}$ is defined by:
\begin{equation}   \label{partfun}
Z_{\Omega}(\hbar,\{t^{\alpha}_k\})=\exp\sum_{g,n,\vec{k}}\frac{\hbar^{g-1}}{n!}\int_{\overline{\modm}_{g,n}}\Omega_{g,n}(e_{\alpha_1}\otimes...\otimes e_{\alpha_n})\cdot\prod_{j=1}^n\psi_j^{k_j}\prod t^{\alpha_j}_{k_j}
\end{equation}
where $\{e_1,...,e_N\}$ is a basis of $V$, $\alpha_i\in\{1,...,N\}$ and $k_j\in\bn$.

For any CohFT $\Omega$ on $(V,\eta)$ define $\Omega^\Theta=\{\Omega^\Theta_{g,n}\}$ to be the CohFT without unit $\Omega^\Theta_{g,n}:V^{\otimes n}\to H^*(\overline{\modm}_{g,n},\bc)$ given by $\Omega^\Theta_{g,n}(v_1\otimes...\otimes v_n)=\Theta_{g,n}\cdot\Omega_{g,n}(v_1\otimes...\otimes v_n)$.

Apply this to the example above to get
$\Omega^\Theta_{g,n}=\Theta_{g,n}\cdot\exp(2\pi^2\kappa_1)$ which has a partition function that stores all of the volume polynomials
\[Z_{\Omega^\Theta}(\hbar,\{t_k\})=\exp\sum_{g,n}\frac{\hbar^{g-1}}{n!}V_{g,n}^\Theta(L_1,...,L_n)|_{\{L_i^{2k}=2^kk!t_k\}}.
\]
Note that the substitution $L_i^{2k}=2^kk!t_k$ requires one to take the highest power of $L_i$ in each monomial, and importantly, to substitute $L_i^0=t_0$ when $L_i$ is missing from a monomial of $V_{g,n}^\Theta(L_1,...,L_n)$.  See \ref{sec:TRpart} for further details.

\section{Hyperbolic geometry and spin structures}  \label{sec:hypspin}
In this section we construct the bundle $E_{g,n}\to\modm_{g,n,\vec{o}}^{\text{spin}}$ over the moduli space of smooth spin curves via hyperbolic geometry and prove that it coincides with the restriction of the bundle $\widehat{E}_{g,n}\to\overline{\modm}_{g,n,\vec{o}}^{\text{spin}}$ defined in Definition~\ref{obsbun}.  The importance of the two constructions via hyperbolic geometry and via algebraic geometry is that they give rise to the definitions of $\widehat{V}^{WP}_{g,n}(L_1,...,L_n)$ in \eqref{supvol}, respectively  $V^{\Theta}_{g,n}(L_1,...,L_n)$ in \eqref{voltheta}. 

We begin with a description of spin hyperbolic structures on a topological surface $\Sigma$ via Fuchsian representations of $\pi_1\Sigma$ into $SL(2,\br)$.  On a spin hyperbolic surface $\Sigma$ the representation produces the associated flat $SL(2,\br)$-bundle $T_\Sigma^\frac12$ which is used to construct the bundle $E_{g,n}$ from the cohomology of the locally constant sheaf of sections of $T_\Sigma^\frac12$.  Using Higgs bundles defined over a smooth curve with labeled points $(\overline{\Sigma},p_1,...,p_n)$ we prove a canonical isomorphism between fibres of $E_{g,n}$ and fibres of $\widehat{E}_{g,n}$ over smooth $\Sigma=\overline{\Sigma}-\{p_1,...,p_n\}$.   Higgs bundles appear naturally here due to a proof by Hitchin \cite{HitSel} of uniformisation---a Riemann surface $\Sigma$ possesses a unique representative, in its conformal class, by a complete finite area hyperbolic surface---which requires parabolic Higgs bundles on $(\overline{\Sigma},D)$ for $D=\sum p_i$ when $\Sigma$ is non-compact.

\subsection{Fuchsian representations}
A hyperbolic metric on an oriented topological surface is defined via a Fuchsian representation
\[\overline{\rho}:\pi_1\Sigma\to PSL(2,\br).\]
The natural constant curvature $-1$ metric $ds^2$ defined on hyperbolic space
\[\bh=\{z\in\bc\mid\text{Im }z>0\},\qquad ds^2=\frac{|dz|^2}{\text{Im}(z)^2}
\] 
is $PSL(2,\br)$ invariant and induces a metric on $\Sigma$ via the quotient $\Sigma\cong \bh/\overline{\rho}(\pi_1\Sigma)$.
 
A {\em boundary class} $\gamma\subset\Sigma$ represents a homotopy class of simple, closed, separating curves such that one component of $\Sigma-\gamma$ is an annulus.  It determines a class $[\gamma]\in H_1(\Sigma,\bz)$ which we also call a boundary class.  A boundary class represents a conjugacy class in $\pi_1\Sigma$ which maps under $\overline{\rho}$ to a conjugacy class in $PSL(2,\br)$.  A conjugacy class in $PSL(2,\br)$ is {\em parabolic} if any representative $A\in PSL(2,\br)$ satisfies $|\tr(A)|=2$ and {\em hyperbolic} if any representative $A\in PSL(2,\br)$ satisfies $|\tr(A)|>2$. 
Boundary classes with parabolic, respectively hyperbolic, images under $\rho:\pi_1\Sigma\to PSL(2,\br)$ correspond to cusps, respectively geodesic boundary components.   In the latter case, the hyperbolic surface is the interior of a compact hyperbolic surface with geodesic boundary component, and we sometimes abuse notation and also denote this compact surface with boundary by $\Sigma$.

We used $\overline{\rho}$ above because we will instead consider representations 
\[\rho:\pi_1\Sigma\to SL(2,\br)\]
such that the composition $\overline{\rho}$ of $\rho$ with the map $SL(2,\br)\to PSL(2,\br)$ is Fuchsian.    
Any closed curve $\gamma\subset \Sigma$ corresponds to a conjugacy class in $\pi_1\Sigma$ and we write $[\gamma]\in\pi_1\Sigma$ for any representative of the conjugacy class associated to $\gamma$.
A Fuchsian representation satisfies the property that $|\tr\rho([\gamma])|\geq 2$ for all simple closed curves $\gamma\subset \Sigma$ and it equals 2 only when $[\gamma]$ is a boundary class.  The geometric meaning of the Fuchsian property uses the fact that for any closed curve $\gamma\subset \Sigma$ there exists a unique closed geodesic $g_\gamma$ in its free homotopy class and $|\tr\rho([\gamma])|=2\cosh(\ell(g_\gamma)/2)$ determines its hyperbolic length $\ell(g_\gamma)$.  The Fuchsian property of $\overline{\rho}:\pi_1\Sigma\to PSL(2,\br)$ can be determined via its circle bundle over $\Sigma$ defined via the action of $PSL(2,\br)$ on the circle at infinity $S^1\cong\partial\bh$.  If the Euler class of this circle bundle is equal to $\pm(2g-2+n)$ then $\overline{\rho}$ is a Fuchsian representation, \cite{GolTop,HitSel}. 

\subsubsection{}  \label{spin}
A Riemannian metric, in particular the hyperbolic metric, on an orientable surface $\Sigma$ determines a principal $SO(2)$ bundle $P_{SO}(\Sigma)$ given by the orthonormal frame bundle of $\Sigma$.   A {\em spin structure} on a Riemannian surface $\Sigma$ is a principal $SO(2)$ bundle $P_{\text{Spin}}(\Sigma)\to \Sigma$ that is a double cover of the orthonormal frame bundle $P_{\text{Spin}}(\Sigma)\to P_{SO}(\Sigma)$ which restricts to a non-trivial double cover on each $SO(2)$ fibre.  Any spin structure is naturally identified with an element of $H^1(P_{SO}(\Sigma),\bz_2)=\text{Hom}(\pi_1(P_{SO}(\Sigma)),\bz_2)$.  The non-trivial double-cover condition on each $SO(2)$ fibre is captured by the exact sequence in cohomology
\[0\to H^1(\Sigma,\bz_2)\to H^1(P_{SO}(\Sigma),\bz_2)\stackrel{r}{\to} H^1(SO(2),\bz_2)\to   0
\]
by requiring that $r$ is non-zero, \cite{MilSpi}.  The rightmost arrow is defined by the vanishing second Stiefel-Whitney class which take values in $H^2(\Sigma,\bz_2)$ and guarantees the existence of a spin structure.  The exact sequence shows that the set of spin structures on $\Sigma$ is an $H^1(\Sigma,\bz_2)$ affine space.

\subsubsection{} The bundle of spinors $S_\Sigma\to\Sigma$ is the associated bundle 
\[S_\Sigma=P_{\text{Spin}}(\Sigma)\times_{SO(2)}\bc^2\]
where $SO(2)$ acts by the natural representation on $\bc^2$ (which is the unique irreducible representation of the complexified Clifford algebra $\text{Spin}(2)\subset Cl_2\otimes\bc=M(2,\bc)$).  The represention of $SO(2)$ decomposes into irreducible representations of weights $\chi=e^{i\alpha}$ and $\chi^{-1}=e^{-i\alpha}$ so the spinor bundle decomposes into complex line bundles $S_\Sigma=T_\Sigma^\frac12\oplus T_\Sigma^{-\frac12}$ where $T_\Sigma^\frac12=P_{\text{Spin}}(\Sigma)\times_{\text{Spin}(2)}\bc_\chi$.  Since the weight of the tangent bundle $T_\Sigma$ is $\chi^2$, 
\[T_\Sigma^\frac12\otimes T_\Sigma^\frac12=P_{\text{Spin}}(\Sigma)\times_{\text{Spin}(2)}\bc_{\chi^2}=P_{SO}(\Sigma)\times_{SO(2)}\bc_{\chi^2}=T_\Sigma\] 
is holomorphic  hence $T_\Sigma^\frac12$ and $T_\Sigma^{-\frac12}$ are holomorphic.

\subsubsection{} The orthonormal frame bundle $P_{SO}(\Sigma)$ and any spin structure of a hyperbolic surface $\Sigma$ arise naturally via representations of $\pi_1\Sigma$ as follows.  The group $PSL(2,\br)$ acts freely and transitively on $P_{SO}(\bh)$, the orthonormal frame bundle of $\bh$, hence the two are naturally identified:
\[P_{SO}(\bh)\cong PSL(2,\br)\to\bh.\]  
The double cover $SL(2,\br)\to PSL(2,\br)$ is a non-trivial double cover on each $SO(2)$ fibre since a path from $I$ to $-I$ in $SL(2,\br)$ lives above the fibre $SO(2)\subset PSL(2,\br)$.  Hence $SL(2,\br)\cong P_{\text{Spin}}(\bh)$ is the unique spin structure.
When $\Sigma=\bh/\overline{\rho}(\pi_1\Sigma)$ is hyperbolic, $PSL(2,\br)$ descends to the orthonormal frame bundle of $\Sigma$: 
\[P_{SO}(\Sigma)\cong  PSL(2,\br)/\overline{\rho}(\pi_1\Sigma)\to \Sigma.\]  
A representation $\rho:\pi_1\Sigma\to SL(2,\br)$ that lives above $\overline{\rho}$ produces a double cover 
\[SL(2,\br)/\rho(\pi_1\Sigma)\to P_{SO}(\Sigma)\]  
which is a non-trivial double cover on each $SO(2)$ fibre since it locally resembles $SL(2,\br)\to PSL(2,\br)$.  Hence $\rho$ defines a spin structure on $\Sigma$.

There is an action of $H^1(\Sigma,\bz_2)$ on representations $\rho$ living above a given representation $\overline{\rho}$ obtained by
multiplying any representation by the representation $\epsilon:\pi_1\Sigma\to\{\pm I\}$ associated to an element of $H^1(\Sigma,\bz_2)$.  Since the set of spin structure on $\Sigma$ is an $H^1(\Sigma,\bz_2)$ affine space, this shows that all spin structures on $\Sigma$ arise via representations $\rho:\pi_1\Sigma\to SL(2,\br)$ once we know that at least one lift $\rho$ of $\overline{\rho}$ exists.

For a given representation $\overline{\rho}:\pi_1\Sigma\to PSL(2,\br)$, the existence of a lift $\rho:\pi_1\Sigma\to SL(2,\br)$ is elementary in the case that $\Sigma$ is non-compact.  Choose a presentation  \[\pi_1\Sigma =\{a_1,a_2,....,a_g,b_1,...,b_g,c_1,...,c_{n}\mid \prod_{i=1}^g[a_i,b_i]\prod_{j=1}^nc_j=\un\}.\] 
Choose any lifts of $\overline{\rho}(a_i)$, $\overline{\rho}(b_i)$ and $\overline{\rho}(c_j)$ in $PSL(2,\br)$ to $\rho(a_i)$, $\rho(b_i)$ and $\rho(c_j)$ in $SL(2,\br)$, for $i=1,...,g$ and $j=1,...,n$.  Then $\prod_{i=1}^g[\rho(a_i),\rho(b_i)]\prod_{j=1}^n\rho(c_j)=\pm\un$ which is the fibre over $\un$.  Since $n>0$, by possibly replacing $\rho(c_n)\to-\rho(c_n)$  we get the existence of a single lift.  When $\Sigma$ is compact, cut it into two pieces $\Sigma=\Sigma_1\cup_\gamma\Sigma_2$ along a simple closed curve $\gamma$ containing the basepoint used to define $\pi_1\Sigma$, say a genus 1 piece and a genus $g-1$ piece ($\Sigma$ is hyperbolic so $g>1$).   Now $\overline{\rho}:\pi_1\Sigma\to PSL(2,\br)$ induces representations $\overline{\rho}_i:\pi_1\Sigma_i\to PSL(2,\br)$, for $i=1,2$.  As above choose lifts of $\rho_i$ of $\overline{\rho}_i$.  The lifts $\rho_1$ and $\rho_2$ necessarily agree on their respective boundary components because they come from $\overline{\rho}$ and both traces are negative by a homological argument given by Corollary~\ref{negtrace} in \ref{quadfrep}.  Hence we can glue to get a lift $\rho$.

\subsubsection{} The disk $D^2$ possesses a unique spin structure.  Its bundle of frames is trivial, i.e. $P_{SO}(D^2)\cong D^2\times S^1$, for any Riemannian metric on $D^2$.  Hence a spin structure over a disk is unique and given by the non-trivial double cover of $D^2\times S^1$ or equivalently the non-trivial element $\eta\in H^1(D^2\times S^1,\bz_2)\cong\bz_2$.  An annulus $\ba$, possesses two spin structures corresponding to the non-trivial (connected) and trivial (disconnected) double covers of $\ba\times S^1$.  One of these spin structures extends to the disk and one does not.  
\begin{definition}  \label{NS}
Given a spin structure over $\Sigma$, a boundary class $\gamma\subset\Sigma$ is said to be {\em Neveu-Schwarz} if the restriction of the spin structure to $\gamma$ is non-trivial, or equivalently if the spin structure extends to a disk glued along $\gamma$.  The boundary class $\gamma$ is {\em Ramond} if the restriction of the spin structure to $\gamma$ is trivial.
\end{definition}
On a surface $\Sigma=\overline{\Sigma}-\{p_1,...,p_n\}$, the boundary component at $p_i$ is Neveu-Schwarz exactly when the spin structure extends over the completion $\Sigma\cup\{p_i\}$ at $p_i$.  It is Ramond if the spin structure does not extend over the completion there.  

\subsubsection{} \label{qf}
A quadratic form $q$ on $H_1(\Sigma,\bz_2)$ is a map $q:H_1(\Sigma,\bz_2)\to\bz_2$ satisfying
\[ q(a+b)=q(a)+q(b)+(a,b)
\]
where $(a,b)$ is the mod 2 intersection form on $H_1(\Sigma,\bz_2)$.  Quadratic forms are called {\em Arf functions} in \cite{DFNMod,NatMod}.  The set of quadratic forms is clearly an $H^1(\Sigma,\bz_2)$ affine space.   A quadratic form naturally associated to any spin structure due to Johnson \cite{JohSpi} is defined as follows.  Represent $[C]\in H_1(\Sigma,\bz_2)$ by a finite sum of disjoint, embedded, oriented closed curves $C=\sum_{i=1}^n C_i$ and define a map 
\[\ell: H_1(\Sigma,\bz_2)\to H_1(P_{SO}(\Sigma),\bz_2)\] 
by $\ell([C])=n\sigma+\sum_{i=1}^n \tilde{C}_i$ where $\sigma$ is the image of the generator of $H_1(SO(2),\bz_2)$ in $H_1(P_{SO}(\Sigma),\bz_2)$ under the natural inclusion of the fibre, and $\tilde{C}_i$ is the lift of $C_i$ to $P_{SO}(\Sigma)$ using its tangential framing.  The map $\ell$ is well-defined on homology since it is invariant under isotopy, trivial on the boundary of a disk which lifts via its tangential framing to $\sigma$, and invariant under replacement of crossings by locally embedded curves.  Identify a given spin structure with an element $\eta\in H^1(P_{SO}(\Sigma),\bz_2)$ satisfying $\eta(\sigma)=1$, and define
\[q_\eta=\eta\circ\ell.
\]
It is routine to check that $q_\eta$ is a quadratic form, and that $\eta\mapsto q_\eta$ defines an isomorphism of $H^1(\Sigma,\bz_2)$ affine spaces between spin structures and quadratic forms.

Neveu-Schwarz and Ramond boundary classes of a spin structure defined in Definition~\ref{NS} can be stated efficiently in terms of the quadratic form of a spin structure.  Equip the disk $D$ with its unique spin structure.  The tangential framing of the boundary $\partial D$ has winding number 1 with respect to the trivialisation hence its lift $\widetilde{\partial D}$ to $D^2\times S^1$ satisfies $\eta(\widetilde{\partial D})=1$.  Thus the quadratic form is given by $q(\partial D)=\eta(\ell(\partial D))=\eta(\sigma+\widetilde{\partial D})=1+1=0$. \\

\noindent {\bf Definition~\ref{NS}*}. Given a spin structure over $\Sigma$ with associated quadratic form $q$, a boundary class $[\gamma]\in H_1(\Sigma)$ is said to be {\em Neveu-Schwarz} if $q([\gamma])=0$ and {\em Ramond} if $q([\gamma])=1$.\\

The boundary type $\vec{\epsilon}\in\{0,1\}^n$ of a spin structure consists of the quadratic form applied to each of the $n$ boundary classes, hence 0, respectively 1, for Neveu-Schwarz, respectively Ramond, boundary classes.  Since a quadratic form is a homological invariant, the number of Ramond boundary classes is necessarily even.  Thus there are $2^{n-1}$ boundary types $\vec{\epsilon}$ for a given topological surface $\Sigma=\overline{\Sigma}-D$, $D=\{p_1,...,p_n\}$.  The Teichm\"uller space of spin hyperbolic surfaces is the same as usual Teichm\"uller space despite the extra data of a spin structure.   It is the action of the mapping class group that differs which is explained as follows.  Fix a topological type of a spin structure, i.e. its boundary type $\vec{\epsilon}$ and its Arf invariant.  Given any point of Teichm\"uller space, equip it with a spin structure of the given topological type.  This choice determines a spin structure, of the same topological type, on any other point in Teichm\"uller space, by continuity and discreteness of the choice. Thus, the same Teichm\"uller space is used when the hyperbolic surfaces are equipped with spin structures and its quotient by the mapping class group defines the moduli space of spin hyperbolic  surfaces.
\begin{definition}   \label{modspace}
For $(L_1,...,L_n)\in\br_{\geq 0}^n$ and $\vec{\epsilon}\in\{0,1\}^n$, define
\begin{align*}
\modm^{\text{spin}}_{g,n,\vec{\epsilon}}(L_1,...,L_n)=\Big\{&(\Sigma,\eta,\beta_1,...,\beta_n)\mid \Sigma \text{ genus }g\text{ oriented hyperbolic surface,} \\
& \beta_i\text{ geodesic boundary component of length }\ell(\beta_i)=L_i,\\
& \text{spin structure }\eta\in H^1(P_{SO}(\Sigma),\bz_2),\quad q_\eta(\beta_i)=\epsilon_i 
\Big\}/\sim.
\end{align*}
\end{definition}
Vanishing boundary lengths correspond to hyperbolic cusps around which the hyperbolic metric is complete.  A spin Riemann surface $\Sigma=\overline{\Sigma}-\{p_1,...,p_n\}$ possesses a unique hyperbolic spin structure in its conformal class which defines a diffeomorphism
\begin{equation}   \label{modiffeo}
\modm_{g,n,\vec{\epsilon}}^{\text{spin}}(0,...,0)\cong\modm_{g,n,\vec{\epsilon}}^{\text{spin}}.
\end{equation}
When $n=0$, the notation $\modm_g^{\text{spin}}$ for the moduli space of spin hyperpolic surfaces and spin Riemann surfaces coincides, which is okay due to the natural isomorphism \eqref{modiffeo}.
The unique hyperbolic spin structure in a conformal class can be proven via gauge theory techniques due to Hitchin, described in \ref{higgscomp2}.  It is also a consequence of usual uniformisation combined with a proof of existence of a lift of any hyperbolic representation $\pi_1\Sigma\to PSL(2,\br)$ to $SL(2,\br)$, followed by adjustments of the representation by $\pm I$ to achieve any desired spin structure.
As usual, we denote the Neveu-Schwarz components of the moduli space by $\modm_{g,n,\vec{o}}^{\text{spin}}(L_1,...,L_n)$ for $\vec{o}=(0,...,0)$.  

The Mayer Vietoris sequence for $\Sigma\cup D=\overline{\Sigma}$ where $D$ is a union of disks around $\{p_i\}\subset\overline{\Sigma}$ gives the exact sequence $H_1(\Sigma\cap D,\bz_2)\to H_1(\Sigma,\bz_2)\to  H_1(\overline{\Sigma},\bz_2)$.  When all boundary classes of a spin structure are Neveu-Schwarz, the associated quadratic form $q:H_1(\Sigma,\bz_2)\to\bz_2$ vanishes on $H_1(\Sigma\cap D,\bz_2)$ hence it is the pull-back of a quadratic form defined on the symplectic vector space $H_1(\overline{\Sigma},\bz_2)$, which reflects the fact that the spin structure extends to $\overline{\Sigma}$.  The {\em Arf invariant} of a quadratic form $q$ defined on a symplectic vector space over $\bz_2$ is a $\bz_2$-valued invariant defined by
\[\text{Arf}(q)=\sum_{i=1}^g q(\alpha_i)q(\beta_i)\]
for any standard symplectic basis $\{\alpha_1,\beta_1,...,\alpha_g,\beta_g\}$ of $H_1(\overline{\Sigma},\bz_2)$, so $(\alpha_i,\beta_j)=\delta_{ij}$, $(\alpha_i,\alpha_j)=0=(\beta_i,\beta_j)$.  (More generally, the intersection form $(\cdot,\cdot)$ is replaced by the symplectic form.) This is independent of the choice of $\{\alpha_i,\beta_i\}$.  A spin structure is {\em even} if its quadratic form has even Arf invariant and {\em odd} if its quadratic form has odd Arf invariant.    Of the $2^{2g}$ spin structures with only Neveu-Schwarz boundary classes, the number of even, respectively odd, spin structures is given by $2^{g-1}(2^g+1)$, respectively $2^{g-1}(2^g-1)$.  In particular both odd and even spin structures exist for $g>0$.

By analysing the action on spin structures of the mapping class group of a genus $g$ surface $\Sigma=\overline{\Sigma}-\{p_1,...,p_n\}$  (consisting of isotopy classes of homeomorphisms that fix each $p_i$), it is proven in \cite{NatMod} that the monodromy of the $H^1(\overline{\Sigma},\bz_2)$ bundle $\modm^{\text{spin}}_{g,n,\vec{\epsilon}}\to\modm_{g,n}$ acts transitively, except in the case of only Neveu-Schwarz boundary classes where there are exactly two orbits.   This uses the symplectic action of the mapping class group on $H^1(\overline{\Sigma},\bz_2)$.  To see this, equivalently consider the action of the mapping class group on quadratic forms.  The idea is that one can choose a basis $\{a_1,b_2,...,a_g,b_g,c_1,...,c_{n-1}\}$ of $H_1(\Sigma,\bz_2)$, where $a_i\cdot b_j=\delta_{ij}$ and $c_i$ are boundary classes, with the following prescribed values of the given quadratic form $q$.   One can arrange $q(a_i)=0=q(b_i)$ for $i>1$ and $q(c_i)=\epsilon_i$.  Finally, $q(a_1)=q(b_1)=$ the Arf invariant of $q$ which is set to be zero if $\vec{\epsilon}\neq 0$.  This is achieved first algebraically, then geometrically.  It is perhaps best understood in the following example.  Suppose $g=n=1$, which necessarily has Neveu-Schwarz boundary value.  Consider two distinct quadratic forms $q_1$ and $q_2$, both with Arf invariant zero, defined on a basis $a_1,b_1$ of $H_1(\Sigma,\bz_2)$ by $q_1(a_1)=1$, $q_1(b_1)=0$ and $q_2(a_1)=0$, $q_2(b_1)=0$.  Consider a second basis $a'_1=a_1+b_1,b'_1=b_1$.  Then $q_1(a'_1)=0=q_1(b'_1)$.  Hence an element of the mapping class group that sends $a_1\to a'_1$ and $b_1\to b'_1$ pulls back $q_1$ to $q_2$.

Since the set of spin structures with fixed boundary type is an affine $H^1(\overline{\Sigma},\bz_2)$ space, this proves connectedness of components with given boundary type and Arf invariant.  Each boundary type determines a connected component of the moduli space of Fuchsian representations $\rho:\pi_1\Sigma\to SL(2,\br)$, except in one case---when all boundary classes are Neveu-Schwarz there are two connected components distinguished by the Arf invariant.

\subsubsection{} \label{quadfrep}
The quadratic form $q_\rho:H_1(\Sigma,\bz_2)\to\bz_2$ associated to a spin structure defined by a Fuchsian representation $\rho:\pi_1\Sigma\to SL(2,\br)$ has a convenient description.  We have renamed $q_{\eta_\rho}=:q_\rho$ where $\eta_\rho\in H^1(PSL(2,\br)/\overline{\rho}(\pi_1\Sigma),\bz_2)$ is the cohomology class defined by the spin structure of $\rho$.  By the decomposition of homology classes into simple closed curves used in the definition of $q_\eta=\eta\circ\ell$ above, it is enough to consider the quadratic form evaluated only on simple closed curves.  We say that $[\gamma]\in\pi_1\Sigma$ is {\em simple} if it can be represented by a simple closed curve in $\Sigma$. 
\begin{lemma}   \label{qufrep}
Given a Fuchsian representation $\rho:\pi_1\Sigma\to SL(2,\br)$, and any simple $[\gamma]\in\pi_1\Sigma$
\begin{equation}  \label{qfrep}
(-1)^{q_\rho(\floor{\gamma})}=-\sgn\tr\rho([\gamma]).
\end{equation}
where $\floor{\gamma}\in H_1(\Sigma,\bz_2)$ is the image of $[\gamma]$ under $\pi_1\Sigma\to H_1(\Sigma,\bz_2)$. 
\end{lemma}
\begin{proof}
Note that the right hand side of \eqref{qfrep} depends only on the homology class $\floor{\gamma}\in H_1(\Sigma,\bz_2)$ since $\floor{\gamma}$ uniquely determines $[\gamma]$ up to conjugation and trace is conjugation invariant.  

Evaluation of the quadratic form $q_\rho$ depends only on a neighbourhood of a simple loop in $\Sigma$ representing $[\gamma]$ since it uses only the tangential lift.  By continuity, the discrete-valued quadratic form does not change in a continuous family.  The sign of the trace separates the hyperbolic elements of $SL(2,\br)$ into two components hence it does not change in a continuous family.   To prove \eqref{qfrep}, we may first deform the representation $\rho:\pi_1\Sigma\to SL(2,\br)$ to any Fuchsian representation in the same connected component.  Moreover, we can use deformations of the representation defined only in a neighbourhood of a simple closed geodesic, that do not necessarily extend to $\Sigma$.

The dependence on a neighbourhood of a simple closed geodesic and deformation invariance of both sides of \eqref{qfrep} reduces the lemma to a single calculation.  We can take any simple closed geodesic in any hyperbolic surface.  The geodesic boundary of a one-holed torus $\Sigma$ is a well-studied example.  Given a Fuchsian representation $\overline{\rho}:\pi_1\Sigma\to PSL(2,\br)$ and $A, B\in PSL(2,\br)$ the image of the generators of $\pi_1\Sigma$, the trace of the commutator $ABA^{-1}B^{-1}$ is well-defined independently of the lift of $\overline{\rho}$ to $\rho$.  The following explicit calculation shows that $\tr(ABA^{-1}B^{-1})<0$.  Conjugate $A$ and $B$ so that $A$ is diagonal:
\[ A=\left(\begin{array}{cc}\lambda&0\\0&\lambda^{-1}\end{array}\right),\quad B=\left(\begin{array}{cc}a&b\\c&d\end{array}\right).
\]
The invariant geodesic of $A$ is given by $x=0$ in $\bh=\{x+iy\mid y>0\}$.  The invariant geodesics of $A$ and $B$ must meet since they lift from generators of $\pi_1$ of the torus.  The two fixed points of $B$ are the roots $z_1$ and $z_2$ of $cz^2+(d-a)z-b=0$, hence $z_1z_2=-b/c$.  They must lie on either side of $0$ on the real axis, hence their product is negative so $bc>0$.   
By direct calculation, $\tr(ABA^{-1}B^{-1})=1-(\lambda^2+\lambda^{-2}-1)bc<1$ since $bc>0$.  By assumption, $\Sigma$ is hyperbolic, so  $|\tr(ABA^{-1}B^{-1})|\geq 2$, hence we must have $\tr(ABA^{-1}B^{-1})\leq -2<0$.

The homology class $\floor{\gamma}$ represented by $\rho([\gamma])=ABA^{-1}B^{-1}$ is trivial hence $q(\floor{\gamma})=0$ and we have just shown $\tr(\rho([\gamma]))<0$ which agrees with \eqref{qfrep}.  Actually it proves \eqref{qfrep} since an element $\eta\in H^1(\Sigma,\bz_2)$ that is non-trivial on a homology class, say $\eta([C])=1$, sends $q(C)\mapsto q(C)+1$ and $\rho(C)\mapsto -\rho(C)\in SL(2,\br)$ which flips the sign of the trace, proving the equivalence of the negative and positive trace cases of \eqref{qfrep}.  Although a general element of a fundamental group is not a commutator, the neighbourhood of any simple closed geodesic is canonical hence behaves as in the calculated example and the lemma is proven.\\

The reduction of \eqref{qfrep} to the single calculation above is convenient, but one can also see the relationship to the sign of the trace directly as follows.  Since $q_\rho$ depends only on a neighbourhood of a simple loop we may assume that $\pi_1\Sigma=\bz$ and $\Sigma=\bh/\bz$ is a hyperbolic annulus with a unique simple closed geodesic $C\subset \Sigma$.  The spin structure is the double cover $ SL(2,\br)/\bz\to  PSL(2,\br)/\bz$.  We may deform the generator $g\in SL(2,\br)$ of $\bz\cong\langle g\rangle$ to any given element, for example a diagonal element, with trace of the same sign.  The tangential lift $\tilde{C}$ of the simple closed geodesic $C$ defines an element of $\pi_1(PSL(2,\br)/\bz)$.  If we start upstairs at $I\in SL(2,\br)/\bz$ and move around the loop downstairs, then the lift of the loop is again a loop in $SL(2,\br)/\bz$ precisely when $\sgn\tr(g)>0$ because $g$ can be deformed to $I$.  In other words $\eta_\rho(\tilde{C})=0$.  The holonomy is non-trivial when $\sgn\tr(g)<0$, or $\eta_\rho(\tilde{C})=1$.  Since $\ell(\floor{\gamma})=\sigma+[\tilde{C}]$ then we have $q_{\eta_\rho}(\floor{\gamma})=\eta_\rho\circ\ell(\floor{\gamma})=\eta_\rho(\sigma)+ \eta_\rho([\tilde{C}])=1$ when $\sgn\tr(g)>0$ and $q_{\eta_\rho}(\floor{\gamma})=0$ when $\sgn\tr(g)<0$ as required.  
\end{proof}

The set of hyperbolic and parabolic elements of $SL(2,\br)$ satisfy $|\tr\rho([\gamma])|\geq 2$, hence it has two components determined by the sign of the trace. 
Given a Fuchsian representation $\rho:\pi_1\Sigma\to SL(2,\br)$, Definition~\ref{NS} and Lemma~\ref{qufrep} show that a boundary class $[\gamma]$ is Neveu-Schwarz if $\tr\rho([\gamma])<0$ and Ramond if $\tr\rho([\gamma])>0$.

A consequence of Lemma~\ref{qufrep} and the homological nature of the quadratic form is the following property. 
\begin{cor}  \label{negtrace} 
Let $\Sigma$ be a surface with boundary classes $\gamma_1,...\gamma_n$.  Any Fuchsian representation $\rho:\pi_1\Sigma\to SL(2,\br)$ satisfies
\[(-1)^n\prod_{i=1}^n\tr(\rho([\gamma_i]))>0.
\]
\end{cor}
This property of the product of traces of Fuchsian representations into $SL(2,\br)$ has been studied particularly in the 2-generator free group case---as the negative trace theorem in \cite{MasMat}---proving that for the pair of pants and the once-punctured torus, the product of the traces of the boundary classes is negative.

\subsection{Flat bundles}

In this section we realise the spinor bundle $S_\Sigma\to\Sigma$ of a hyperbolic surface equipped with a spin structure as a flat bundle.  Equivalently, there exists a flat connection on $S_\Sigma$, which must differ from the lift of the Levi-Civita connection by cohomological considerations---see Remark~\ref{flateuler}.  The flat structure is visible via representations of $\pi_1\Sigma$ into $SL(2,\br)$.
\subsubsection{}  \label{sec:spinor} 
The right action of $\text{Spin}(2)=SO(2)$ on $P_{\text{Spin}}(\Sigma)\cong SL(2,\br)/\rho(\pi_1\Sigma)$ (where $\rho(\pi_1\Sigma)$ acts on the left of $SL(2,\br)$) is used to define the associated spinor bundle 
\begin{equation}  \label{spinor}
S_\Sigma=P_{\text{Spin}}(\Sigma)\times_{SO(2)}\bc^2\cong\left(\bh\times\bc^2\right)/\rho(\pi_1\Sigma).
\end{equation}
The flat real bundle $T_\Sigma^\frac12$ is obtained by replacing $\bc^2$ with $\br^2$ in \eqref{spinor}.  The right hand side of \eqref{spinor} defines a flat bundle over $\Sigma$ associated to the representation $\rho:\pi_1\Sigma\to SL(2,\br)$ where the action is given by $g\cdot(z,v)=(g\cdot z,g\cdot v)$.  The map $SL(2,\br)\times\bc^2\ni(g,u)\mapsto (g\cdot i,gu)\in\bh\times\bc^2$ defines the isomorphism in \eqref{spinor}.   It is well-defined on orbits $(gk^{-1},ku)$, $k\in SO(2)$  and descends to the quotient by $\rho(\pi_1\Sigma)$ on both sides.   

The spinor bundle $S_\Sigma$ is flat hence holomorphic.  We show below that $T_\Sigma^\frac12$ is a subbundle of $S_\Sigma$ in two different ways, compatible with the flat, respectively holomorphic, structure of $S_\Sigma$.   It is the underlying flat real bundle $T_\Sigma^\frac12\stackrel{r}{\to} S_\Sigma$ which is the fixed point set of the real involution on $S_\Sigma$.  It is also a holomorphic subbundle $T_\Sigma^\frac12\stackrel{h}{\to}  S_\Sigma$ which is an eigenspace of the action of $SO(2)$.  The images of $r$ and $h$ intersect trivially.

The weights $\chi^{\pm 1}$, defined in \ref{spin}, of the $SO(2)$ representation of $\bc^2=\bc_\chi\oplus\bc_{\chi^{-1}}$ defines a decomposition of $S_\Sigma$ into holomorphic line bundles $S_\Sigma=T_\Sigma^\frac12\oplus T_\Sigma^{-\frac12}$.  With respect to this decomposition, $SL(2,\br)$ acts via $SU(1,1)$, i.e. the matrix of any $g\in SL(2,\br)$, with respect to a basis of eigenvectors of $\chi^{\pm 1}$, lives in $SU(1,1)$.  With respect to the decomposition $\bc^2=\bc_\chi\oplus\bc_{\chi^{-1}}$, the real structure $\sigma$ on $\bc^2$ (which is complex conjugation with respect to a complex structure different to that on $\bc_\chi\oplus\bc_{\chi^{-1}}$) is given by $(u,v)\mapsto(\overline{v},\overline{u})$.  The real structure commutes with the actions of the structure groups of the bundle, $SO(2)$ on the left hand side of \eqref{spinor} and $SL(2,\br)$ on the right hand side of \eqref{spinor}.  (Note that $SL(2,\br)$ commutes with complex conjugation and $SU(1,1)$ commutes with $\sigma(u,v)=(\overline{v},\overline{u})$ which is the same group action and real structure with respect to different bases.)  Hence the bundle $S_\Sigma$ is equipped with a real structure $\sigma$ with fixed point set the underlying flat real bundle $T_\Sigma^\frac12$, obtained by replacing $\bc^2$ with $\br^2$ on both sides of \eqref{spinor}.  In \ref{realstr} the real structure on $S_\Sigma$ will involve the Hermitian metric used to reduce the structure group to $SO(2)$.   

\begin{remark}  \label{flateuler}
Note that the flat bundle $T_\Sigma^\frac12$ has non-zero Euler class.  The Euler class can be obtained via a metric connection on $T_\Sigma^\frac12$ as described in Section~\ref{sec:eulerform}, so in particular if the metric connection were flat, the Euler class would vanish.   There is no contradiction here because $\br^2$ admits no metric invariant under $SL(2,\br)$, so we cannot find a metric on $T_\Sigma^\frac12$ which is preserved by its flat connection.  This example is discussed by Milnor and Stasheff in \cite[p.312]{MStCha}.
\end{remark}

\subsubsection{}   \label{conjiso}
A Hermitian metric $h$ on a line bundle $L\to\Sigma$ defines an isomorphism 
$\overline{L}\stackrel{\cong}{\to}L^\vee$ 
by $\ell\mapsto h(\overline{\ell},\cdot)$, where $\overline{L}$ is the conjugate bundle, defined via conjugation of transition functions.  For example, a metric on a Riemann surface compatible with its conformal structure is equivalent to a Hermitian metric $h^2$ on $T_\Sigma$, and moreover it is equivalent to a Hermitian metric on any power $K_{\Sigma}^{\otimes n}$ such as a choice of spin structure $K_{\Sigma}^{1/2}$.  Hence 
\[\overline{K_{\Sigma}^{\otimes n}}\ \stackrel{h^*}{\cong}\left(K_{\Sigma}^{-1}\right)^{\otimes n}\]  
where the isomorphism $h^*$ depends on the Hermitian metric on $K_{\Sigma}^{\otimes n}$ via $\ell\mapsto h(\overline{\ell},\cdot)^{2n}$.

\subsubsection{} \label{realstr}
The real structure $\sigma$ defined on the spinor bundle $S_\Sigma=T_\Sigma^\frac12\oplus T_\Sigma^{-\frac12}$ in \ref{sec:spinor} is induced by the isomorphism $\overline{T}_\Sigma^{\frac12}\stackrel{h^*}{\cong}T_\Sigma^{-\frac12}$, from the Hermitian metric $h$ on $T_\Sigma^\frac12$ which is the square root of the hyperbolic metric on $\Sigma$.  It is defined on local sections by
\[\sigma(u,v)=(h^{-1}\overline{v},h\overline{u}).\]  
The underlying real bundle $T_\Sigma^\frac12$ is the subbundle of fixed points of $\sigma$ which is locally given by $(u,h\overline{u})$.  In particular $u\mapsto(u,h\overline{u})$ defines a natural isomorphism between the flat real subbundle and the holomorphic subbundle given by an eigenspace of the action of $SO(2)$, both isomorphic to $T_\Sigma^\frac12$.

\subsubsection{} \label{flatcoh}
A flat bundle $E$ over a surface $\Sigma$ defines a locally constant sheaf given by its sheaf of locally flat sections which we also denote by $E$.  We denote its sheaf cohomology by $H^i_{dR}(\Sigma,E)$.   We will apply this to the spinor bundle $E=S_\Sigma$ and its underlying real bundle $E=T_\Sigma^\frac12$.  The sheaf cohomology can be calculated in different ways, and the label $dR$ for de Rham, following Simpson \cite{SimHig}, refers to its calculation via the following complex which uses the covariant derivative $d_A$ defined by the flat connection on $E$:
\begin{equation}  \label{covcomplex}
A^0_\Sigma(E)\stackrel{d_A}{\longrightarrow}A^1_\Sigma(E)\stackrel{d_A}{\longrightarrow}A^2_\Sigma(E).
\end{equation}
Here $A^k_\Sigma(E):=\Gamma(\Sigma,\Lambda^k(T^*\Sigma)\otimes E)$ denotes global $C^\infty$ differential $k$-forms with coefficients in $E$.  It defines a complex because $d_A\circ d_A=F^A\in\Omega^2(\text{End}E)$ is given by the curvature which vanishes in this case.  Define $H^i_{dR}(\Sigma,E)$ for $i=0,1,2$ to be the cohomology of the complex.  We rarely use the complex \eqref{covcomplex} directly and instead mainly use \v{C}ech cohomology to calculate $H^i_{dR}(\Sigma,E)$.

\subsubsection{} \label{cech}
The sheaf cohomology $H^i_{dR}(\Sigma,E)$ can be calculated using \v{C}ech cohomology applied to an open cover of $\Sigma$ obtained from a triangulation.  A triangulation of $\Sigma$ is a simplicial complex  $\cc=\displaystyle\mathop{\cup}_{k=0}^2\cc_k$ where $\cc_k$ denotes $k$-simplices $\sigma:\Delta_k\to\Sigma$, and we further require the regularity condition that each 2-simplex is a homeomorphism onto its image.  The regularity condition ensures that 2-simplices incident at an edge or vertex are distinct.  We identify simplices with their images in $\Sigma$ and refer to them as faces, edge and vertices of the triangulation.   To each simplex $\sigma$ of the triangulation associate the open set $U_\sigma\subset\Sigma$ given by the union of the interiors of all simplices whose closure contains $\sigma$.  Hence, to each vertex of the triangulation $v\in\cc_0$, associate the open set $U_v\subset\Sigma$ given by the union of the interiors of all simplices whose closure meets $v$, as in Figure~\ref{covtri}, so it includes the vertex $v$, no other vertices, and the interiors of all incident edges and faces.  
\begin{center}
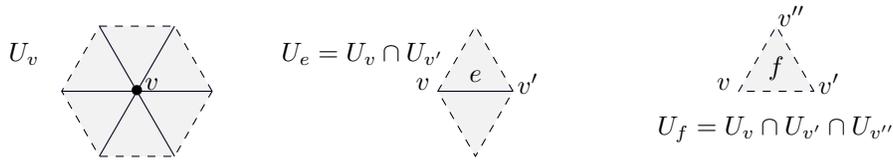

\begin{tikzpicture}[scale=1]
\filldraw[fill=gray!10!white, dashed, draw=blue!10!black]
(1,0)--(.5,.866)--(-.5,.866)--(-1,0)--(-.5,-.866)--(.5,-.866)--cycle;
\draw[ draw=blue!10!black]
(-1,0)--(1,0);
\draw[ draw=blue!10!black]
(-.5,-.866)--(.5,.866);
\draw[ draw=blue!10!black]
(-.5,.866)--(.5,-.866);
\draw (.2,.1) node {$v$};
\draw (0,0) node {$\bullet$};
\draw (-1.5,.5) node {$U_v$};

\filldraw[fill=gray!10!white, dashed, draw=blue!10!black]
(5,0)--(4.5,.866)--(4,0)--(4.5,-.866)--cycle;
\draw[ draw=blue!10!black]
(4,0)--(5,0);
\draw (3.8,.1) node {$v$};
\draw (5.2,.1) node {$v'$};
\draw (4.5,.2) node {$e$};
\draw (3,.5) node {$U_e=U_v\cap U_{v'}$};

\filldraw[fill=gray!10!white, dashed, draw=blue!10!black]
(9,0)--(8.5,.866)--(8,0)--cycle;
\draw (7.8,.1) node {$v$};
\draw (9.2,.1) node {$v'$};
\draw (8.7,1) node {$v''$};
\draw (8.5,.3) node {$f$};
\draw (8.5,-.5) node {$U_f=U_v\cap U_{v'}\cap U_{v''}$};
\end{tikzpicture}
\captionof{figure}{Open cover associated to triangulation}\label{covtri}
\end{center}

This produces an open cover:
\begin{equation}  \label{tricov} 
\Sigma=\bigcup_{\sigma\in\cc}U_\sigma.
\end{equation} 
We allow more general cell decompositions where faces of the triangulation can be polygons, not only triangles.  For $v$ and $v'$ vertices of an edge $e$, and the vertices of a face $f$ we have
\[U_e=U_v\cap U_{v'},\quad U_f=\bigcup_{v\in f}U_v.\]  
Note that $U_v\cap U_{v'}$ or $U_v\cap U_{v'}\cap U_{v''}$ is empty if there is no edge containing $v$ and $v'$, or face containing $v$, $v'$ and $v''$.  For example, given a triangulation, where faces are indeed triangles, for more than three distinct vertices $\{v_i\}$ the intersection is empty $\displaystyle\bigcap_i U_{v_i}=\varnothing$.
On a compact surface, one can define the open cover using only the vertices $\Sigma=\bigcup_{v\in\cc_0}U_v$ so that the sets associated to edges and faces are not part of the cover, and instead arise as intersections.  This results in fewer coboundary maps in the construction of \ref{cechcoh}.

We allow a generalisation of triangulations, where some of the vertices are missing (from both $\Sigma$ and the triangulation) which is particularly useful for non-compact $\Sigma$.  In this case, the regularity condition on a face is required only in its domain which is a 2-simplex with some vertices removed.  Hence $U_e$ and $U_f$ may not arise as intersections of $U_v$ for $v\in\cc_0$ justifying the open cover \eqref{tricov}.   The set of vertices may be empty, as is the case for ideal triangulations, in which case there are no open sets $U_v$.

\subsubsection{}  \label{cechcoh}
The \v{C}ech cohomology of the sheaf of locally constant sections can be calculated from a subspace of the \v{C}ech cochains with respect to the open cover \eqref{tricov} of the sheaf of locally constant sections of $E$, defined by 
\[\displaystyle C^k(\Sigma,E)=\bigoplus_{\sigma\in\cc_k} \Gamma(U_\sigma,E),\quad k=0,1,2.\]  
The coboundary map $\delta$ is given by restriction and the \v{C}ech cohomology $H^\bullet_{dR}(\Sigma,E)$ is equal to the cohomology of the complex
\begin{equation} \label{cechseq}
0\to \bigoplus_{v\in\cc_0} \Gamma(U_v,E)\stackrel{\delta}{\to}\bigoplus_{e\in\cc_1} \Gamma(U_e,E)\stackrel{\delta}{\to}\bigoplus_{f\in\cc_2} \Gamma(U_f,E)\to 0.
\end{equation}
Note that $C^k(\Sigma,E)=0$ for $k>2$ since we have discarded such intersections which contribute trivially to the cohomology.  If we allow more general cell decompositions where faces of the triangulation can be polygons, not only triangles, then we include non-trivial $C^k(\Sigma,E)$ for $k>2$, but still $H^k_{dR}(\Sigma,E)=0$ for $k>2$.

Since the cohomology of \eqref{cechseq} defines the sheaf cohomology $H_{dR}^k(\Sigma,E)$ it is independent of the choice of cell decomposition of $\Sigma$.  It follows that duality of triangulations gives duality of cohomology groups.

\subsubsection{}   \label{cechhom}

\v{C}ech cohomology was calculated in \ref{cechcoh} using a {\em good open cover}, meaning that intersections of open sets in the cover are contractible, which is achieved from the regularity condition on triangulations.  

If we relax the regularity condition in \ref{cech} on a triangulation $\cc=\displaystyle\mathop{\cup}_{k=0}^2\cc_k$ of $\Sigma$ so that a 2-simplex is not necessarily one-to-one onto its image, we describe a construction, used in \cite{SWiJTG}, of the sheaf cohomology of $E$ as follows.  It coincides with the dual of the construction in \ref{cechcoh} when the triangulation satisfies the regularity condition.

For $\sigma\in\cc$, let $\cv_\sigma=H^0(\sigma,E)$ denote the covariant constant sections $s|_{\sigma}$ of $E$ over $\sigma$.  Here we identify $\sigma$ with its image.
Define 
\[C_k(\Sigma,E)=\displaystyle\bigoplus_{\sigma\in\cc_k}\cv_\sigma\]
and boundary maps 
\[\begin{array}{rcl}C_{k+1}(\Sigma,E)&\stackrel{\partial}{\to}& C_k(\Sigma,E)\\
s|_{\sigma}&\mapsto& s|_{\partial\sigma}=\bigoplus (-1)^{\epsilon_i }s|_{\sigma_i}
\end{array}
\]
where $\displaystyle\partial\sigma=\bigcup_i (-1)^{\epsilon_i}\sigma_i$ as oriented simplices.  
A section $s|_{\sigma}$ is well-defined on the pull-back of $E$ to the cell, but possible multiply-defined on the boundary of $\sigma$, and we use the extension from the interior in the definition of $\partial$.  This ambiguity arises precisely due to the relaxation of the regularity condition in \ref{cech}.  

It is clear that $\partial^2=0$ since the contribution at any vertex of a 2-cell essentially gives the covariant constant section extended to the vertex, appearing with opposite sign due to orientations, or vanishing of the square of the usual boundary map on simplices.  The same argument applies to higher dimensional simplices and their codimension two cells.  One can approach the vertex along two edges, and the vanishing then reflects the trivial local holonomy of the flat connection.    

Denote by $H_k(\Sigma,E)$ the homology of the complex
\[ C_2(\Sigma,E)\stackrel{\partial}{\to} C_1(\Sigma,E)\stackrel{\partial}{\to} C_0(\Sigma,E).
\]
\subsubsection{}
There is a natural symplectic structure on $S_\Sigma$ and $T_\Sigma^\frac12$ arising from the symplectic form on $\bc^2$ and $\br^2$ preserved by the $SL(2,\br)$ action.  Hence there is a natural isomorphism $C_k(\Sigma,S_\Sigma)\cong C_k(\Sigma,S_\Sigma^\vee)\cong C_k(\Sigma,S_\Sigma)^\vee$ which gives a natural isomorphism
\[C_k(\Sigma,S_\Sigma)^\vee\cong C^k(\Sigma,S_\Sigma).
\]
Moreover, $(\partial\eta,f)=(\eta,\delta f)$ since both sides use the symplectic form applied to the extension of $\eta$ and $f$ or $\eta$ and the restriction of $f$ which is the same.  Thus we see that
\[H_k(\Sigma,S_\Sigma)^\vee\cong H_{dR}^k(\Sigma,S_\Sigma)\]  
and the same isomorphism holds for $T_\Sigma^\frac12$.

When the triangulation is regular, the isomorphism between cohomology and homology is visible via the cochains in \ref{cechcoh} and the chains in \ref{cechhom} coinciding,  $C^k(\Sigma,S_\Sigma)=C_k(\Sigma,S_\Sigma)$, while the maps $\delta$ and $\partial$ go in opposite directions.  In terms of the open sets $U_\sigma$ defined in \ref{cech}, $\delta$ are restriction maps while $\partial$ are extension maps.  

\subsubsection{} \label{fatgraph} 
An ideal triangulation of a non-compact surface $\Sigma$ is a triangulation with no vertices, and all faces triangles.  The number of faces and edges is $4g-4+2n$, respectively $6g-6+3n$ for $\Sigma=\overline{\Sigma}-\{p_1,...,p_n\}$ of genus $g$.  Dual to an ideal triangulation is a trivalent fatgraph $\Gamma=V(\Gamma)\cup E(\Gamma)$ which is an embedded graph that is a deformation retract (a spine) of $\Sigma$, and consists of only vertices $V(\Gamma)$ and edges $E(\Gamma)$, and no faces.  The fatgraph $\Gamma$ has {\em type} $(g,n)$.  A trivalent fatgraph determines an ideal triangulation uniquely, and hence the two notions are equivalent.  

With respect to an ideal triangulation, $H_{dR}^k(\Sigma,T_\Sigma^\frac12)$ is conveniently calculated using the dual fatgraph.  
The complex is rather simple since there are only 2-cochains and 1-cochains.   Or dually,  using the fatgraph $\Gamma$ there are only 0-chains and 1-chains.   We can equally work with the restriction of the flat bundle $T_\Sigma^\frac12|_\Gamma$ which we also denote by $T_\Sigma^\frac12$.  Following \ref{cechhom}, for $e\in E(\Gamma)$, let $\cv_e$ denote the covariant constant sections $s|_e$ of $T_\Sigma^\frac12$ over $e$, and for $v\in V(\Gamma)$, let $\cv_v$ denote the covariant constant sections $s|_v$ of $T_\Sigma^\frac12$ over $v$.
Define 
\[C_0(\Gamma,T_\Sigma^\frac12)=\displaystyle\bigoplus_{v\in V(\Gamma)}\cv_v,\quad C_1(\Gamma,T_\Sigma^\frac12)=\displaystyle\bigoplus_{e\in E(\Gamma)}\cv_e\]
and boundary maps 
\[\begin{array}{rcl}C_{1}(\Gamma,T_\Sigma^\frac12)&\stackrel{\partial}{\to}& C_0(\Gamma,T_\Sigma^\frac12)\\
s|_e&\mapsto& s|_{\partial e}=s|_{e_+}-s|_{e_-}
\end{array}
\]
where $e_\pm\in V(\Gamma)$ are the vertices bounding the oriented edge $e$.

The sheaf cohomology $H_{dR}^k(\Sigma,T_\Sigma^\frac12)$ is given by the homology of the complex
\begin{equation}  \label{GammaHom}
C_1(\Gamma,T_\Sigma^\frac12)\stackrel{\partial}{\to} C_0(\Gamma,T_\Sigma^\frac12).
\end{equation}
We have $H_{dR}^1(\Sigma,T_\Sigma^\frac12)\cong H_1(\Gamma,T_\Sigma^\frac12)=\ker\partial$ and 
$H_{dR}^0(\Sigma,T_\Sigma^\frac12)\cong H_0(\Gamma,T_\Sigma^\frac12)=0$.  The vanishing of $H_{dR}^0(\Sigma,T_\Sigma^\frac12)$ uses the ideal triangulation so in particular there are no $0$-cochains.

\begin{thm}  \label{cohbun}
For any hyperbolic spin surface $\Sigma$ with Neveu-Schwarz geodesic boundary components of lengths $(L_1,...,L_n)\in\br_{\geq 0}^n$
\[H^1_{dR}(\Sigma,T_\Sigma^\frac12)\cong\br^{4g-4+2n}
\]
and this defines a vector bundle 
\[ E_{g,n}\to\modm^{\text{spin}}_{g,n,\vec{o}}(L_1,...,L_n)
\]
with fibres $H^1_{dR}(\Sigma,T_\Sigma^\frac12)$.
\end{thm}
\begin{proof}
First consider the case when $\Sigma$ is non-compact hence admits an ideal triangulation.  
A hyperbolic spin surface is equivalent to a flat $SL(2,\br)$ connection over the dual fatgraph $\Gamma$ of the (truncated) ideal triangulation of $\Sigma$.  Arbitrarily orient each edge of $\Gamma$.  The flat connection is equivalent to associating an element $g_e\in SL(2,\br)$ to each oriented edge $e$ of $\Gamma$.   The holonomy around any oriented loop $\gamma\subset\Gamma$ is the product $g_\gamma=\prod g_e^{\pm1}$ of the elements along edges of the loop with $\pm1$ determined by whether the orientation of the edge agrees with the orientation of the loop.  The holonomy around any oriented loop satisfies $|\tr g_\gamma|\geq 2$.

An element of $H^1_{dR}(\Sigma,T_\Sigma^\frac12)\cong\ker\partial$ in \eqref{GammaHom} is a collection of vectors $v_e\in\br^2$ assigned to each oriented edge, satisfying a condition at each vertex.  We choose the convention that the trivialisation of $T_\Sigma^\frac12$ over an oriented edge $e$ is induced from the trivialisation of $T_\Sigma^\frac12$ over its source vertex $e_-$.  Hence 
\[\partial v_e|_{e_+}=g_ev_e,\quad\partial v_e|_{e_-}=-v_e.\]
The condition at a vertex is the vanishing of the sum of contributions from the three oriented edges adjacent to the given vertex, such as $\sum g_ev_e=0$ for a vertex with only incoming edges, or more generally each summand is $g_ev_e$ or $-v_e$.

Choose an ideal triangulation of $\Sigma$ with dual fatgraph $\Gamma$ that admits a {\em dimer} covering $D\subset E(\Gamma)$ which is a collection of $2g-2+n$ edges such that each vertex of $\Gamma$ is the boundary of a unique edge in the dimer.  The existence of such an  ideal triangulation, or equivalently such a fatgraph, together with a dimer covering, is proven by construction as follows.  Any trivalent fatgraph (without dimer) can be constructed by repeated application of the following two modifications of a trivalent fatgraph.

1. Introduce two new vertices on (the interiors of) edges (possibly the same edge) and attach a new edge to these vertices as in the left figure.

2.  Introduce one new vertex on (the interior of) an edge and attach a lollipop = graph consisting of a loop attached to an edge as in the right figure.
\begin{center}

\begin{tikzpicture}[
  edge/.style={thick},
  dashededge/.style={thick, dashed}
]

\def\r{2}

\draw[dashededge]
  (\r,20) arc[start angle=20, end angle=-20, radius=\r];

\draw[dashededge]
  (\r+2,20) arc[start angle=160, end angle=200, radius=\r];
  
  \draw[edge] (2.1,19.3)--(3.9,19.3);

  \draw[dashededge]
  (\r+6,20) arc[start angle=160, end angle=200, radius=\r];

\draw[edge] (6.5,19.3) circle (.6);
\draw[edge] (7.1,19.3)--(7.9,19.3);

\end{tikzpicture}
\end{center}
Such a modification adds two vertices and three edges  $(|V|,|E|)\mapsto(|V|+2,|E|+3)$ and the Euler characteristic changes by $\chi\mapsto\chi-1$.  Apply this initially to a circle to produce any connected trivalent fatgraph.  The proof of this is seen by the reverse procedure of removing edges so that the graph remains connected (or allow disconnected fatgraphs).
To produce a fatgraph equipped with a dimer, use the above construction with the condition that one must attach only to non-dimer edges, and label the new edge, or stem of the lollipop, a dimer edge.

Given a fatgraph $\Gamma$ with dimer $D$, we will prove that for all edges $e$ of $D$ the vectors $v_e\in\br^2$ can be arbitrarily and independently assigned, and they uniquely determine the vectors on all other edges, hence they produce a basis of $2(2g-2+n)$ vectors for $H_1(\Gamma,T_\Sigma^\frac12)$.  In Remark~\ref{whitehead} below we show how to produce a basis of $2(2g-2+n)$ vectors for $H_1(\Gamma,T_\Sigma^\frac12)$ for any dual fatgraph $\Gamma$, not necessarily admitting a dimer covering.

Given $e_0\in D$, choose an arbitrary non-zero $v_{e_0}\in\br^2$ and set $v_{e}=0$ for all other dimer edges $e\in D\backslash\{e_0\}$.  Since $\Gamma$ is trivalent, $\Gamma\backslash D$ is a collection of embedded loops.  Along an oriented loop $\gamma\subset\Gamma\backslash D$, the vertex condition on elements of $\ker\delta$ uniquely determines each vector $v_e$ on an edge $e\in\gamma$ from the preceding edge.  For example, if the orientation on each edge agrees with the orientation on $\gamma$, then $ge_i=e_{i+1}$ where $e_i$ and $e_{i+1}$ are consecutive oriented edges in $\gamma$.  

If a loop $\gamma\subset\Gamma\backslash D$ avoids $e_0$, then we must have $v_e=g_\gamma v_e$ where $e$ is an edge of $\gamma$ and $g_\gamma$ is the holonomy around the loop starting from $e$.  But $g_\gamma-I$ is invertible, or equivalently $g_\gamma$ does not have eigenvalue 1, since non-boundary loops satisfy $|\tr g_\gamma|> 2$ and boundary loops satisfy $\tr g_\gamma\leq - 2$ by the Neveu-Schwarz requirement.  Hence $v_e=0$ for all edges $e\in\gamma$.

If a loop $\gamma\subset\Gamma\backslash D$ meets $e_0$, then we now have 
\[(g_\gamma-I)v_e-v_{e_0}=0\]
(or  $(g_\gamma-I)v_e+g_{e_0}v_{e_0}=0)$ and since $g_\gamma-I$ is invertible this uniquely determines $v_e\in\br^2$ and all vectors along $\gamma$.  

Hence a choice of non-zero $v_{e_0}\in\br^2$ uniquely determines a vector in $\ker\delta$.  Clearly elements of $\ker\delta$ associated to different dimer edges are linearly independent because each vanishes on the other dimer edges.  We also see that if an element of $\ker\delta$ vanishes on all dimer edges then it vanishes identically.  Hence each edge $e\in D$ determines two independent vectors in $H_1(\Gamma,T_\Sigma^\frac12)$, and the union over the $2g-2+n$ edges in $D$ produces a basis of $2(2g-2+n)$ vectors for $H_1(\Gamma,T_\Sigma^\frac12)$.

We have proved $H^1_{dR}(\Sigma,T_\Sigma^\frac12)\cong\br^{4g-4+2n}$ which is the first part of the Theorem.  In fact we have a canonical isomorphism between $H^1_{dR}(\Sigma,T_\Sigma^\frac12)$ and $(\br^2)^D$, for $D\subset E(\Gamma)$ a dimer covering.  But this gives a local trivialisation over the moduli space $\modm^{\text{spin}}_{g,n,\vec{o}}(L_1,...,L_n)$ since a choice of ideal triangulation defines the Teichm\"uller space of the moduli space.  A choice of $D\subset E(\Gamma)$ is well-defined on the Teichm\"uller space producing a trivial bundle $(\br^2)^D$, from which we get a local trivialisation over the moduli space.\\

When $\Sigma$ is compact it has genus $g>1$, and we choose a decomposition $\Sigma=\Sigma_1\cup\Sigma_2$ into genus $g-1$ and genus 1 surfaces glued along boundary annuli.  We have $H^k_{dR}(\Sigma_1\cap\Sigma_2,T_\Sigma^\frac12)=0$ for $k=0,1$ by hyperbolicity of the holonomy as follows.  For $U\cup V =\Sigma_1\cap\Sigma_2$, the sequence \eqref{cechseq} becomes
\[0\to \Gamma(U,T_\Sigma^\frac12)\stackrel{\delta}{\to}\Gamma(V,T_\Sigma^\frac12)\to 0
\]
with boundary map $\delta=g_\gamma-I$ where $g_\gamma$ is the holonomy around a loop $\gamma\subset\Sigma_1\cap\Sigma_2$.  But $g_\gamma$ is hyperbolic so it satisfies $|\tr g_\gamma|>2$ and in particular $g_\gamma-I$ is invertible, and the cohomology groups $H^k_{dR}(\Sigma_1\cap\Sigma_2,T_\Sigma^\frac12)=0$ vanish.

Hence the Mayer-Vietoris sequence gives
\[
0\to H^1_{dR}(\Sigma,T_\Sigma^\frac12)\to H^1_{dR}(\Sigma_1,T_\Sigma^\frac12)\oplus H^1_{dR}(\Sigma_2,T_\Sigma^\frac12)\to 0.
\]
We have shown above that $H^1_{dR}(\Sigma_1,T_\Sigma^\frac12)\cong\br^{4g-6}$ and $H^1_{dR}(\Sigma_2,T_\Sigma^\frac12)\cong\br^2$ and they define local trivialisations over the respective moduli spaces of bundles $E_{g-1,1}$ and $E_{1,1}$.  This gives a local decomposition $E_g\cong E_{g-1,1}\oplus E_{1,1}$ proving that $E_g$ is indeed a vector bundle.   The decomposition $\Sigma=\Sigma_1\cup\Sigma_2$ does not make sense over the moduli space since the mapping class group does not preserve the decomposition, and is only well-defined over Teichm\"uller space.  Nevertheless, it does make sense locally which is enough to prove that $E_g$ is a rank $4g-4$ vector bundle.

\end{proof}

\begin{remark}  \label{whitehead}
In Theorem~\ref{cohbun}, one can drop the assumption that the dual fatgraph $\Gamma$ of the ideal triangulation of $\Sigma$ must admit a dimer covering.  On any dual fatgraph $\Gamma$, there exists a collection $C\subset E(\Gamma)$ of $2g-2+n$ edges of $\Gamma$ on which the vectors $v_e\in\br^2$ can be independently assigned, and which uniquely determine the vectors on all other edges.  We call such a collection $C$ a {\em base} of edges of $\Gamma$.  Each edge $e\in C$ determines two independent vectors in $H_1(\Gamma,T_\Sigma^\frac12)$, and the union over the $2g-2+n$ edges in $C$ produces a basis of $2(2g-2+n)$ vectors for $H_1(\Gamma,T_\Sigma^\frac12)$.

To prove the existence of a base of edges, begin with a fatgraph with a dimer covering as constructed in the proof of Theorem~\ref{cohbun}.  Any ideal triangulation of $\Sigma$ can be obtained from any another ideal triangulation by Whitehead moves.
\begin{center}
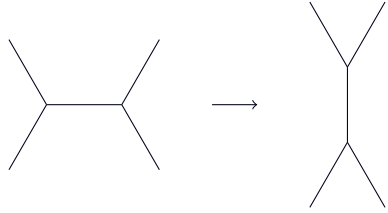

\begin{tikzpicture}[scale=1]
\draw[ draw=blue!10!black]
(-.5,-.866)--(0,0);
\draw[ draw=blue!10!black]
(-.5,.866)--(0,0);
\draw[ draw=blue!10!black]
(0,0)--(1,0);
\draw[ draw=blue!10!black]
(1,0)--(1.5,.866);
\draw[ draw=blue!10!black]
(1,0)--(1.5,-.866);

\draw[ draw=blue!10!black]
(3.5,-1.366)--(4,-0.5);
\draw[ draw=blue!10!black]
(3.5,1.366)--(4,0.5);
\draw[ draw=blue!10!black]
(4,0.5)--(4,-0.5);
\draw[ draw=blue!10!black]
(4,0.5)--(4.5,1.366);
\draw[ draw=blue!10!black]
(4,-0.5)--(4.5,-1.366);

\draw [->, draw=blue!10!black] 
(2.2,0)--(2.8,0);

\end{tikzpicture}
\captionof{figure}{Whitehead move}\label{flip}
\end{center}
Under a Whitehead move, the existence of a dimer covering is only preserved if the contracted edge is a dimer edge.  Nevertheless, we will consider Whitehead moves that destroy the dimer covering.  There is a natural bijection of edges under Whitehead moves, and a base of edges is sent to a base of edges under this bijection.   Since we compute cohomology of $\Sigma$, which is independent of the choice of $\Gamma$, there is a natural isomorphism $H_1(\Gamma,T_\Sigma^\frac12)\cong H_1(\Gamma',T_\Sigma^\frac12)$ when $\Gamma$ and $\Gamma'$ are related by a Whitehead move.  In particular, the image $C'\subset E(\Gamma')$ of a base of edges $C\subset E(\Gamma)$ under the Whitehead move inherits the following two properties of $C$---for $e\in C$ the vectors $v_e\in\br^2$ can be independently assigned, and uniquely determine the vectors on all edges in $\Gamma\backslash C$---and thus is also a base of edges.  Thus, beginning with a fatgraph equipped with a dimer, which is chosen to be the base of edges, via Whitehead moves, we can find a base of edges on any fatgraph.  The figure below gives an example of a graph that does not admit a dimer covering.  In place of a dimer, the thickened edges form a base of edges.
\end{remark}

\begin{center}
\begin{tikzpicture}[scale=.7,
  edge/.style={thick}
]

\def\r{0.5}

\coordinate (v1) at (0,0);
\coordinate (v2) at (3,0);
\coordinate (v3) at (1.5,2.5);
\coordinate (v4) at (1.5,1);

\draw[edge] (v1) circle (\r);
\draw[edge] (v2) circle (\r);
\draw[edge] (v3) circle (\r);

\draw[edge,line width=3pt] 
  (v4) -- ($(v1)! \r-.21 ! (v4)$);

\draw[edge,line width=3pt] 
  (v4) -- ($(v2)! \r-.21 ! (v4)$);

\draw[edge] 
  (v4) -- ($(v3)! \r-.18 ! (v4)$);

\node[circle, fill=black, inner sep=1.6pt] at (v4) {};

\end{tikzpicture}
\end{center}

\subsection{Higgs bundles}  \label{sec:higgs}
In this section we will prove that the restriction of the bundle $\widehat{E}_{g,n}\to\overline{\modm}_{g,n,\vec{o}}^{\text{spin}}$ defined in Definition~\ref{obsbun} to the smooth moduli space gives the bundle $E_{g,n}\to \modm^{\text{spin}}_{g,n,\vec{o}}$ defined by Theorem~\ref{cohbun} combined with the isomorphism $\modm_{g,n,\vec{o}}^{\text{spin}}\cong\modm_{g,n,\vec{o}}^{\text{spin}}(0,...,0)$.  
The constructions of the bundles $\widehat{E}_{g,n}$ and $E_{g,n}$ over the moduli spaces of stable and smooth spin curves respectively use the cohomology of different sheaves.  We will prove that over smooth spin curves $\Sigma=\overline{\Sigma}-D$ the following sheaf cohomology groups are isomorphic
\begin{equation}  \label{flatholiso}
H^1_{dR}(\Sigma,T_\Sigma^\frac12)\cong H^1(\overline{\Sigma},T_{\overline{\Sigma}}^\frac12(-D))
\end{equation}  
when the spin structure has Neveu-Schwarz boundary components.
The natural way to prove the isomorphism \eqref{flatholiso} relating flat and holomorphic structures on bundles over $\Sigma$ uses Higgs bundles.   More precisely, there is a natural identification of any flat structure on a bundle $E\to\Sigma$, with an extension of $E$ to $\overline{\Sigma}$ equipped with a holomorphic structure, Higgs field and parabolic structure.    Applied to the spinor bundle $E=S_\Sigma$, this gives a natural way to realise uniformisation of $\Sigma$ which naturally associates a unique hyperbolic metric on $\Sigma$ in the conformal class defined by $\Sigma$.  Furthermore, it gives an isomorphism between the respective moduli spaces.  We will see that the sheaves on both sides of \eqref{flatholiso} arise naturally from this proof of uniformisation.    

The use of Higgs bundles achieves two goals.  It relates the sheaf cohomologies arising from a flat structure and a holomorphic structure on a bundle.  It also relates cohomological constructions on a non-compact Riemann surface $\Sigma=\overline{\Sigma}-D$ and on the compact pair $(\overline{\Sigma},D)$.  We will start with the case when $\Sigma$ is compact, i.e. $D=\varnothing$.  This will simplify the exposition and focus only on the first goal.  Then we will consider the general case, which requires parabolic structures on bundles over $(\overline{\Sigma},D)$.  The general proof essentially follows the proof in the compact case with some technical adjustments.

\subsubsection{}  \label{higgscomp1}

Higgs bundles over a compact Riemann surface $\Sigma$ with canonical bundle $K_\Sigma$ were defined by Hitchin in \cite{HitSel} as follows.
\begin{definition}  \label{defHiggs}
A Higgs bundle over a compact Riemann surface $\Sigma$ is a pair $(E,\phi)$ where $E$ is a holomorphic vector bundle over $\Sigma$ and $\phi\in H^0(\text{End}(E)\otimes K_\Sigma)$.
\end{definition}  
The pair $(E,\phi)$ is {\em stable} if for any $\phi$-invariant subbundle $F\subset E$, i.e. $\phi(F)\subset F\otimes K_\Sigma$, we have $\frac{c_1(F)}{\text{rank\ }F}<\frac{c_1(E)}{\text{rank\ }E}$.
When $\phi=0$, every subbundle is $\phi$-invariant and the definition of stable reduces to the usual definition of stable for a holomorphic bundle $E$.

A Hermitian structure on $E$ is a Hermitian metric $H$ defined on $E$ with respect to its complex structure.  It defines a reduction of the structure group of $E$ from $GL(n,\bc)$ to $U(n)$.  The holomorphic structure and Hermitian metric $H$ on $E$ together define a unitary connection $A$ on $E$ via $d_A=\overline{\partial}+\overline{\partial}^*$, where $\overline{\partial}_A=\overline{\partial}$ is the natural operator on $E$ and $\partial_A$ is the adjoint of $\overline{\partial}_A$ with respect to $H$.  The curvature of a unitary connection $A$ on $E$ is a unitary endomorphism valued two-form $F_A$.  Since $[\phi,\phi^*]$ is also a unitary endomorphism valued two-form, they can be compared.  The connection $A$ (or equivalently the Hermitian metric $H$) is said to satisfy the Higgs bundle equations if
\begin{equation}  \label{HiggsEq}
F_A+[\phi,\phi^*]=0
\end{equation}
Importantly, \eqref{HiggsEq} is equivalent to the connection $A+\phi+\phi^*$ being a flat $SL(2,\bc)$ connection.  This relation between holomorphic and flat structures will be used to relate those structures on $T_\Sigma^\frac12$.

One can consider a broader class of sections $\phi$, allowing them to be smooth endomorphism valued one-forms and add to \eqref{HiggsEq} the equation 
\[ 
\overline{\partial}_A\phi=0
\] 
which is the condition that $\phi$ is holomorphic.  This makes the invariance of the equations under the unitary gauge group clear but now $\overline{\partial}\mapsto \overline{\partial}_A$.   Note that constant unitary gauge transformations are both holomorphic gauge transformations and smooth gauge transformations, and in particular they preserve $\overline{\partial}$.

\begin{thm}[Hitchin \cite{HitSel}]  \label{hitchin}
A stable Higgs bundle $(E,\phi)$ of degree zero admits a unique unitary connection $A$ satisfying \eqref{HiggsEq}.  Conversely a Higgs bundle $(E,\phi)$ which admits a connection $A$ satisfying \eqref{HiggsEq} is of degree zero and stable.
\end{thm}

\subsubsection{}  \label{higgscomp2}
Apply Theorem~\ref{hitchin} to the spinor bundle $E=S_\Sigma=T_\Sigma^\frac12\oplus T_\Sigma^{-\frac12}$ with Higgs field
\begin{equation}  \label{higgs}
\phi=\frac12\left(\begin{array}{cc}0&1\\0&0\end{array}\right)\in H^0(\text{End}(E)\otimes K_\Sigma)
\end{equation}
 where $1$ is the natural section of $\co_\Sigma\cong T_\Sigma^\frac12\otimes T_\Sigma^\frac12\otimes K_\Sigma$ which gives a linear map $T_\Sigma^{-\frac12}\to T_\Sigma^\frac12\otimes K_\Sigma$.  The only $\phi$-invariant subbundle of $S_\Sigma$ is $T_\Sigma^\frac12$ and for $g>1$ we have $1-g=c_1(T_\Sigma^\frac12)<\frac12c_1(S_\Sigma)=0$, so the pair $(S_\Sigma,\phi)$ is stable.  (More generally, one can choose $(S_\Sigma,\phi)$ for $\phi=\frac12\left(\begin{array}{cc}0&1\\q&0\end{array}\right)$ for $q\in H^0(K_\Sigma^2)$, a quadratic differential.  We will not consider this here.)  

Hitchin \cite{HitSel}, showed that the two sides of Theorem~\ref{hitchin} applied to $(S_\Sigma,\phi)$ naturally correspond to a hyperbolic metric and a conformal structure, leading to a proof of uniformisation as follows.  The key idea is to show that $A$ is reducible so the associated Hermitian metric on $S_\Sigma$ is also reducible and defines a Hermitian metric on $T_\Sigma^\frac12$.  Theorem~\ref{hitchin} produces a  unique unitary connection $A$ on $S_\Sigma$.  For a constant $\alpha\in\br$, $(A,e^{i\alpha}\phi)$ also satisfies \eqref{HiggsEq}.  We can act by a constant unitary gauge transformation, 
which preserves \eqref{HiggsEq} and holomorphicity of $\phi$, to get 
\[u_\alpha\cdot(A,e^{i\alpha}\phi)=(u_\alpha\cdot A,e^{i\alpha}u_\alpha\cdot\phi)=(u_\alpha\cdot A,\phi),\quad u_\alpha=\left(\begin{array}{cc}e^{-i\alpha/2}&0\\0&e^{i\alpha/2}\end{array}\right).\]  
Since $(A,\phi)$ and $(u_\alpha\cdot A,\phi)$ satisfy \eqref{HiggsEq}, by the uniqueness of $A$ we must have $u_\alpha\cdot A=A$ for each $\alpha\in\br$ so the connection $A$ is reducible.     

Corresponding to the reducible connection $A$ is a reducible Hermitian metric $H=h\oplus h^{-1}$ on $S_\Sigma$ where $h$ is defined on $T_\Sigma^\frac12$ so $h^2$ defines a Hermitian metric on $\Sigma$ with real part a Riemannian metric.  Write $h^2=h_0^2dz\otimes d\bar{z}$ where $h_0=h_0(z,\bar{z})$ is a locally defined real-valued function.  The curvature of the connection on $T_\Sigma^\frac12$, is given by $(\partial_{\bar{z}}\partial_z\log h_0 )d\bar{z}\wedge dz$ and satisfies 
\eqref{HiggsEq}.  This yields 
\[\partial_{\bar{z}}\partial_z\log h_0 d\bar{z}\wedge dz +\frac14h_0^2 dz\wedge d\bar{z}=0\] 
or $\partial_{\bar{z}}\partial_z\log h_0 =\frac14h_0^2$.  Hence the Gaussian curvature of the associated Riemannian metric is
\[K=-\frac{2}{h^2}\frac{\partial^2}{\partial z \partial\bar{z}}\log h_0^2=-1\]
which proves uniformisation for a compact Riemann surface $\Sigma$---it possesses a hyperbolic metric in its conformal class.  The $SL(2,\br)$ holonomy of the flat connection $A+\phi+\phi^*$ lives above the $PSL(2,\br)$ holonomy of the developing map of the hyperbolic metric on $\Sigma$.

\subsubsection{}  \label{higgscomp3}
We are now in a position to compare $H^k(\Sigma,T_\Sigma^\frac12)$ and $H^k_{dR}(\Sigma,T_\Sigma^\frac12)$.  The flat connection $A^\phi=A+\phi+\phi^*$ on $S_\Sigma$ coming out of Theorem~\ref{hitchin} is given in terms of its $(1,0)$ and $(0,1)$ parts by
\[\partial_{A^\phi}=\left(\begin{array}{cc}\partial+h^{-1}\partial h&\frac12\\0&\partial-h^{-1}\partial h\end{array}\right),\quad 
\overline{\partial}_{A^\phi}=\left(\begin{array}{cc}\overline{\partial}&0\\ \frac12 h^2&\overline{\partial}\end{array}\right)
\]
where, as above, the upper right term is a linear map $T_\Sigma^{-\frac12}\to T_\Sigma^\frac12\otimes K_\Sigma$ and the lower left term is its adjoint $T_\Sigma^\frac12\to T_\Sigma^{-\frac12}\otimes \overline{K}_\Sigma$.  Note that $\phi^*$ is an $\text{End}(S_\Sigma)$-valued $(0,1)$ form, so a Hermitian metric $\frac12h^2=\frac12h_0^2dz\otimes d\bar{z}$ naturally lives in the lower left position, rather than a quadratic differential which would yield an $\text{End}(S_\Sigma)$-valued $(1,0)$ form.

The connection $A^\phi$ is compatible with the real structure $\sigma$
\[d_{A^\phi}\circ\sigma=\sigma\circ d_{A^\phi}\]
and it is enough to prove $\partial_{A^\phi}\circ\sigma=\sigma\circ\overline{\partial}_{A^\phi}$:
\begin{align*}
\partial_{A^\phi}\circ\sigma\left(\begin{array}{c}u\\v\end{array}\right)
&=\partial_{A^\phi}\left(\begin{array}{c}h^{-1}\overline{v}\\h\overline{u}\end{array}\right)
=\left(\begin{array}{c}\frac12h\overline{u}+h^{-1}\partial\overline{v}\\h\partial\overline{u}\end{array}\right)\\
&=\sigma\left(\begin{array}{c}\overline{\partial}u\\\frac12h^2u+\overline{\partial}v\end{array}\right)
=\sigma\circ\overline{\partial}_{A^\phi}\left(\begin{array}{c}u\\v\end{array}\right).
\end{align*}
Hence it defines a flat $SU(1,1)\cong SL(2,\br)$ connection on the bundle $S_\Sigma$.

\subsubsection{}  \label{higgscomp4}
The Higgs field defines a complex
\[
0\to \Omega^0_\Sigma(S_\Sigma)\stackrel{\phi\cdot}{\to}\Omega^1_\Sigma(S_\Sigma)\to 0.
\]
Simpson \cite{SimHig} defined the Dolbeault cohomology of $S_\Sigma$ to be the hypercohomology of this complex $H^k_{\text{Dol}}(\Sigma,S_\Sigma):=\bh^k([\Omega^0_\Sigma(S_\Sigma)\to\Omega^1_\Sigma(S_\Sigma)])$ and proved the following relation with the sheaf cohomology of the flat bundle $S_\Sigma$.
\begin{thm}[Simpson \cite{SimHig}]   \label{simthm}
When $\Sigma$ is compact, there is a canonical isomorphism
\[H^k_{\text{dR}}(\Sigma,S_\Sigma)\cong H^k_{\text{Dol}}(\Sigma,S_\Sigma),\quad k=0,1,2.
\]
\end{thm}
An application of this theorem is the following crucial canonical isomorphism.

\begin{thm}   \label{caniso}
When $\Sigma$ is compact, there is a canonical isomorphism
\begin{equation}   \label{cohiso}
H^k(\Sigma,T_\Sigma^\frac12)^\vee\stackrel{\cong}{\to} H^k_{dR}(\Sigma,T_\Sigma^\frac12),\quad k=0,1,2
\end{equation}
where $T_\Sigma^\frac12$ represents the sheaf of locally holomorphic sections on the left hand side, and the sheaf of locally constant sections on the right hand side.
\end{thm}
\begin{proof}
The first step is to evaluate the hypercohomology in Simpson's theorem.  Hypercohomology is an invariant of the quasi-isomorphism class of a complex of sheaves.  For $\phi$ given by \eqref{higgs}, the map
$T_\Sigma^\frac12\oplus T_\Sigma^{-\frac12}\stackrel{\phi\cdot}{\to}(T_\Sigma^\frac12\oplus T_\Sigma^{-\frac12})\otimes K_\Sigma$ has kernel $T_\Sigma^\frac12$ and cokernel $T_\Sigma^{-\frac12}\otimes K_\Sigma$ and defines an isomorphism $T_\Sigma^{-\frac12}\stackrel{\cong}{\to}T_\Sigma^\frac12\otimes K_\Sigma$.  Hence the natural inclusions given by the vertical arrows below define a quasi-isomorphism:
\[
\begin{array}{ccc} \Omega^0_\Sigma(T_\Sigma^\frac12)&\stackrel{0\cdot}{\to}&\Omega^1_\Sigma(T_\Sigma^{-\frac12})\\
\downarrow&&\downarrow\\
\Omega^0_\Sigma(T_\Sigma^\frac12\oplus T_\Sigma^{-\frac12})&\stackrel{\phi\cdot}{\to}&\Omega^1_\Sigma(T_\Sigma^\frac12\oplus T_\Sigma^{-\frac12}).
\end{array}
\]
Thus $H^k_{\text{Dol}}(\Sigma,S_\Sigma)=\bh^k(C^\bullet)$ where $C^\bullet=[ \Omega^0_\Sigma(T_\Sigma^\frac12)\to\Omega^1_\Sigma(T_\Sigma^{-\frac12})]$ and the arrow is the zero map.  The hypercohomology can be calculated from a long exact sequence
\[..\to H^{k-1}(\Sigma,\Omega^1_\Sigma(T_\Sigma^{-\frac12}))\hspace{-.5mm}\to\bh^k(C^\bullet)\hspace{-.5mm}\to
H^k(\Sigma,T_\Sigma^\frac12)\hspace{-.5mm}\to H^k(\Sigma,\Omega^1_\Sigma(T_\Sigma^{-\frac12}))\to...
\]

Thus 
\[\bh^0(C^\bullet)\cong H^0(\Sigma,T_\Sigma^\frac12)=0\] 
for $g>1$ since $\deg T_\Sigma^\frac12=1-g<0$, and 
\[\bh^2(C^\bullet)\cong H^2(\Sigma,T_\Sigma^\frac12)=0\] 
for $g>1$ since $H^1(\Sigma,\Omega^1_\Sigma(T_\Sigma^{-\frac12}))\cong H^0(\Sigma,T_\Sigma^\frac12)^\vee=0$.   We see that \eqref{cohiso} is proven for $k=0$ and 2 by Theorem~\ref{simthm} and the injection $H^k_{dR}(\Sigma,T_\Sigma^\frac12)\hookrightarrow H^k_{dR}(\Sigma,S_\Sigma)=0$.

It remains to prove the $k=1$ case.  The sequence
\[0\to H^{0}(\Sigma,\Omega^1_\Sigma(T_\Sigma^{-\frac12}))\to\bh^1(C^\bullet)\to
H^1(\Sigma,T_\Sigma^\frac12)\to 0
\]
splits giving 
\[\bh^1(C^\bullet)\cong H^1(\Sigma,T_\Sigma^\frac12)\oplus H^1(\Sigma,T_\Sigma^\frac12)^\vee
\]
which uses the isomorphism $H^1(\Sigma,T_\Sigma^\frac12)^\vee\cong H^0(\Sigma,K_\Sigma\otimes T_\Sigma^{-\frac12})$.
The complex vector space $H^1(\Sigma,T_\Sigma^\frac12)$ is equipped with a Hermitian metric induced from the Hermitian metric on $T_\Sigma^\frac12$---see Section~~\ref{sec:eulerform}.   Hence its dual vector space is isomorphic to its complex conjugate.  Equivalently
\[\bh^1(C^\bullet)\cong H^1(\Sigma,T_\Sigma^\frac12)\otimes_\br\bc
\]
which completes the calculation of the hypercohomology.

We have $H^1_{dR}(\Sigma,S_\Sigma)=H^1_{dR}(\Sigma,T_\Sigma^\frac12)\otimes\bc$ by construction.  So Simpson's theorem proves that there is a canonical isomorphism
\[H^1(\Sigma,T_\Sigma^\frac12)\otimes_\br\bc\cong H^1_{dR}(\Sigma,T_\Sigma^\frac12)\otimes_\br\bc.
\]
To see the real structure of the isomorphism, we need to understand the proof of the canonical isomorphism in \cite{SimHig} which uses a quasi-isomorphism between the complexes 
\[A^0_\Sigma(S_\Sigma)\stackrel{D_i}{\to}A^1_\Sigma(S_\Sigma)\stackrel{D_i}{\to}A^2_\Sigma(S_\Sigma)\]
for $D_1=\overline{\partial}_A$ and $D_2=d_A+\phi+\phi^*$ and the identity map on $A^k_\Sigma(S_\Sigma)$.
The kernel of $D_1$ naturally produces representatives in $H^1(\Sigma,T_\Sigma^\frac12)\oplus H^{0}(\Sigma,\Omega^1_\Sigma(T_\Sigma^{-\frac12}))$ since $A$ is diagonal and when $H^0(\Sigma,T_\Sigma^{-\frac12})\neq0$, the sequence is
\[H^0(\Sigma,T_\Sigma^{-\frac12})\to H^0(\Sigma,T_\Sigma^{-\frac12})\oplus H^1(\Sigma,K_\Sigma\otimes T_\Sigma^\frac12)\to H^1(\Sigma,K_\Sigma\otimes T_\Sigma^\frac12)\]
which has vanishing cohomology.  The map to the kernel of $D_2$  is described as follows.  Given a $T_\Sigma^{-\frac12}$-valued holomorphic 1-form $\eta\in H^0(\Sigma,K_\Sigma\otimes T_\Sigma^{-\frac12})\subset A^1_\Sigma(T_\Sigma^{-\frac12})$ 
then $(h^{-1}\overline{\eta},\eta)\in A^1_\Sigma(T_\Sigma^\frac12\oplus T_\Sigma^{-\frac12})= A^1_\Sigma(S_\Sigma)$
and in fact takes its values in the real part $A^1_\Sigma(T_\Sigma^\frac12)$ (using the antidiagonal embedding $T_\Sigma^\frac12\to S_\Sigma$ which differs from the first factor embedding---see \ref{sec:spinor}).  

For $\eta\in  H^0(\Sigma,K_\Sigma\otimes T_\Sigma^{-\frac12})$,
\begin{align}  \label{covder}
d_{A^\phi}\left(\begin{array}{c}h^{-1}\overline{\eta}\\ \eta \end{array}\right)
&=\partial_{A^\phi}\left(\begin{array}{c}h^{-1}\overline{\eta}\\ 0\end{array}\right)+
\overline{\partial}_{A^\phi}\left(\begin{array}{c}0\\ \eta \end{array}\right)
\\
&=\left(\begin{array}{c}h^{-1}\partial\overline{\eta}\\ 0\end{array}\right)+\left(\begin{array}{c}0\\ \overline{\partial}\eta \end{array}\right)
=\left(\begin{array}{c}0\\ 0\end{array}\right) \nonumber
\end{align}
where the first equality uses the fact that $\eta$ is a $(1,0)$ form and the second equality uses $\overline{\partial}_A=\overline{\partial}$ and $\partial_A=\partial+h^{-1}\partial h$.  The final equality uses the holomorphicity of $\eta$.  Hence $(h^{-1}\overline{\eta},\eta)$ is a cocycle in $A^1_\Sigma(T_\Sigma^\frac12)$.

Thus we have defined a natural map
\begin{equation}  \label{sh2dr}
\begin{array}{ccc}H^1(\Sigma,T_\Sigma^\frac12)^\vee&\to& H^1_{dR}(\Sigma,T_\Sigma^\frac12)\\
\eta&\mapsto &(h^{-1}\overline{\eta},\eta)
\end{array}
\end{equation}
which indeed defines an isomorphism by the following lemma.\footnote{The author is grateful to Edward Witten for explaining the proof of this lemma.}
\begin{lemma}  \label{cohdir}
Given a cocycle $\alpha\in A^1_\Sigma(S_\Sigma)$ so $d_{A^\phi}\alpha=0$, there exists a unique $\beta\in A^0_\Sigma(S_\Sigma)$ such that 
\begin{equation}  \label{cocyc}
\alpha-d_{A^\phi}\beta=\left(\begin{array}{c}0\\ \ast\end{array}\right)dz+\left(\begin{array}{c}\ast\\0\end{array}\right)d\bar{z}.
\end{equation}
\end{lemma}
\begin{proof}
Let $\beta=\left(\begin{array}{c}w\\h\overline{w}\end{array}\right)$ and decompose $\alpha$ into its $(1,0)$ and $(0,1)$ parts.  
\[\alpha=\alpha'+\alpha''=\left(\begin{array}{c}u\\ v\end{array}\right)+\left(\begin{array}{c}h^{-1}\overline{v}\\ h\overline{u}\end{array}\right)\]
It is enough to solve $\alpha'-\partial_{A^\phi}\beta=\left(\begin{array}{c}0\\ \ast\end{array}\right)$ since $\overline{\partial}_{A^\phi}$ sends $\beta$ to a $(0,1)$-form.  Hence
\[Pw:=\partial w+(h^{-1}\partial h) w+\tfrac12h\overline{w}=u.
\]
Here $P$ is a real linear elliptic operator acting on a rank 2 real vector bundle.  It has trivial kernel because if $Pw=0$ then its complex conjugate equation is $\frac12h^2w+\overline{\partial}(h\overline{w})=0$ hence
\[Pw=0\ \Rightarrow\ 0=\overline{\partial}_A(Pw)=\overline{\partial}_A\partial_Aw+\tfrac12\overline{\partial}_A(h\overline{w})=(\overline{\partial}_A\partial_A-\tfrac14h^2)w\ \Rightarrow\ w=0
\]
where the second implication uses the fact that the operator $\overline{\partial}_A\partial_A-\tfrac14h^2$ is negative definite which follows from the following standard argument that the operator $\overline{\partial}_A\partial_A$ is negative semi-definite.
\begin{align*}
\int_\Sigma \langle\overline{\partial}_A\partial_As,s\rangle&=
-\int_\Sigma \langle\partial_As,\partial_As\rangle+\int_\Sigma\partial\langle\partial_As,s\rangle\\
&=-\int_\Sigma \langle\partial_As,\partial_As\rangle+\int_\Sigma d\langle\partial_As,s\rangle
=-\int_\Sigma \langle\partial_As,\partial_As\rangle\leq 0.
\end{align*}
The replacement of $\partial$ by $d$ in the second equality, which leads to vanishing of the integral, uses the three facts: $d=\partial+\overline{\partial}$, $\langle\partial_As,s\rangle$ is a $(0,1)$ form, and the space of $(0,2)$ forms is zero.  Hence $P$ is invertible, and we can solve $Pw=u$ uniquely.

By the reality condition, the vanishing of the first coefficient of $dz$ guarantees the vanishing of the second coefficient of $d\overline{z}$ as required.
\end{proof}
Lemma~\ref{cohdir} shows that we may assume any cocycle in $A^1_\Sigma(T_\Sigma^\frac12)$ is of the form in the right hand side of \eqref{cocyc} hence we can use \eqref{covder}, which only needs the given $(1,0)$ and $(0,1)$ decomposition of the right hand side of \eqref{cocyc}, to deduce that the $dz$ part is holomorphic, i.e. lives in $H^0(\Sigma,K_\Sigma\otimes T_\Sigma^{-\frac12})$.  By the reality condition the cocycle lives in the image of \eqref{sh2dr}.  Thus the map in \eqref{sh2dr} is surjective onto equivalence classes of cocycles representing classes in $H^1_{dR}(\Sigma,T_\Sigma^\frac12)$.   It is injective since if $(h^{-1}\overline{\eta},\eta)=d_{A^\phi}\beta$ is exact, by the invertibility of the elliptic operator $P$, i.e. the uniqueness statement in Lemma~\ref{cohdir}, $\beta=0$.

Hence we have proven
\[H^1(\Sigma,T_\Sigma^\frac12)^\vee\cong H^1_{dR}(\Sigma,T_\Sigma^\frac12).
\]
\end{proof}
We have proved that the fibres over a point represented by a smooth compact hyperbolic surface of the bundles $\widehat{E}_{g}\to\overline{\modm}_{g}^{\text{spin}}$ defined in Definition~\ref{obsbun} and $E_{g}\to \modm^{\text{spin}}_{g}$ defined in Theorem~\ref{cohbun}  are canonically isomorphic.  The importance of the canonical isomorphism is that the {\em bundles} are isomorphic over the moduli space of smooth spin curves.  An analogous canonical isomorphism exists for the usual moduli space using $H^1(\Sigma,T_\Sigma)$ and $H^1_{dR}(\Sigma,{\bf g}_\rho)$ where ${\bf g}_\rho$ is the flat ${\bf sl}(2,\br)$-bundle associated to a representation $\rho:\pi_1\Sigma\to SL(2,\br)$.

\subsubsection{}  \label{noncompact}
We now consider general $\Sigma=\overline{\Sigma}-D$, dropping the earlier assumption that $\Sigma$ is compact.  The arguments in \ref{higgscomp1}, \ref{higgscomp2}, \ref{higgscomp3} and \ref{higgscomp4} generalise.    When $\Sigma$ is not compact, the bundle $S_\Sigma$ can have different extensions to $\overline{\Sigma}$.  
We will use the extension of $S_\Sigma$ given by
\[E\cong T_{\overline{\Sigma}}^\frac12(-D)\oplus T_{\overline{\Sigma}}^{-\frac12}.
\]
The bundle $E$ naturally possesses a parabolic structure which we now define, following Mehta and Seshadri \cite{MSeMod}. 
\begin{definition}
Let $(\overline{\Sigma},D)$ be a compact surface containing $D=\sum p_i$ and $E$ a holomorphic vector bundle over $\overline{\Sigma}$.  A parabolic structure on $E$ is a flag at each point $p_i$, $E_{p_i}=F_1^i\supset F_2^i\supset ...\supset F_{r_i}^i$, with attached weights $0\leq\alpha_1^i< \alpha_2^i< ... <\alpha_{r_i}^i<1$.
\end{definition}  
Define the multiplicity of $\alpha_j^i$ to be $k_j^i=\dim F_j^i-\dim F_{j+1}^i$, $j=1,...,r_i-1$ and $k_{r_i}^i=\dim F_{r_i}^i$.  The parabolic degree of $E$ is defined to be
\[\text{pardeg}\hspace{.7mm}E=\deg E+\sum_{i,j}k_j^i\alpha_j^i.
\]
A parabolic Higgs bundle generalises Definition~\ref{defHiggs} where the Higgs field has poles on $D$ and preserves the flag structure.
\begin{definition}
A parabolic Higgs bundle over $(\overline{\Sigma},D)$ is a pair $(E,\phi)$ where $E$ is a holomorphic vector bundle over $(\overline{\Sigma},D)$ equipped with a parabolic structure $\{F_j^i,\alpha_j^i\}$ and $\phi\in H^0(\text{End}(E)\otimes K_{\overline{\Sigma}}(D))$ which satisfies $\Res_{p_i}\phi F_j^i\subset F_j^i$.
\end{definition}
Note that some authors also write $K_{\overline{\Sigma}}(\log D))=K_{\overline{\Sigma}}(D)$ where the two coincide over a curve $\overline{\Sigma}$ but differ on higher dimensional varieties.

The following pair is a parabolic Higgs bundle generalising the construction in \ref{higgscomp2}.
\[E\cong T_{\overline{\Sigma}}^\frac12(-D)\oplus T_{\overline{\Sigma}}^{-\frac12},\quad\phi=\frac12\left(\begin{array}{cc}0&1\\0&0\end{array}\right)\in H^0(\text{End}(E)\otimes K_{\overline{\Sigma}}(D)).
\]
Following \cite{BGGPar}, at each point $p_i$ of $D$, $E_{p_i}$ is equipped with the trivial flag $E_{p_i}$ of weight $1/2$.  
Note that $\phi$ does indeed have a pole at each point $p_i$ of $D$ and we take its residue to test for stability.  We see the pole in the upper right element of $\phi$ which gives a map $T_{\overline{\Sigma}}^{-\frac12}\to T_{\overline{\Sigma}}^\frac12(-D)\otimes K_{\overline{\Sigma}}(D)$, or an element of 
\[\co_{\overline{\Sigma}}\cong T_{\overline{\Sigma}}^\frac12\otimes T_{\overline{\Sigma}}^\frac12(-D)\otimes K_{\overline{\Sigma}}(D).\] Locally, the upper right element of $\phi$ produces $z/dz:T_{\overline{\Sigma}}^{-\frac12}\to T_{\overline{\Sigma}}^\frac12(-D)$ which is the residue of $1=z/dz\cdot dz/z$.  For the same reason as described in \ref{higgscomp2}, the pair $(E,\phi)$ is stable, which now means that for any $\phi$-invariant sub-parabolic bundle $F\subset E$, we have $\frac{\text{pardeg}\hspace{.7mm}(F)}{\text{rank\ }F}<\frac{\text{pardeg}\hspace{.7mm}(E)}{\text{rank\ }E}$.  Note that the weights $1/2$ at each point correspond to the Neveu-Schwarz boundary components which is necessary here.  In \cite{BGGPar}, the choice of a Neveu-Schwarz spin structure is not stated explicitly but it is implicit due to the choice of parabolic weights.  Such a choice is arbitrary since that paper is concerned only with the underlying hyperbolic surface, or equivalently the reduction of the representation from $SL(2,\br)$ to $PSL(2,\br)$.

\begin{thm}[Simpson \cite{SimHar}]  \label{simpsonhar}
A stable parabolic Higgs bundle $(E,\phi)$ of parabolic degree zero admits a unique unitary connection $A$ with regular singularities satisfying \eqref{HiggsEq}.  Conversely a parabolic Higgs bundle $(E,\phi)$ which admits a connection $A$ with regular singularities satisfying \eqref{HiggsEq} is of parabolic degree zero and stable.
\end{thm}
The connection must preserve the weight spaces of the parabolic structure on the bundle.  This condition is automatic for our application since the weight space is the entire fibre.  A regular singularity means a pole of order 1 of an algebraic connection---see \cite[p.724]{SimHar} for details.  Biswas, Gastesi and Govindarajan \cite{BGGPar} applied Theorem~\ref{simpsonhar} to the stable parabolic bundle $E\cong T_{\overline{\Sigma}}^\frac12(-D)\oplus T_{\overline{\Sigma}}^{-\frac12}$ to prove uniformisation of $\Sigma$ by a complete hyperbolic metric analogous to the argument of Hitchin presented in \ref{higgscomp2}.

Simpson proved in  \cite{SimHar} that there is a natural quasi-isomorphism between the de Rham complex of forms with coefficients in the flat bundle, and the Dolbeault complex with coefficients in the corresponding Higgs bundle.   A consequence is the equality of cohomology groups.
\begin{thm}[\cite{DPSDir,SimHar}]   \label{simthmpar}
For a spin structure with Neveu-Schwarz boundary components, there is a canonical isomorphism
\[H^k_{\text{dR}}(\Sigma,S_\Sigma)\cong H^k_{\text{Dol}}(\overline{\Sigma},T_{\overline{\Sigma}}^\frac12(-D)\oplus T_{\overline{\Sigma}}^{-\frac12})^\vee.
\]
\end{thm}
\begin{remark}  \label{parorb}
When the spin structure has Neveu-Schwarz boundary components, we have an isomorphism
\[
H^k_{\text{Dol}}(\overline{\Sigma},T_{\overline{\Sigma}}^\frac12(-D)\oplus T_{\overline{\Sigma}}^{-\frac12})\cong H^k(\cc,\theta^\vee\oplus\theta)
\]
where $\cc$ is an orbifold curve as described in Section~\ref{sec:theta} with non-trivial isotropy group $\bz_2$ at $D$, $\theta^2=\omega_\cc(D)$ and its coarse curve is $p:(\cc,D)\to(\overline{\Sigma},D)$.  The push-forward of a bundle over $\cc$ to the coarse curve $\overline{\Sigma}$ is a bundle on $\overline{\Sigma}$ equipped with a parabolic structure \cite{BodRep,FStSei}. 
We find that
\[ p_*(\theta^\vee\oplus\theta)=T_{\overline{\Sigma}}^\frac12(-D)\oplus T_{\overline{\Sigma}}^{-\frac12}\]
equipped with the trivial flag of weight $1/2$ at each point of $D$.  
\end{remark}
Theorem~\ref{simthmpar} allows us to drop the assumption that $\Sigma$ is compact in Theorem~\ref{caniso}.
\begin{thm}   \label{caniso1}
There is a canonical isomorphism
\begin{equation}   \label{cohiso1}
H^k(\overline{\Sigma},T_{\overline{\Sigma}}^\frac12(-D))^\vee\stackrel{\cong}{\to} H^k_{dR}(\Sigma,T_\Sigma^\frac12),\quad k=0,1,2
\end{equation}
for spin structures with Neveu-Schwarz boundary components.
\end{thm}
The proof is the same as the proof of Theorem~\ref{caniso}.  The direct argument of Lemma~\ref{cohdir} goes through when we replace cohomology with cohomology with compact supports.

\subsubsection{} \label{orb}
In \ref{noncompact} the sheaf cohomology of a flat bundle over non-compact $\Sigma$ was related to the sheaf cohomology of a bundle over a compactification $\overline{\Sigma}$ of $\Sigma$.  A conformal structure on a punctured surface can compactify in different ways and we show here that it naturally compactifies to an orbifold curve $\cc$ with $\bz/2$ orbifold structure at $D=\cc-\Sigma$.  This is important to relate to the bundle $E_{g,n}$ constructed in Section~\ref{sec:theta}

As in Remark~\ref{parorb}, we push forward bundles over $\cc$ using the map $p:(\cc,D)\to(\overline{\Sigma},D)$ that forgets the orbifold structure at $D$.  For Neveu-Schwarz divisor $D$, as explained in the introduction, the non-trivial representation induced by $\theta^\vee$ along $D$ makes the local sections vanish on $D$ hence:
\[ p_*\theta^\vee=T_{\overline{\Sigma}}^\frac12(-D)
\]
and in particular 
\[H^1(\theta^\vee)=H^1((\overline{\Sigma},T_{\overline{\Sigma}}^\frac12(-D)).
\]
Hence by Theorem~\ref{caniso1}, over a smooth spin complete hyperbolic surface $\Sigma$ with Neveu-Schwarz boundary components, there is a canonical isomorphism of cohomology groups $H^1(\theta^\vee)^\vee\cong H^1_{dR}(\Sigma,T_\Sigma^\frac12)$ which allows us to prove the following.
\begin{cor}  \label{isobun}
The bundles defined in Definition~\ref{obsbun} and  Theorem~\ref{cohbun} are isomorphic on the smooth part of the Neveu-Schwarz component of the moduli space:
\begin{equation}  
\widehat{E}_{g,n}|_{\modm_{g,n,\vec{o}}^{\text{spin}}}\cong E_{g,n}.
\end{equation}
\end{cor}

\subsection{Euler form of $E_{g,n}$}   \label{sec:eulerform}

A canonical Euler form of $E_{g,n}\to\modm_{g,n,\vec{o}}^{\text{spin}}$ is constructed by using the natural hyperbolic metric associated to each curve of the moduli space.  More precisely, an Euler form is constructed on the dual bundle $E_{g,n}^\vee$ which is equivalent to an Euler form on $E_{g,n}$ via $e(E_{g,n}^\vee)=(-1)^ne(E_{g,n})$.  It is used in the definition of the volume of the moduli space of super hyperbolic surfaces.

 Let $E\to M$ be a real oriented bundle of rank $N$.   An Euler form 
\[e(E)\in\Omega^N(M)\] 
is uniquely determined by a choice of Riemannian metric $\langle\cdot,\cdot\rangle$ on $E$ together with a metric connection $A$, meaning that $d\langle s_1,s_2\rangle=\langle\nabla^As_1,s_2\rangle+\langle s_1,\nabla^As_2\rangle$ for sections $s_1$ and $s_2$ of $E$.  The curvature of the connection is an endomorphism-valued 2-form $F_A\in\Omega^2(M,\text{End}(E))$.  The endomorphism preserves the metric $\langle\cdot,\cdot\rangle$ hence $F_A$ is locally $so(N)$-valued.  The Pfaffian defines a map $\text{pf}:so(N)\to\br$ rather like the determinant.  It vanishes for $N$ odd and for $N$ even is defined using (but independent of the choice of) an orthonormal basis $\{e_1,...,e_N\}$ by 
\[\frac{1}{(N/2)!}B\wedge B\wedge ... \wedge B=:\text{pf}(B)e_1\wedge ... \wedge e_N,\qquad B\in\wedge^2\br^N\cong so(N).\] 
It satisfies $\text{pf}(B)^2=\det(B)$.  It is invariant under conjugation by $O(N)$, i.e. $\text{pf}(gBg^{-1})=\text{pf}(B)$ for $g\in O(N)$, hence makes sense on the associated ${\bf so}(N)$ bundle, and in particular on $F_A$.  The Euler form is defined as a polynomial in the curvature $F_A$ using the Pfaffian \cite{OsbRep} 
\begin{equation}  \label{eulpf}
e(E):=\left(\frac{1}{4\pi}\right)^N\text{pf}(F_A).
\end{equation}
The Bianchi identity $\nabla^AF_A=0$ implies that $e(E)$ is closed, i.e. $de(E)=0$.  When $M$ is compact, the cohomology class of the Euler form is independent of the choice of metric and connection, and represents the {\em Euler class} of $E$ which is defined via the Thom class of $E$, \cite{MStCha}.

A complex bundle $E\to M$ equipped with a Hermitian metric is naturally a real oriented bundle of even rank with a Riemannian metric.  Furthermore, if $E$ is holomorphic then the Hermitian metric induces a unique natural Hermitian connection compatible with both the holomorphic structure and the Hermitian metric, known as the {\em Chern} connection, and this is a metric connection with respect to the underlying Riemannian metric on $E$.  In this case, since $\det(iu)=\text{pf}(u^{\br})$, where $u^{\br}$ is the image of $u\in{\bf u}(N/2)$ in ${\bf so}(N)$, then \eqref{eulpf} coincides with the Chern-Weil construction of the top Chern form of $E$ realising $e(E)=c_{N/2}(E)$.

Here we define a canonical Euler form $e(E^\vee_{g,n})$ for the bundle $E^\vee_{g,n}\to\modm_{g,n,\vec{o}}^{\text{spin}}$.  It uses a canonical Hermitian metric on $E^\vee_{g,n}$, defined similarly to the definition of the Weil-Petersson metric.  For a smooth, spin, complete hyperbolic surface $\Sigma=\overline{\Sigma}-D$ with Neveu-Schwarz divisor $D$,  via
Theorem~\ref{caniso1} and Serre duality we have 
\[H^1_{dR}(\Sigma,T_\Sigma^\frac12)\cong H^1(\overline{\Sigma},T_{\overline{\Sigma}}^\frac12(-D))^\vee \cong H^0(\overline{\Sigma},K_{\overline{\Sigma}}^{3/2}(D))\] 
The $3/2$ differentials give the analogue of holomorphic quadratic differentials used to define the Weil-Petersson metric.  Now
\[ \eta,\xi\in H^0(\overline{\Sigma},K_{\overline{\Sigma}}^{3/2}(D))
\]
define a Hermitian metric
\begin{equation}  \label{hermetric} 
\langle\eta,\xi\rangle:=\int_\Sigma\frac{\overline{\eta}\xi}{\sqrt{h}}
\end{equation}
where $h$ is the hyperbolic metric on $\Sigma$.   If $\Sigma$ is compact the integral clearly exists.  When  $\Sigma$ is non-compact, i.e. $D\neq\varnothing$, to see that the integral exists, consider a local coordinate $z$ with $z=0$ corresponding to a point of $D$ and a cusp of the metric.  Locally, the hyperbolic metric is given by $h=\frac{|dz|^2}{|z|^2(\log|z|)^2}$ and the $3/2$ differentials are given by $\eta=\frac{f(z)dz^{3/2}}{z}$ and $\xi=\frac{g(z)dz^{3/2}}{z}$ where $f(z)$ and $g(z)$ are holomorphic at $z=0$.  The local contribution to the metric $\int_{|z|<\epsilon}\frac{\overline{f}g\log|z||dz|^2}{|z|}$ exists since 
\begin{equation}  \label{locest}
\int_{|z|<\epsilon}\frac{|\log|z||}{|z|}|dz|^2=\int_0^\epsilon|\log r| drd\theta=2\pi|\epsilon\log\epsilon-\epsilon|<2\pi \quad\Leftarrow\quad\epsilon<1.
\end{equation}
For $h$ a hyperbolic metric, $\sqrt{h}$ is a metric on the spin bundle $T\Sigma^{1/2}$.  It is worth pointing out that the proof described in \ref{higgscomp2} of the existence of a complete hyperbolic metric in a conformal class due to Hitchin \cite{HitSel} (and more generally for cusped surfaces in \cite{BGGPar}), produces the Hermitian metric on the bundle $T\Sigma^{1/2}$ directly without requiring a square root.

The metric \eqref{hermetric} arises from the super generalisation of the Weil-Petersson Hermitian metric---see for example \cite[eq.(24)]{RabTei}.  The super Weil-Petersson Hermitian metric in local coordinates $(z|\theta)$ uses $(\text{Im}\hspace{.4mm} z+\frac12\theta\bar{\theta})^2$ in place of $(\text{Im}\hspace{.4mm} z)^2$ which appears in the usual Weil-Petersson Hermitian metric since $h=|dz|^2/(\text{Im}\hspace{.4mm} z)^2$ locally.  The expansion of $(\text{Im}\hspace{.4mm} z+\frac12\theta\bar{\theta})^2$ produces the term $\theta\bar{\theta}\text{Im}\hspace{.4mm} z$ which, after integrating out the fermionic directions, corresponds to the factor of $1/\sqrt{h}$ in \eqref{hermetric}, and the term $(\text{Im}\hspace{.4mm} z)^2$ which corresponds to the usual factor of $1/h$ in the Weil-Petersson Hermitian metric.  This appears in \cite{RabTei} in equation (25) in terms of $S=\theta S^0+ S^1$, a function locally representing a quadratic differential plus a  $3/2$ differential, as 
\[\langle S_1,S_2\rangle=\int_{\bh/\Gamma}|dz|^2\left[\overline{S}_1^0S_2^0(\text{Im}\hspace{.4mm} z)^2+\overline{S}_1^1S_2^1(\text{Im}\hspace{.4mm} z)\right]
\]
where the second summand locally represents the Hermitian metric \eqref{hermetric}.

The bundle $E_{g,n}^\vee$ is holomorphic and its complex structure, given by $\xi\mapsto i\xi$ for $\xi\in H^0(\overline{\Sigma},K_{\overline{\Sigma}}^{3/2}(D))$, is compatible with the Hermitian metric on $E_{g,n}^\vee$ constructed above.  This uniquely determines the Chern connection, a metric connection $A$ on $E_{g,n}$ satisfying $\overline{\partial}_A=\overline{\partial}$ the natural operator defining the holomorphic structure on $E_{g,n}^\vee$.  Then $e(E_{g,n}^\vee)$ is defined to be the Pfaffian of the curvature of $A$ via \eqref{eulpf}.

\begin{remark}  \label{pullbackeuler}
The Euler form $e(E_{g,n})$ is defined above for the bundle over the moduli space of complete hyperbolic metrics $E_{g,n}\to\modm_{g,n,\vec{o}}^{\text{spin}}(0,...,0)$.    Using the diffeomorphism $\modm_{g,n,\vec{o}}^{\text{spin}}(L_1,...,L_n)\stackrel{\cong}{\longrightarrow}\modm_{g,n,\vec{o}}^{\text{spin}}(0,...,0)$,  we define the Euler form of $E_{g,n}\to\modm_{g,n,\vec{o}}^{\text{spin}}(L_1,...,L_n)$ to be the pull back of the Euler form $e(E_{g,n})$.   In the formula for the volume $\widehat{V}^{WP}_{g,n}(L_1,...,L_n)=\int_{\modm_{g,n,\vec{o}}^{\text{spin}}(L_1,...,L_n)}e(E_{g,n}^\vee)\exp\omega^{WP}$ defined in \eqref{supvol}, we can consider the entire integral via its pull-back to $\modm_{g,n,\vec{o}}^{\text{spin}}(0,...,0)$, and we see that the Euler  form does not change while the pull-back of $\omega^{WP}$ depends explicitly on $L_i$ following Mirzakhani's symplectic reduction argument in \cite{MirWei}.
\end{remark}

\subsubsection{}\label{eulerform}
In the following theorem we prove that the Euler form $e(E_{g,n}^\vee)$ defined in Section~\ref{sec:eulerform} extends to the compactification $\overline{\modm}_{g,n,\vec{o}}^{\text{spin}}$ and  defines a cohomology class in $H^*(\overline{\modm}_{g,n,\vec{o}}^{\text{spin}},\br)$.  We do this by proving that the Hermitian metric that defines $e(E_{g,n}^\vee)$ extends smoothly from $E_{g,n}$ to its extension $\widehat{E}_{g,n}\to\overline{\modm}_{g,n,\vec{o}}^{\text{spin}}$.  This enables us to conclude that the cohomology class defined by the extension of $e(E_{g,n}^\vee)$ coincides with the Euler class of $\widehat{E}_{g,n}^\vee$.

\begin{theorem}  \label{eulerfc}
The extension of the Euler form $e(E_{g,n}^\vee)$ to $\overline{\modm}_{g,n,\vec{o}}^{\text{spin}}$ defines a cohomology class which coincides with the Euler class $e(\widehat{E}_{g,n}^\vee)\in H^*(\overline{\modm}_{g,n,\vec{o}}^{\text{spin}},\br)$ of the extension bundle $E_{g,n}^\vee$.
\end{theorem}
\begin{proof}
The Hermitian metric \eqref{hermetric} on $E_{g,n}^\vee$ extends to a Hermitian metric on the bundle $\widehat{E}_{g,n}^\vee\to\overline{\modm}_{g,n,\vec{o}}^{\text{spin}}$ due to behaviour of the poles of the $3/2$ differentials representing fibres of $E_{g,n}^\vee$ as follows.
An element of $\overline{\modm}_{g,n,\vec{o}}^{\text{spin}}$ is a pair $(\cc,\theta)$ consisting of a line bundle $\theta$ over a stable twisted curve $\cc$ and an isomorphism $\theta^2\cong\omega_{\cc}^{\text{log}}$.  Labeled points $p_j$ are orbifold points with isotropy subgroup $\bz_2$ and $\theta$ is an orbifold bundle which defines a representation $\bz_2\to\bz_2$ at each $p_i$.  When $\cc$ is a nodal curve, the nodes also have isotropy subgroup $\bz_2$ and again $\theta$ defines a representation $\bz_2\to\bz_2$ at each node.  The pull-back of $\theta$ to the normalisation of $\cc$ is an orbifold bundle on each component.  In particular, points in the fibre of $\widehat{E}_{g,n}^\vee$ given by elements of $H^0(\overline{\Sigma},K_{\overline{\Sigma}}^{3/2}(D))$ have the same simple pole behaviour at nodes and at labeled points.  The pole at a node is present if the behaviour at the node is Neveu-Schwarz and removable if the behaviour at the node is Ramond.  Thus the estimate \eqref{locest} applies also at nodes to prove that the Hermitian metric on $H^0(\overline{\Sigma},K_{\overline{\Sigma}}^{3/2}(D))$ is well-defined when $\Sigma$ is nodal.  
The conclusion is that the Hermitian metric on $E_{g,n}^\vee$ extends to a Hermitian metric on $\widehat{E}_{g,n}^\vee$.  Furthermore, it extends to a smooth Hermitian metric on $\widehat{E}_{g,n}^\vee$ because the hyperbolic metric $h$ varies smoothly outside of nodes and has a canonical form around nodes, and the Hermitian metric is defined via an integral over $1/\sqrt{h}$ times smooth sections.  

We conclude that the Euler form $e(E_{g,n}^\vee)$, constructed from the curvature of the natural metric connection $A$, which is determined uniquely from the Hermitian metric and the holomorphic structure on $E_{g,n}^\vee$, extends to $\overline{\modm}_{g,n,\vec{o}}^{\text{spin}}$.  The Euler class of $\widehat{E}_{g,n}^\vee$ is determined by a choice of any connection on $\widehat{E}_{g,n}^\vee$, so we choose the metric connection of the extension of the Hermitian metric on $\widehat{E}_{g,n}^\vee$, to conclude that the cohomology class defined by the extension of $e(E_{g,n}^\vee)$ coincides with the Euler class $e(\widehat{E}_{g,n}^\vee)\in H^*(\overline{\modm}_{g,n,\vec{o}}^{\text{spin}},\br)$.
\end{proof}
\begin{remark}   
The Weil-Petersson form is the imaginary part of the natural Hermitian metric on the (co)tangent bundle over $\modm_{g,n}$ defined by
\begin{equation} \label{WPmetric}
\langle\eta,\xi\rangle:=\int_\Sigma\frac{\overline{\eta}\xi}{h},\qquad\eta,\xi\in H^0(\overline{\Sigma},K_{\overline{\Sigma}}^2(D))\cong H^1(\overline{\Sigma},T_{\overline{\Sigma}}(-D))^\vee.
\end{equation}
for
\[ \eta,\xi\in H^0(\overline{\Sigma},K_{\overline{\Sigma}}^2(D))\cong H^1(\overline{\Sigma},T_{\overline{\Sigma}}(-D))^\vee.\]
This Hermitian metric does not extend to $\overline{\modm}_{g,n}$ since it blows up as a cusp forms in a family of hyperbolic metrics.  This contrasts with the behaviour of the Hermitian metric defined on $E_{g,n}^\vee$ which does extend to $\overline{\modm}_{g,n,\vec{o}}^{\text{spin}}$.

The explanation for the difference in behaviour lies in the singularities of a meromorphic quadratic differential $\eta\in H^0(\overline{\Sigma},K_{\overline{\Sigma}}^2(D))$---it has simple poles near labeled points and double poles near nodes.  This is explained as follows. Locally, a holomorphic quadratic differential is the tensor square of a holomorphic differential.  As a node forms in a family of curves, a holomorphic differential gains simple poles on each side, with residues summing to zero.  This can be seen by considering the relative dualising sheaf of a family that deforms a nodal curve.  Thus, as a node forms in a family of curves, a holomorphic quadratic differential gains double poles on each side, with equal biresidues.  The condition of simple poles at labeled points is a consequence of the local deformation theory of a curve containing a labeled point which leads to elements of $H^1(\overline{\Sigma},T_{\overline{\Sigma}}(-D))$.

In a local coordinate $z$ near a labeled point, the hyperbolic metric is given by $h=\frac{|dz|^2}{|z|^2(\log|z|)^2}$, the quadratic differentials are $\eta=\frac{f(z)dz^{2}}{z}$ and $\xi=\frac{g(z)dz^{2}}{z}$ for $f(z)$ and $g(z)$ holomorphic at $z=0$, and the analogue of \eqref{locest} giving the local contribution to the metric becomes
\[
\int_{|z|<\epsilon}(\log|z|)^2|dz|^2=\int_0^\epsilon(\log r)^2 rdrd\theta<\infty
\]
which prove that the Weil-Petersson metric is well-defined.
Whereas, near a node $\eta=\frac{f(z)dz^{2}}{z^2}$ and $\xi=\frac{g(z)dz^{2}}{z^2}$, so the local contribution to the metric diverges:
\[
\int_{|z|<\epsilon}\frac{(\log|z|)^2}{|z|^2}|dz|^2=\int_0^\epsilon\frac{(\log r)^2}{r}drd\theta=\infty
\]
showing that the Weil-Petersson metric does not extend to $\overline{\modm}_{g,n}$.
In contrast, the proof of Theorem~\ref{eulerfc} shows that the Hermitian metric on $E_{g,n}^\vee$ does extend to $\overline{\modm}_{g,n,\vec{o}}^{\text{spin}}$ which relies on the fact that the order of the pole of an element of $H^0(\overline{\Sigma},K_{\overline{\Sigma}}^{3/2}(D))$ is simple both at a labeled point and at a node. 

The different behaviour is reflected quite simply via the calculation of dimensions of $H^0(\overline{\Sigma},K_{\overline{\Sigma}}^{3/2}(D))$ and $H^0(\overline{\Sigma},K_{\overline{\Sigma}}^2(D))$ on a stable curve.  For simplicity, consider the case of an irreducible genus $g$ curve $\overline{\Sigma}$ with exactly one node:
\[\dim H^0(\overline{\Sigma},K_{\overline{\Sigma}}^{2}(D))=3g-3+n=3(g-1)-3+n+4-1
\]
where the right hand side is calculated on the normalisation of $\overline{\Sigma}$ using simple poles on labeled points and double poles at the two extra points minus the one condition of a common biresidue.  In contrast,
\[\dim H^0(\overline{\Sigma},K_{\overline{\Sigma}}^{3/2}(D))=2g-2+n=2(g-1)-2+n+2
\]
where the right hand side is calculated on the normalisation of $\overline{\Sigma}$ using simple poles on labeled points and at the two extra points.  (The calculation above shows the case of Neveu-Schwarz nodal points.  For Ramond nodal points, the section is holomorphic at the two extra points.) 
\end{remark}

\subsubsection{} 

\begin{proof}[{\bf Proof of Theorem~\ref{volequal}.}]
We must show that
\[\widehat{V}^{WP}_{g,n}(L_1,...,L_n)=2^{1-g-n}V^{\Theta}_{g,n}(L_1,...,L_n)
\]
where $V^{\Theta}_{g,n}(L_1,...,L_n)=\int_{\overline{\modm}_{g,n}}\hspace{-2mm}\Theta_{g,n}\exp\left\{2\pi^2\kappa_1\hspace{-.6mm}+\hspace{-.6mm}\frac12\sum_{i=1}^n L_i^2\psi_i\right\}$
and $\widehat{V}^{WP}_{g,n}(L_1,...,L_n)$ has the following equivalent expressions:
\begin{align*}
\widehat{V}^{WP}_{g,n}(L_1,...,L_n)&=\int_{\modm_{g,n,\vec{o}}^{\text{spin}}(L_1,...,L_n)}e(E_{g,n}^\vee)\exp(\omega^{\text{WP}}(L_1,...,L_n))\\
&=\int_{\modm_{g,n,\vec{o}}^{\text{spin}}(0,...,0)}e(E_{g,n}^\vee)\exp(f^*\omega^{\text{WP}}(L_1,...,L_n))
\\
&=\int_{\overline{\modm}_{g,n,\vec{o}}^{\text{spin}}}e(\widehat{E}_{g,n}^\vee)\exp(2\pi^2\kappa_1+\frac12\sum_{i=1}^n L_i^2\psi_i)
\end{align*}
where the first equality is the definition \eqref{supvol}.  The second equality uses the pull-back of the diffeomorphism $f:\modm_{g,n,\vec{o}}^{\text{spin}}(L_1,...,L_n)\longrightarrow\modm_{g,n,\vec{o}}^{\text{spin}}(0,...,0)$ where as discussed in Remark~\ref{pullbackeuler} the Euler form pulls back to the canonical Euler form.  The third equality uses the extension of $e(E_{g,n})$ to the compactification proven in Theorem~\ref{eulerfc} together with Mirzakhani's expression for the pull-back of the Weil-Petersson form, proven in \cite{MirWei} via symplectic reduction.  Thus $\widehat{V}^{WP}_{g,n}(L_1,...,L_n)$ can be calculated cohomologically over the moduli space of stable curves $\overline{\modm}_{g,n,\vec{o}}^{\text{spin}}$ using the Euler class $e(\widehat{E}_{g,n})\in H^*(\overline{\modm}_{g,n,\vec{o}}^{\text{spin}},\br)$.  The push-forward of this cohomological calculation under the forgetful map $\overline{\modm}_{g,n,\vec{o}}^{\text{spin}}\stackrel{p}{\longrightarrow}\overline{\modm}_{g,n}$ leads to the relation
\begin{align*}
\widehat{V}^{WP}_{g,n}(L_1,...,L_n)&=\int_{\overline{\modm}_{g,n}}\hspace{-2mm}p_*e(\widehat{E}_{g,n}^\vee)\exp\left\{2\pi^2\kappa_1+\frac12\sum_{i=1}^n L_i^2\psi_i\right\}\\
&=2^{1-g-n}\int_{\overline{\modm}_{g,n}}\hspace{-2mm}\Theta_{g,n}\exp\left\{2\pi^2\kappa_1+\frac12\sum_{i=1}^n L_i^2\psi_i\right\}
\end{align*}
where the first equality uses the fact the classes $\kappa_1$ and $\psi_i$ pull back from  $\overline{\modm}_{g,n}$ to $\overline{\modm}_{g,n,\vec{o}}^{\text{spin}}$ (reflecting the fact that the Weil-Petersson form pulls back from the smooth moduli space $\modm_{g,n}(L_1,...,L_n)$ to $\modm_{g,n,\vec{o}}^{\text{spin}}(L_1,...,L_n)$) and the second equality uses 
\[
\Theta_{g,n}=2^{g-1+n}p_*e(\widehat{E}_{g,n}^\vee)=(-1)^n2^{g-1+n}p_*e(\widehat{E}_{g,n})
\] 
from Definition~\ref{theta} in Section~\ref{sec:theta}.
\end{proof}

\section{Moduli space of super hyperbolic surfaces}   \label{sec:mirz}   
In this section we describe Mirzakhani's recursion relations between volumes of moduli spaces of hyperbolic surfaces \cite{MirSim} and the generalisation of Mirzakhani's argument by Stanford and Witten \cite{SWiJTG} who derive the recursion \eqref{volrecWP} via the volumes of moduli spaces of super hyperbolic surfaces.   We also describe Mirzakhani's proof of the Kontsevich-Witten theorem since the proof of Theorem~\ref{main} follows Mirzakhani's arguments closely.  

\subsection{Moduli space of hyperbolic surfaces}
Define the moduli space of complete oriented hyperbolic surfaces 
\[
\modm_{g,n}(\vec{0})=\{\Sigma\mid \Sigma= \text{genus }g \text{ oriented hyperbolic surface with } n\text{ labeled cusps}\}/\sim\] 
where the quotient is by isometries preserving each cusp.  Note that (generically) a hyperbolic surface appears twice in $\modm_{g,n}(\vec{0})$ equipped with each of its two orientations.  
Define the moduli space of oriented hyperbolic surfaces with fixed length $\vec{L}=(L_1,...,L_n)\in\br^n_{\geq 0}$ geodesic boundary components by
\begin{align*}
\modm_{g,n}(\vec{L})=\Big\{(\Sigma,\beta_1,...,\beta_n)\mid \Sigma& \text{ genus }g\text{ oriented hyperbolic surface},\\
&\partial \Sigma=\sqcup\beta_i \text{ are geodesic}, L_i=\ell(\beta_i)\Big\}/\sim
\end{align*} 
where again the quotient is by isometries preserving each $\beta_i$.  Any non-trivial isometry must rotate each $\beta_i$ non-trivially.  The moduli spaces are all diffeomorphic $\modm_{g,n}(\vec{0})\cong\modm_{g,n}(\vec{L})$ and we will see below that the varying parameters $\vec{L}\in\br^n_{\geq 0}$ give a family of deformations of a natural symplectic structure on $\modm_{g,n}(\vec{0})$.

\subsubsection{}
The hyperbolic metric on $\Sigma$ induces a Hermitian metric on the vector space of meromorphic quadratic differentials $H^0(\overline{\Sigma},K_{\overline{\Sigma}}^2(D))$ via \eqref{WPmetric}, hence a Hermitian metric on $T_{[\Sigma]}\modm_{g,n}(\vec{0})$ known as the Weil-Petersson metric.  The Weil-Petersson symplectic form $\omega^{WP}$ on $\modm_{g,n}(\vec{0})$ is the imaginary part of the Weil-Petersson metric.  It defines a volume form on $\modm_{g,n}(\vec{0})$ with finite integral known as the Weil-Petersson volume of $\modm_{g,n}(\vec{0})$: 
\[V_{g,n}^{WP}:=\int_{\modm_{g,n}(\vec{0})}\exp\left\{\omega^{WP}\right\}.\]

\subsubsection{}
Teichm\"uller space gives a way to realise $\omega^{WP}$ via local coordinates on $\modm_{g,n}(\vec{0})$.  Fix a smooth genus $g$ oriented surface $\Sigma_{g,n}=\overline{\Sigma}_{g,n}-\{q_1,...,q_n\}$.  A {\em marking} of a genus $g$ hyperbolic surface $\Sigma=\overline{\Sigma}-\{p_1,...,p_n\}$ is an orientation preserving homeomorphism $f:\Sigma_{g,n}\stackrel{\cong}\to \Sigma$.  Define the Teichm\"uller space of marked hyperbolic surfaces $(\Sigma,f)$ of type $(g,n)$ to be
\[\ct_{g,n}=\{(\Sigma,f)\}/\sim\]
where the equivalence is given by $(\Sigma,f)\sim(T,g)$ if $g\circ f^{-1}:\Sigma\to T$ is isotopic to an isometry.
The mapping class group $\text{Mod}_{g,n}$ of isotopy classes of orientation preserving diffeomorphisms of the surface that preserve boundary components acts on $\ct_{g,n}$ by its action on markings.  The quotient of Teichm\"uller space by this action produces the moduli space
\[\modm_{g,n}(\vec{0})=\ct_{g,n}/\text{Mod}_{g,n}.\]

\subsubsection{}
Global coordinates for Teichm\"uller space, known as Fenchel-Nielsen coordinates, are defined as follows.  Choose a maximal set of disjoint embedded isotopically inequivalent simple closed curves on the topological surface $\Sigma_{g,n}$.  The complement of this collection is a union of pairs of pants known as a pants decomposition of the surface $\Sigma_{g,n}$.  Each pair of pants contributes Euler characteristic  $-1$, so there are $2g-2+n = -\chi(\Sigma)$ pairs of pants in the decomposition, and hence $3g - 3 + n$  closed geodesics (not counting the boundary classes.) 
A marking $f : \Sigma_{g,n}\to\Sigma$ of a hyperbolic surface with $n$ cusps $\Sigma$ induces a pants decomposition on $\Sigma$ from $\Sigma_{g,n}$. The isotopy classes of embedded closed curves can be represented by a collection $\{\gamma_1, ..., \gamma_{3g-3+n}\}$ of disjoint embedded simple closed geodesics which cuts $\Sigma$ into hyperbolic pairs of pants with geodesic and cusp boundary components.  Their lengths $\ell_1, ..., \ell_{3g-3+n}$ give half the Fenchel-Nielsen coordinates, and the other half are the twist parameters $\theta_1, ..., \theta_{3g-3+n}$ which we now define.  Any hyperbolic pair of pants contains three geodesic arcs giving the shortest paths between boundary components, or horocycles around cusps.   The simple closed geodesic $\gamma_i$ intersects the geodesic arcs on the pair of pants on one side of $\gamma_i$ at a pair of (metrically opposite) points on $\gamma_i$, and similarly $\gamma_i$ intersects the geodesic arcs on the pair of pants on the other side of $\gamma_i$ at a pair of (metrically opposite) points on $\gamma_i$.  The oriented distance between these points lies in $[0,\ell_i/2]$ and after a choice that fixes the ambiguity arising from choosing one out of a pair of points the oriented distance lies in $[0,\ell_i]$ which defines $\theta_i(\text{mod\ }\ell_i)$.  A further lift $\theta_i\in\br$ is obtained by continuous paths in $\ct_{g,n}$ which amount to rotations around $\gamma_i$.  The coordinates $(\ell_j,\theta_j)$ for $j = 1, 2, ..., 3g -3+n$ give rise to an isomorphism 
\[\ct_{g,n}\cong(\br^+\times\br)^{3g-3+n}.\]

\subsubsection{}
The Fenchel-Nielsen decomposition induces an action of $S^1$ along each simple closed geodesic $\gamma_i$ by rotation.  In local coordinates $\theta_i\mapsto\theta_i+\phi$ for $\phi\in\br/\ell_i\bz\cong S^1$.  This action defines a vector field, given locally by $\partial/\partial\theta_i$.  Wolpert proved that $\partial/\partial\theta_i$ is a Hamiltonian vector field with respect to $\omega^{WP}$ with Hamiltonian given by $\ell_i$.  In other words $(\ell_1,...,\ell_{3g-3+n},\theta_1,...,\theta_{3g-3+n})$ are Darboux coordinates for $\omega^{WP}$.  This is summarised in the following theorem.
\begin{thm}[Wolpert \cite{WolWei}]  \label{woldar}
\begin{equation}  \label{WP}
\omega^{WP}=\sum d\ell_j\wedge d\theta_j.
\end{equation}
\end{thm}
Since $\omega^{WP}$ is defined over $\modm_{g,n}(\vec{0})$ it follows that this expression for $\omega^{WP}$ is invariant under the action of the mapping class group $\text{Mod}_{g,n}$.  There are a finite number of pants decompositions up to the action of the mapping class group, each class consisting of infinitely many geometrically different types. Thus once a topological pants decomposition of the surface is chosen a given hyperbolic surface has infinitely many geometrically different pants decompositions equivalent under $\text{Mod}_{g,n}$. Each different decomposition associates different lengths and twist parameters, hence different coordinates, to the same hyperbolic surface.

Wolpert proved that the Weil-Petersson symplectic form $\omega^{WP}$ extends from $\modm_g$ to $\overline{\modm}_{g}$ and coincides with $2\pi^2\kappa_1$ defined in \eqref{kappaclass}.  His proof extends to $\modm_{g,n}(\vec{0})$ and importantly gives
\[V_{g,n}^{WP}=\int_{\modm_{g,n}}\exp\left\{\omega^{WP}\right\}
=\int_{\overline{\modm}_{g,n}}\exp\left\{2\pi^2\kappa_1\right\}.\]

\subsubsection{}
Wolpert's local formula \eqref{WP} generalises below in \eqref{WPLi} to define a  symplectic form $\omega^{WP}(\vec{L})$ on $\modm_{g,n}(\vec{L}))$ which pulls back under the isomorphism
\[\modm_{g,n}(\vec{0})\cong\modm_{g,n}(\vec{L})
\]
to define a family of deformations of the Weil-Petersson symplectic form, depending on the parameters $\vec{L}=(L_1,...,L_n)$.  The pairs of pants  decomposition of an oriented hyperbolic surface with cusps naturally generalises to an oriented hyperbolic surface with geodesic boundary components.  The lengths and twist parameters of the $3g-3+n$ interior geodesics gives rise to Fenchel-Nielsen coordinates $(\ell_1,...,\ell_{3g-3+n},\theta_1,...,\theta_{3g-3+n})$ on the Teichm\"uller space 
\[\ct_{g,n}(\vec{L})=\{(\Sigma,f)\}/\sim\]
of marked genus $g$ oriented hyperbolic surfaces with geodesic boundary components of lengths $\vec{L}=(L_1,...,L_n)\in\br_{\geq0}^n$ and an isomorphsim
$\ct_{g,n}(\vec{L})\cong(\br^+\times\br)^{3g-3+n}.$  
Wolpert's local formula \eqref{WP} can be used to define a symplectic form
\begin{equation} \label{WPLi}
\omega^{WP}(\vec{L})=\sum d\ell_j\wedge d\theta_j
\end{equation}
again known as the Weil-Petersson symplectic form, on $\ct_{g,n}(\vec{L})$.  It is invariant under the mapping class group and descends to the moduli space 
\[\modm_{g,n}((\vec{L})=\ct_{g,n}(\vec{L})/\text{Mod}_{g,n}.\]
Wolpert's result \cite{WolHom} generalises to show that $\omega^{WP}(\vec{L})$ extends to $\overline{\modm}_{g,n}$.

Mirzakhani \cite{MirWei} proved that $\modm_{g,n}(\vec{L})$ arises as a symplectic quotient of a symplectic manifold with $T^n$ action and moment map $(\frac12 L_1^2,...,\frac12 L_n^2)$.  Each level set of the moment map or equivalently each choice of $\vec{L}=(L_1,...,L_n)$ gives a symplectic quotient.  Quite generally, the symplectic form on the quotient is a deformation by first Chern classes of line bundles related to the $T^n$ action.  In this case it is $\omega^{WP}+\sum\frac12 L_i^2\psi_i$ where $\psi_i=c_1(\cl_i)\in H^2(\overline{\modm}_{g,n})$ are defined in \ref{psiclass} which produces:
\begin{equation}  \label{WPvolcomp}
V_{g,n}^{WP}(\vec{L})=\int_{\modm_{g,n}(\vec{L})}\exp\left\{\omega^{WP}(\vec{L})\right\}
=\int_{\overline{\modm}_{g,n}}\exp\left\{2\pi^2\kappa_1+\frac12\sum_{i=1}^n L_i^2\psi_i\right\}.
\end{equation}
The extension of $\omega^{WP}(\vec{L})$ to $\overline{\modm}_{g,n}$ uses Wolpert's theorem together with the extensions of the classes $\psi_i$ from  $\modm_{g,n}(\vec{L})$ to $\overline{\modm}_{g,n}$.
In particular the volumes depend non-trivially on $L_i$ proving that  $\omega^{WP}(\vec{L})$ is a non-trivial deformation of $\omega^{WP}$.

\subsection{Mirzakhani's volume recursion}  \label{sec:WPvolrec}
Mirzakhani proved the following recursion relations between the volumes $V_{g,n}^{WP}(L_1,...,L_n)$.
\begin{thm}[Mirzakhani \cite{MirSim}] \label{th:mirzvolrec}
\begin{align}  \label{mirzvolrec}
L_1V_{g,n}^{WP}(L_1,...,L_n)&=\frac12\int_0^\infty\int_0^\infty xyD^M(L_1,x,y)P_{g,n+1}(x,y,L_2,..,L_n)dxdy\\
&+\sum_{j=2}^n\int_0^\infty xR^M(L_1,L_j,x)V^{WP}_{g,n-1}(x,L_2,..,\hat{L}_j,..,L_n)dx
\nonumber
\end{align}
where $\displaystyle P_{g,n+1}(x,y,L_K)=V^{WP}_{g-1,n+1}(x,y,L_K)+\hspace{-3mm}\mathop{\sum_{g_1+g_2=g}}_{I \sqcup J = K}\hspace{-2mm} V^{WP}_{g_1,|I|+1}(x,L_I)V^{WP}_{g_2,|J|+1}(y,L_J)$\\
for $K=\{2,...,n\}$.
\end{thm}
The kernels in \eqref{mirzvolrec} are defined by
\[H^M(x,y)=1-\tfrac{1}{2}\tanh\tfrac{x-y}{4}-\tfrac{1}{2}\tanh\tfrac{x+y}{4}
\]
which uniquely determine $D^M(x,y,z)$ and $R^M(x,y,z)$ via
\[\frac{\partial}{\partial x}D^M(x,y,z)=H^M(x,y+z),\ \frac{\partial}{\partial x}R^M(x,y,z)=\hspace{-.5mm}\frac12\hspace{-.5mm}\left(H^M(z,x+y)+H^M(z,x-y)\right)
\]
and the initial conditions $D^M(0,y,z)=0=R^M(0,y,z)$.  Explicitly
\begin{equation}  \label{Rlog}
R^M(x,y,z)=x-\log\left(\frac{\cosh\frac{y}{2}+\cosh\frac{x+z}{2}}{\cosh\frac{y}{2}+\cosh\frac{x-z}{2}}\right)
\end{equation}
and $D^M(x,y,z)$ is given by the relation
\begin{equation} \label{DRid}
D^M(x,y,z)=R^M(x,y,z)+R^M(x,z,y)-x
\end{equation}
which follows from
\begin{equation} \label{Hid}
2H^M(x,y+z)=H^M(z,x+y)+H^M(z,x-y)+H^M(y,x+z)+H^M(y,x-z)-2.
\end{equation}

The relations \eqref{mirzvolrec} uniquely determine $V_{g,n}^{WP}(L_1,...,L_n)$ from 
\[V^{WP}_{0,3}=1,\quad V^{WP}_{1,1}=\frac{1}{48}(4\pi^2+L^2).\]
The first two calculations are 
\[V^{WP}_{0,4}=\frac12(4\pi^2+\sum L_i^2),\qquad V^{WP}_{1,2}=\frac{1}{384}(4\pi^2+\sum L_i^2)(12\pi^2+\sum L_i^2).\]

Mirzakhani used the recursion \eqref{mirzvolrec} to prove that the top coefficients of the polynomial $V_{g,n}^{WP}(L_1,...,L_n)$ satisfy Virasoro constraints which proves Theorem~\ref{KW} of Witten-Kontsevich.  See the Proof of Theorem~\ref{KWthm} in Section~\ref{sec:kdv}.

 The proof of Theorem~\ref{th:mirzvolrec} uses an unfolding of the volume integral to an integral over associated moduli spaces.  This allows the integral to be related to volumes over simpler moduli spaces.  A non-trivial decomposition of the constant function on the moduli space is used to achieve the unfolding.  This is explained in this section, particularly because the same ideas are required in the super moduli space case.

\subsubsection{}   \label{Rder}
The functions $D^M(x,y,z)$, $R^M(x,y,z)$ and the identity \eqref{DRid} have the following geometric interpretation.  Given $x>0,y>0,z>0$ there exists a unique hyperbolic pair of pants with geodesic boundary components $\beta_1$, $\beta_2$ and $\beta_3$ of respective lengths $x$, $y$ and $z$.
\begin{center}
{\includegraphics[scale=0.1]{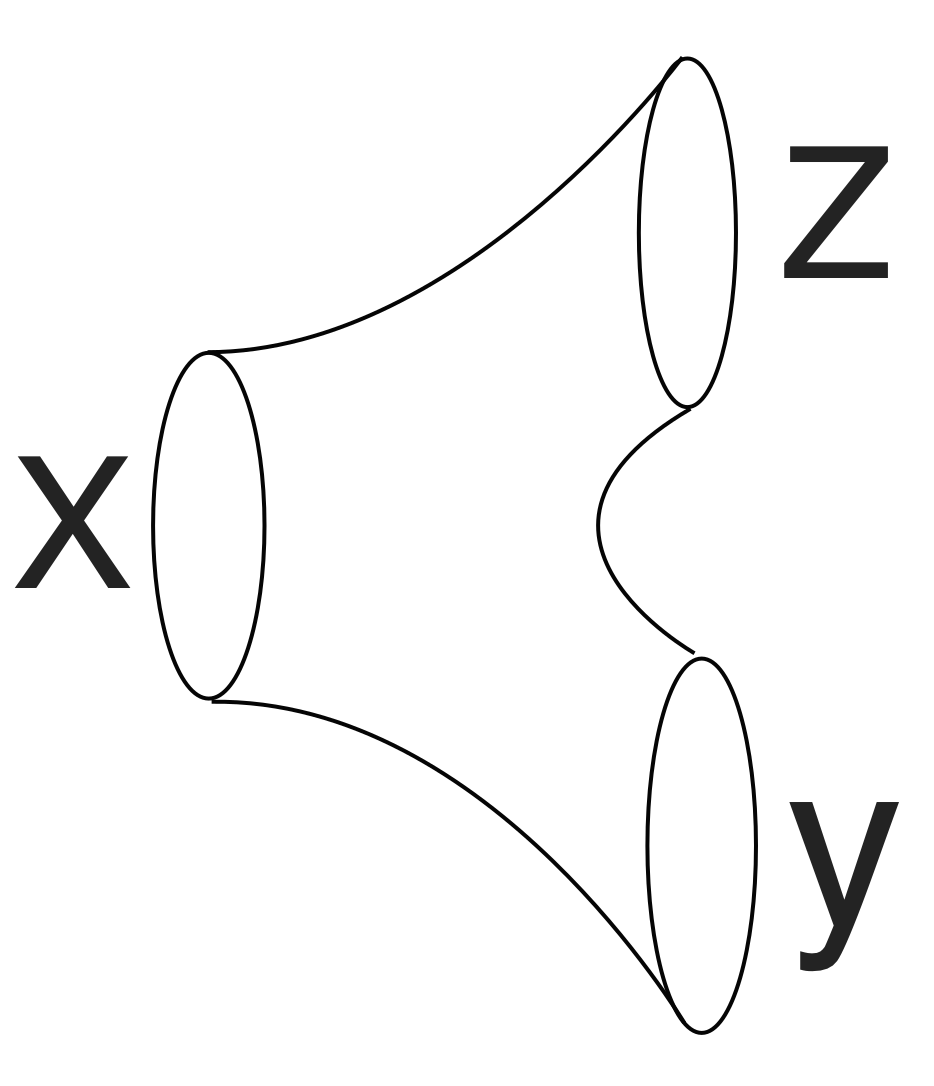}}
\end{center}
Consider geodesics orthogonal to the boundary component $\beta_1$.  Travel along any such geodesic beginning at $\beta_1$ and stop if the geodesic meets itself or a boundary component.  Such geodesics have four types of behaviour and their initial points partition $\beta_1=I_1\sqcup I_2\sqcup I_3\sqcup I_4$.

(i) The geodesic meets itself, or $\beta_1$ for a second time;

(ii) the geodesic meets $\beta_2$;

(iii) the geodesic meets $\beta_3$;

(iv) the geodesic remains embedded for all time.\\
The initial points of geodesics of types (i), (ii), (iii) and (iv) lie in $I_1\subset\beta_1$, respectively $I_2\subset\beta_1$,  respectively $I_3\subset\beta_1$, respectively $I_4\subset\beta_1$.  The subset $I_1$ is a disjoint union of two open intervals while each of $I_2$ and $I_3$ is a single open interval.  The subset $I_4$ given by initial points of geodesics of types (iv) consist of the four points given by the intersection of the closures of $I_1$, $I_2$ and $I_3$.

The kernels $D^M(x,y,z)$ and $R^M(x,y,z)$ arise from this partition of $\beta_1$.
We have $D^M(x,y,z)=\ell(I_1)$ where $\ell(I_1)$ is the length of  $I_1$ using the hyperbolic metric, and $R^M(x,y,z)=\ell(I_1\cup I_2)$.  Hence $R^M(x,z,y)=\ell(I_1\cup I_3)$ so in particular 
\[R^M(x,y,z)+R^M(x,z,y)=\ell(I_1)+\ell(I_2)+\ell(I_1)+\ell(I_3)=\ell(I_1)+x=D^M(x,y,z)+x\] 
which is \eqref{DRid}.

\subsubsection{}  \label{mcshaneid}
Mirzakhani \cite{MirSim} proved the following non-trivial sum of functions of lengths of geodesics on a hyperbolic surface, known as a McShane identity because it generalises an identity of McShane \cite{McSRem}.  Given a hyperbolic surface $\Sigma$ with $n$ geodesic boundary components $\beta_1,...,\beta_n$, define $\cp_{i}$, respectively $\cp_{ij}$, to be the set of isometric embeddings $P\to \Sigma$ of hyperbolic pairs of pants with geodesic boundary, which meet the boundary of $\Sigma$ precisely at $\beta_i$, respectively at $\beta_i$ and $\beta_j$.  Denote by $\ell_{\partial_iP}$ the length of the $i$th geodesic boundary component of $P$.  Define $R^M(P)=R^M(\ell_{\partial_1P}=L_1,\ell_{\partial_2P}=L_j,\ell_{\partial_3P})$ for $R^M$ defined in \eqref{Rlog}, and $D^M(P)=D^M(\ell_{\partial_1P}=L_1,\ell_{\partial_2P},\ell_{\partial_3P})$ for $D^M$ defined in \eqref{DRid}.
\begin{thm}[Mirzakhani \cite{MirSim}]  \label{mirmcsh}
Given a genus $g$ hyperbolic surface $\Sigma$ with $n$ geodesic boundary components $\beta_1$,..., $\beta_n$ of lengths $L_1,...,L_n$ we have:
\begin{equation}   \label{mcshane}
L_1 =\sum_{P\in\cp_1} D^M(P)+\sum_{j=2}^n \sum_{P\in\cp_{1j}} R^M(P).
\end{equation}
\end{thm}

The proof of Theorem~\ref{mirmcsh} partitions $\beta_1$ into a countable collection of disjoint interval associated to embedded pairs of pants $P\subset \Sigma$, together with a measure zero subset, using geodesics perpendicular to $\beta_1$.  The length of each interval is determined by a pair of pants, as in \ref{mcshaneid}.  The identity \eqref{mcshane} sums these lengths to get $L_1=\ell(\beta_1)$.

The sum over pairs of pants is topological, so it depends only on the topology of $\Sigma$, since an isometrically embedded pair of pants in $\Sigma$ is uniquely determined by a topological embedding of a pair of pants into $\Sigma$.  The left hand side of \eqref{mcshane} is independent of the hyperbolic metric on $\Sigma$, whereas each summand on the right hand side dependends on the hyperbolic metric of $\Sigma$.  The importance of \eqref{mcshane} is that it allows one to integrate the constant function $L_1$ over the moduli space.

\subsubsection{}

Mirzakhani used the identity \eqref{mcshane} to integrate functions of a particular form over the moduli space \cite{MirSim}.  Applied to the constant function, this yields the volume of the moduli space. Given a closed curve $\gamma_0\subset\Sigma_{g,n}$ in a topological surface surface $\Sigma_{g,n}$, its mapping class group orbit $\text{Mod}_{g,n}\cdot\gamma_0$ gives a well-defined collection of closed geodesics in any hyperbolic surface  $\Sigma\in\modm_{g,n}(\vec{L})$.  Define a function over $\modm_{g,n}(\vec{L})$ of the form 
\[F(\Sigma)=\hspace{-2mm}\sum_{\gamma\in \text{Mod}_{g,n}\cdot\gamma_0} \hspace{-2mm}f(l^\Sigma_{\gamma})\] 
where $f$ is an arbitrary function and the length of the geodesic $l^\Sigma_{\gamma}$ shows the dependence on the hyperbolic surface $\Sigma\in\modm_{g,n}(\vec{L})$.  When $f$ decays fast enough the sum is well-defined on the moduli space.  More generally,
one can consider an arbitrary (decaying) function on collections of geodesics and sum over orbits of the mapping class group acting on the collection.   Mirzakhani unfolded the integral of $F$ to an integral over a moduli space $\widetilde{\modm}_{g,n}(\vec{L})$ of pairs $(\Sigma,\gamma)$ consisting of a hyperbolic surface $\Sigma$ and a collection of geodesics $\gamma\subset \Sigma$.  
\[\begin{array}{c}
\ct_{g,n}(\vec{L})\\\downarrow\\\widetilde{\modm}_{g,n}(\vec{L})\\\downarrow\\\modm_{g,n}(\vec{L})
\end{array}\]
The unfolded integral 
\[\int_{\modm_{g,n}(\vec{L})}F\cdot d{\rm vol}=\int_{\widetilde{\modm}_{g,n}(\vec{L})}f(l_{\gamma})\cdot d{\rm vol}\]
can be expressed in terms of an integral over the simpler moduli space obtained by cutting $\Sigma$ along the geodesic $\gamma$.  

The identity \eqref{mcshane} is exactly of the right form for Mirzakhani's scheme since it expresses the constant function $F=L_1$ as a sum of functions of lengths over orbits of the mapping class group.  In this case,
\[ L_1V_{g,n}^{WP}(\vec{L})=\int_{\modm_{g,n}(\vec{L})}F\cdot d{\rm vol}=\int_{\widetilde{\modm}_{g,n}(\vec{L})}f(l_{\gamma_1},l_{\gamma_2})\cdot d{\rm vol}\]
expresses the volume $V_{g,n}^{WP}(\vec{L})$ recursively in terms of the simpler volumes $V_{g',n'}^{WP}(\vec{L}')$ where $2g'-2+n'<2g-2+n$ which gives Theorem~\ref{th:mirzvolrec}.  

The polynomiality of $V_{g,n}^{WP}(\vec{L})$ is immediate from its identification with intersection numbers on $\overline{\modm}_{g,n}$ via \eqref{WPvolcomp}.  Polynomiality also follows from the following property of the kernel proven in \cite{MirSim}.  Define
\[F^M_{2k+1}(t)=\int_0^\infty x^{2k+1}H^M(x,t)dx.\]
Then
\[\frac{F^M_{2k+1}(t)}{(2k+1)!}=\sum_{i=0}^{k+1}\zeta(2i)(2^{2i+1}-4)\frac{t^{2k+2-2i}}{(2k+2-2i)!}
\]
so $F^M_{2k+1}(t)$ is a degree $2k+2$ polynomial in $t$ with leading coefficient $t^{2k+2}/(2k+2)$.  We prove analogous properties in Section~\ref{sec:kerprop} for kernels arising out of super hyperbolic surfaces which we will need when proving the Virasoro constraints in Section~\ref{sec:kdv}.
Polynomiality of the double integrals uses the same result.  By the change of coordinates $x=u+v$, $y=u-v$ one can prove
\begin{equation}  \label{dint}
\int_0^\infty\int_0^\infty x^{2i+1}y^{2j+1}H^M(x+y,t)dxdy=\frac{(2i+1)!(2j+1)!}{(2i+2j+3)!}F^M_{2i+2j+3}(t).
\end{equation}

\subsection{Super hyperbolic surfaces}
A locally ringed space $(M,\cf)$ is a pair given by a sheaf of rings $\cf$ over a topological space $M$ such that all stalks of $\cf$ are local rings.  A fundamental example is given by the sheaf $C^\infty(\br^m)$ of locally smooth functions on open sets of $\br^m$.  The fundamental super commutative example is
\[ \br^{m|n}=(\br^m,\co_{\br^{m|n}}),\qquad \co_{\br^{m|n}}=C^\infty(\br^m)\otimes\Lambda^*(\br^n).\]
A supermanifold is a locally ringed space $\widehat{M}=(M,\co_{\widehat{M}})$ locally isomorphic to $\br^{m|n}$.  Similarly, we define $\bc^{m|n}=(\bc^m,\co_{\bc^{m|n}})$ for $\co_{\bc^{m|n}}=\co_{\bc^m}\otimes\Lambda^*(\bc^n)$ where $\co_{\bc^m}$ is the sheaf of locally holomorphic functions.  A complex supermanifold is a locally ringed space locally isomorphic to $\bc^{m|n}$.   A morphism between two supermanifolds $(M_1,\co_{\widehat{M}_1})\to(M_2,\co_{\widehat{M}_2})$ is a pair $(f,F)$ consisting of a continuous map $f:M_1\to M_2$ between the two underlying topological spaces and a graded sheaf homomorphism $F:\co_{\widehat{M}_2}\to f_*\co_{\widehat{M}_1}$.  A family of supermanifolds is realised via a supermanifold defined over a base supermanifold $\widehat{M}\to S$ which is a morphism between $\widehat{M}$ and $S$. 

\subsubsection{} 
A {\em super Riemann surface} is a complex supermanifold $\widehat{\Sigma}$ of dimension $(1|1)$ with a dimension $(0|1)$ subbundle $\cd\subset T_{\widehat{\Sigma}}$ that is everywhere non-integrable.  Equivalently, $\cd$ and $\{\cd,\cd\}=\cd^2$ are linearly independent or
$T_{\widehat{\Sigma}}/\cd\cong\cd^2$. 
The transition functions are superconformal transformations of $\bc^{(1|1)}$ locally given by:
\begin{equation}  \label{superconformal}
\hat{z}=u(z)+\theta\eta(z)\sqrt{u'(z)},\quad \hat{\theta}=\eta(z)+\theta\sqrt{u'(z)+\eta(z)\eta'(z)}.
\end{equation}  
The dimension $(0|1)$ subbundle $\cd\subset T_{\widehat{\Sigma}}$ is locally generated by the super vector field $D$ given locally in superconformal coordinates by
\[D=\theta\frac{\partial}{\partial z}+\frac{\partial}{\partial \theta}.
\]
A vector field $v$ generates a superconformal transformation if the Lie derivative with respect to $v$ of $D$ preserves $D$, i.e. $[v,D]=\lambda D$ where $[\cdot,\cdot]$ is the commutator on even elements and anti-commutator on odd elements.  For example,
\[v=z\frac{\partial}{\partial z}+\frac12\theta\frac{\partial}{\partial \theta}
\]
satisfies $[v,D]=-\frac12D$ and
generates the scaling $(z|\theta)\mapsto (\lambda z|\lambda^{1/2}\theta)$ for $\lambda\in\bc^*$.  

The restriction of the tangent bundle of a super Riemann surface $\widehat{\Sigma}$ to its underlying Riemann surface $\Sigma\to\widehat{\Sigma}$ can be identified with $T_\Sigma\oplus T_\Sigma^\frac12$ , where the second factor gives fermionic directions.  Analogous to the deformation theory of the moduli space of Riemann surfaces, the tangent space to the moduli space of super Riemann surfaces is given by the cohomology group of the log-tangent bundle
\[H^1(\overline{\Sigma},\left(T_{\overline{\Sigma}}\oplus T_{\overline{\Sigma}}^\frac12\right)\otimes \co(-D))=H^1(\Sigma,T_{\overline{\Sigma}}(-D))\oplus H^1(\Sigma, T_{\overline{\Sigma}}^\frac12(-D))\]
for $D=\overline{\Sigma}-\Sigma$.   The component $H^1(\Sigma,T_{\overline{\Sigma}}(-D))$ is tangent along the bosonic directions which is isomorphic to the tangent space of the usual moduli space and $H^1(\Sigma, T_{\overline{\Sigma}}^\frac12(-D))$ is tangent along the fermionic directions---see \cite{FKPReg,LRoMod,WitNot}.  More generally it is shown in \cite{RSVGeo} that for any holomorphic line bundle $L\to\overline{\Sigma}$,  $H^0(\overline{\Sigma},L)\oplus H^0(\overline{\Sigma},L\otimes T_\Sigma^{-\frac12})$ is naturally a superspace  with $H^0(\overline{\Sigma},L)$ its even part and $H^0(\overline{\Sigma},L\otimes T_{\overline{\Sigma}}^{-\frac12})$ its odd part, and similarly for $H^1$, which can be identified with the cohomology of a holomorphic line bundle over a super Riemann surface.  

\subsubsection{}
In order to make contact with the work of Stanford and Witten \cite{SWiJTG}, we consider the functor of points of a supermanifold $\widehat{M}$ defined to be the set of morphisms from any supermanifold $P$ to $\widehat{M}$:
\[ \widehat{M}(P):=\text{Hom}(P,\widehat{M}).
\]
This produces a rather concrete description of (the points of) a supermanifold as a set.  We mainly take $P=\ba^{0|L}_\br:=(\{\text{pt}\},\Lambda_L(\br))$ where $\Lambda_L(\br)$ is the Grassmann algebra, defined below.  
\subsubsection{} \label{grassmann}
Define $\Lambda_L=\Lambda_L(\br)$ to be the Grassmann algebra over $\br$ with generators $\{1,e_1,e_2,....,e_L\}$.   We can similarly define $\Lambda_L(\bc)$ by replacing the field $\br$ by $\bc$.  An element $a\in\Lambda$ is a sum of monomials
\[a=a^\#+\sum_ia_ie_i+\sum_{i<j}a_{ij}e_i\wedge e_j+\sum_{i<j<k}a_{ij}e_i\wedge e_j\wedge e_k+...
\]
in the $2^N$ dimensional vector space $\Lambda_N$.  The element $a^\#\in\br$ is the {\em body} of $a$.  Define $\displaystyle\Lambda(\br)=\lim_{N\to\infty}\Lambda_N(\br)$ and $\ba^{0|\bullet}_\br=(\{\text{pt}\},\Lambda(\br))$.  The Grassmann algebra decomposes into even polynomials $\Lambda^0(\br)$, and odd polynomials $\Lambda^1(\br)$:
\[\Lambda(\br)=\Lambda^0(\br)\oplus \Lambda^1(\br)
\]
also known as the {\em bosonic} (even) and {\em fermionic} (odd) parts.  

\subsubsection{} 
Denote by $\br^{(m|n)}_{\bullet} =\br^{(m| n)}(\ba^{0|\bullet}_\br)$ points of the supermanifold $\br^{(m| n)}$, which are represented by
\[\br^{(m| n)}_\bullet =\{(z_1,z_2,...,z_m|\theta_1,...,\theta_n)\mid z_i\in \Lambda^0(\br),\ \theta_j\in \Lambda^1(\br)\}.
\]
Define $\bc^{(m| n)}_\bullet$ similarly.
Linear maps on $\br^{(m| n)}_\bullet$ are given by $(m+n)\times (m+n)$ matrices
\[G=\left(\begin{array}{c|c}A & B \\\hline C & D\end{array}\right)
\]
with even $m\times m$ blocks and $n\times n$ blocks $A$ and $D$, and odd $m\times n$ and $n\times m$ blocks $B$ and $C$.  The super transpose $G^{st}$ is defined by:
\[\left(\begin{array}{c|c}A&B\\\hline C&D\end{array}\right)^{st}=\left(\begin{array}{c|c}A^t&C^t\\\hline-B^t&D^t\end{array}\right)
\]
and the {\em Berezinian}, a generalisation of the determinant is defined by:
\[\text{Ber}\left(\begin{array}{c|c}A&B\\\hline C&D\end{array}\right)=\frac{\det(A-BD^{-1}C)}{\det(D)}
\]
which is invariant under the super transpose due to oddness of $B$ and $C$ .
Define
\[
M(2|1)=\left\{\left(\begin{array}{c|c}\begin{array}{cc}a&b\\c&d\end{array}&\begin{array}{c}\alpha\\\beta\end{array}\\\hline\begin{array}{cc}\gamma&\delta\end{array}&e\end{array}\right)\ \vline\  a,b,c,d,e\in\Lambda_0,\  \alpha,\beta,\gamma,\delta\in \Lambda_1\right\}
\]
and define $\text{OSp}(1|2)\subset M(2|1)$ (the label $(2|1)$ has switched) to be those elements of Berezinian equal to one that preserve the following bilinear form $J$:
\[ \text{OSp}(1|2)=\{ G\in M(2|1)\mid G^{st}JG=J,\ \text{Ber}(G)=1\},\quad J=\left(\begin{array}{ccc}0&1&0\\-1&0&0\\0&0&-1\end{array}\right).\]
The conditions $G^{st}JG=J$ 
and $\text{Ber}(G)=1$ lead to the following form of any element $G\in\text{OSp}(1|2)$: 
\begin{equation}  \label{matrixosp}
G=\left(\begin{array}{c|c}\begin{array}{cc}a\quad&\quad b\\c\quad&\quad d\end{array}&\begin{array}{c}\alpha\\\beta\end{array}\\ \hline a\beta-c\alpha\quad b\beta-d\alpha&1-\alpha\beta\end{array}\right)
\in\text{OSp}(1|2)
\end{equation}
where $ad-bc=1+\alpha\beta$.

\subsubsection{} \label{suphyp} Super hyperbolic space $\widehat{\bh}=(\bh,\co_{\widehat{\bh}})$ is the complex supermanifold with sheaf $\co_{\widehat{\bh}}=\co_\bh\otimes\Lambda^*(\bc)$ where $\co_\bh$ is the sheaf of locally holomorphic functions.  The inclusion $\co_\bh\to C^\infty(\bh)$ defines a natural map from the real super hyperbolic space to the (complex) super hyperbolic space.  Denote by $\widehat{\bh}_\bullet$ the $\ba^{0|\bullet}_\bc$ points of the family $\widehat{\bh}\times \ba^{0|\bullet}_\bc\to\ba^{0|\bullet}_\bc$.  It is realised by:
\[
\widehat{\bh}_\bullet=\{(z|\theta)\in\bc^{(1|1)}_\bullet\mid \text{Im }z^\#>0\}.
\]
There is an action of $\text{OSp}(1|2)$ on $\widehat{\bh}_\bullet$ which extends the action of the group $PSL(2,\br)$ of conformal transformations $z\mapsto \frac{az+b}{cz+d}$ of $\bh$, given by:
\[(z|\theta)\mapsto\left(\frac{az+b}{cz+d}+\theta\frac{\gamma z+\delta}{(cz+d)^2}\right|\left.\frac{\gamma z+\delta}{cz+d}+\frac{\theta}{cz+d}\right)
\]
where $\gamma=a\beta-c\alpha$ and $\delta=b\beta-d\alpha$.  
A discrete subgroup of $\text{OSp}(1|2)$ is Fuchsian if its image is Fuchsian under the map $\text{OSp}(1|2)\to SL(2,\br)$ defined by 
\[g\mapsto f^\#\left(\begin{array}{cc}a^\#&b^\#\\c^\#&d^\#\end{array}\right).\]  
The quotient of $\widehat{\bh}_\bullet$ by a Fuchsian subgroup defines (the $\ba^{0|\bullet}_\bc$ points of) a super hyperbolic surface.  The action by $\text{OSp}(1|2)$ on $\widehat{\bh}_\bullet$ is of the form \eqref{superconformal} hence the quotient super hyperbolic surface defines a super Riemann surface.
\subsubsection{}   \label{supermod}
The Teichm\"uller space of super hyperbolic surfaces has analogous constructions to those of usual Teichm\"uller space.  Coordinates on the Teichm\"uller space of super hyperbolic surfaces are constructed via representations, see Crane-Rabin \cite{CRaSup} and Natanzon \cite{NatMod}, via ideal triangulations, see Penner and Zeitlin \cite{PZeDec}, and via pairs of pants decompositions, see Stanford and Witten \cite{SWiJTG}.   The bosonic part of the Teichm\"uller space is the same as usual Teichm\"uller space despite the extra data of a spin structure as explained in \ref{qf}.  The quotient of the Teichm\"uller space of super hyperbolic surfaces by the mapping class group of the underlying hyperbolic surface gives rise to a well-defined moduli space.


\subsection{Recursion for super volumes}  \label{sec:sup}

Stanford and Witten \cite{SWiJTG} proved a generalisation of Mirzakhani's volume recursion using a generalisation of the identity \eqref{mcshane} to super hyperbolic surfaces.  They used torsion of the complex associated to the local system of a representation $\pi_1\Sigma\to\text{OSp}(1|2)$ to define the super volume measure, and via a generalisation of arguments of Mirzakhani reduced the calculation of the volume to an analysis of super hyperbolic pairs of pants. 
 
Given a super hyperbolic surface $\Sigma$ with $n$ geodesic boundary components denoted $\beta_1,...,\beta_n$, define $\cp_{i}$, respectively $\cp_{ij}$, to be the set of isometric embeddings $P\to \Sigma$ of super hyperbolic pairs of pants with geodesic boundary, which meet the boundary of $\Sigma$ precisely at $\beta_i$, respectively at $\beta_i$ and $\beta_j$.  A pair of pants $P(x,y,z|\alpha,\beta)$ now depends on three boundary lengths $x,y,z$ and two odd moduli $\alpha$, $\beta$.  As before $\ell_{\partial_iP}$ is the length of the $i$th geodesic boundary component of $P$, and $\alpha_P,\beta_P$ are its odd moduli.  Using a similar argument to the derivation of $D$ and $R$ in \ref{Rder}, Stanford and Witten derived
\[
\widehat{R}(x,y,z|\alpha,\beta)=x-\log\left(\frac{\cosh\frac{y}{2}+\cosh\frac{x+z}{2}-\frac12\alpha\beta(e^{\frac{x+z}{2}}+1)}{\cosh\frac{y}{2}+\cosh\frac{x-z}{2}-\frac12\alpha\beta(e^{\frac{x}{2}}+e^\frac{z}{2})}\right)
\]
which restricts to \eqref{Rlog} when $\alpha=0=\beta$.  Using $\alpha^2=0=\beta^2$, we can expand to get:
\begin{equation}  \label{supR}
\widehat{R}(x,y,z|\alpha,\beta)=R^M(x,y,z)-\alpha\beta \frac{2\pi e^{\frac{x+z}{4}}}{\cosh(\frac{y}{4})}R(x,y,z)
\end{equation}
and
\[\int_{\widehat{\modm}_{0,3}(x,y,z)}\widehat{R}(x,y,z|\alpha,\beta)d\mu=R(x,y,z)
\]
where the moduli space $\widehat{\modm}_{0,3}(x,y,z)$ is the vector space spanned by the two odd moduli $\alpha$, $\beta$. Integration is over the measure $d\mu$ which includes the odd variables $\alpha,\beta$ and a factor $\frac{1}{2\pi}\cosh(\frac{y}{4}) e^{-\frac{x+z}{4}}$ from the torsion of the circle as described in \cite{SWiJTG}.  This gives a geometric meaning to the kernel 
\[R(x,y,z)=\frac12 H(z,x+y)+\frac12 H(z,x-y)\]
for $H(x,y)=\frac{1}{4\pi}\left(\frac{1}{\cosh((x-y)/4)}-\frac{1}{\cosh((x+y)/4)}\right)
$
defined in \eqref{kerDR}.  

If we instead write $H(x,y)$ as
\[ H(x,y)=\frac{1}{2\pi}\left(\frac{e^\frac{-x+y}{4}}{1+e^\frac{-x+y}{2}}+\frac{e^\frac{x+y}{4}}{1+e^\frac{x+y}{2}}\right)
\] 
then it emphasises its similarities with Mirzakhani's kernel:
\[H^M(x,y)=\frac{1}{1+\exp\frac{x+y}{2}}+\frac{1}{1+\exp\frac{x-y}{2}}\]
and hence the resemblance of  $D(x,y,z)$ and $R(x,y,z)$ with Mirzakhani's kernels $D^M(x,y,z)$ and $R^M(x,y,z)$.

Define $\widehat{D}(x,y,z|\alpha,\beta)=\widehat{R}(x,y,z|\alpha,\beta)+\widehat{R}(x,z,y|\alpha',\beta')-x$ where $(\alpha',\beta')$ is an unspecified transformation of $(\alpha,\beta)$ which is unimportant after integration over the odd variables:
\[\int_{\widehat{\modm}_{0,3}(x,y,z)}\widehat{D}(x,y,z|\alpha,\beta)d\mu=D(x,y,z).
\]
For $P$ a super pair of pants, define $\widehat{R}(P)=\widehat{R}(\ell_{\partial_1P}=L_1,\ell_{\partial_2P}=L_j,\ell_{\partial_3P}|\alpha_P,\beta_P)$ and $\widehat{D}(P)=\widehat{D}(\ell_{\partial_1P}=L_1,\ell_{\partial_2P},\ell_{\partial_3P}|\alpha_P,\beta_P)$.
\begin{thm}[\cite{SWiJTG}]   \label{supermc}
For any super hyperbolic surface $\Sigma$ with $n$ geodesic boundary components of lengths $L_1,...,L_n$
\[L_1=\sum_{P\in\cp_1}\widehat{D}(P)+\sum_{j=2}^n\sum_{P\in\cp_{1j}}\widehat{R}(P).
\]
\end{thm}
In \cite{HPZSup} Huang, Penner and Zeitlin prove a super McShane identity in the case $(g,n)=(1,1)$ in a different way using a generalisation of Penner coordinates.

Following Mirzakhani's methods, Stanford and Witten applied Theorem~\ref{supermc} to produce the following recursion using the kernels $D(x,y,z)$ and $R(x,y,z)$ defined in \eqref{kerDR}.
\begin{thm}[\cite{SWiJTG}] 
\begin{align}  \label{volrecWP} 
L_1\widehat{V}^{WP}_{g,n}(L_1,L_K)=&\frac12\int_0^\infty\int_0^\infty xyD(L_1,x,y)P_{g,n+1}(x,y,L_K)dxdy\\
&\frac12\sum_{j=2}^n\int_0^\infty xR(L_1,L_j,x)\widehat{V}^{WP}_{g,n-1}(x,L_{K\backslash\{j\}})dx
\nonumber
\end{align}
where $K=\{2,...,n\}$ and
\[P_{g,n+1}(x,y,L_K)=\widehat{V}^{WP}_{g-1,n+1}(x,y,L_K)+\mathop{\sum_{g_1+g_2=g}}_{I \sqcup J = K}\widehat{V}^{WP}_{g_1,|I|+1}(x,L_I)\widehat{V}^{WP}_{g_2,|J|+1}(y,L_J).
\]
\end{thm}
Note that Stanford and Witten use a different normalisation $V^{SW}_{g,n}$ of the volume in \cite{SWiJTG}:
\[ V^{SW}_{g,n}(L_1,...,L_n)=(-2)^n\widehat{V}^{WP}_{g,n}(L_1,...,L_n)   =(-1)^n2^{1-g} V^\Theta_{g,n}(L_1,...,L_n). 
\]
Multiply \eqref{volrecWP} by $(-2)^n$ and absorb this into each volume, which replaces the coefficients $\frac12$ and $\frac12$ of the $D$ and $R$ terms by $-\frac14$ and $-1$, so that \eqref{volrecWP} now agrees with \cite[(5.42)]{SWiJTG}.  One can substitute $\widehat{V}^{WP}_{g,n}(L_1,...,L_n)=2^{1-g-n}V^{\Theta}_{g,n}(L_1,...,L_n)$ into \eqref{volrecWP} to retrieve \eqref{volrec}. 
The proof of \eqref{volrecWP} by Stanford and Witten uses supergeometry and currently has some non-rigorous aspects, which when made rigorous would produce a new proof of \eqref{taufn} in the spirit of Mirzakhani's proof of Theorem~\ref{KW}.

\section{Virasoro constraints}   \label{sec:kdv}
In this section we will represent recursion relations between polynomials via Virasoro constraints satisfied by associated partition functions.  Corollary~\ref{kdvtheta} shows that the top degree part of the recursion \eqref{volrec} can be represented by explicit Virasoro constraints.  Moreover, the whole recursion \eqref{volrec} can be indirectly represented by Virasoro constraints, which we express via topological recursion in the next section.

\subsection{KdV tau functions} 
A tau function $Z(\hbar,t_0,t_1,...)$ of the KdV hierarchy (equivalently the KP hierarchy in odd times $p_{2k+1}=t_k/(2k+1)!!$) gives rise to a solution $U=\frac{\partial^2}{\partial t_0^2}\log Z$ of the KdV hierarchy 
\begin{equation}\label{kdv}
U_{t_1}=UU_{t_0}+\frac{\hbar}{12}U_{t_0t_0t_0},\quad U(t_0,0,0,...)=f(t_0).
\end{equation}
The first equation in the hierarchy is the KdV equation \eqref{kdv}, and later equations $U_{t_k}=P_k(U,U_{t_0},U_{t_0t_0},...)$ for $k>1$ determine $U$ uniquely from $U(t_0,0,0,...)$, \cite{MJDSol}.

\subsubsection{} The Br\'ezin-Gross-Witten solution $U^{\text{BGW}}=\hbar\partial^2_{t_0}\log Z^{\text{BGW}}$ of the KdV hierarchy arises out of a unitary matrix model studied in \cite{BGrExt,GWiPos}.  It is defined by the initial condition
\[
U^{\text{BGW}}(t_0,0,0,...)=\frac{\hbar}{8(1-t_0)^2}.
\]
The first few terms of $\log Z^{\text{BGW}}$ are
\begin{align}  \label{lowg}
\log Z^{\text{BGW}}&=-\frac{1}{8}\log(1-t_0)+\frac{3\hbar}{128}\frac{t_1}{(1-t_0)^3}+\frac{15\hbar^2}{1024}\frac{t_2}{(1-t_0)^5}+\frac{63\hbar^2}{1024}\frac{t_1^2}{(1-t_0)^6}+...\\
&=\frac{1}{8}t_0+\frac{1}{16}t_0^2+\frac{1}{24}t_0^3+\hbar\frac{3}{128}t_1+\hbar\frac{9}{128}t_0t_1+\hbar^2\frac{15}{1024}t_2+\hbar^2\frac{63}{1024}t_1^2+...\nonumber
\end{align}

\subsubsection{}  The Kontsevich-Witten tau function $Z^{\text{KW}}$ given in Theorem~\ref{KW} is defined by the initial condition 
\[U^{\text{KW}}(t_0,0,0,...)=t_0\] 
for $U^{\text{KW}}=\hbar\partial^2_{t_0}\log Z^{\text{KW}}$.  The low genus terms of $\log Z^{\text{KW}}$ are 
\[\log Z^{\text{KW}}(\hbar,t_0,t_1,...)=\hbar^{-1}(\frac{t_0^3}{3!}+\frac{t_0^3t_1}{3!}+\frac{t_0^4t_2}{4!}+...)+\frac{t_1}{24}+...
\]
For each integer $m\geq-1$, define the differential operator
\begin{align}  \label{virop}
\widehat{\cl}_m& =  \frac{\hbar}{2} \hspace{-4mm}\mathop{\sum_{i+j=m-1}} \hspace{-3mm}(2i+1)!!(2j+1)!! \frac{\partial^2}{\partial t_i \partial t_j} +\sum_{i =0}^\infty \frac{(2i+2m+1)!!}{(2i-1)!!} t_i \frac{\partial}{\partial t_{i+m}} \\
&\qquad+ \frac{1}{8} \delta_{m,0}+\frac12\frac{t_0^2}{\hbar}\delta_{m,-1}   \nonumber
\end{align}
where the sum over $i+j=m-1$ is empty when $m=0$ or $-1$ and $\frac{\partial}{\partial t_{-1}}$ is the zero operator.
The Br\'ezin-Gross-Witten and Kontsevich-Witten tau functions satisfy the following equations \cite{DVVLoo,GNeUni,KonInt}.
\[
(2k+1)!!\frac{\partial}{\partial t_k}Z^{BGW}(\hbar,t_0,t_1,t_2,...)=\widehat{\cl}_kZ^{BGW}(\hbar,t_0,t_1,t_2,...),\quad k=0 ,1,2,...
\]
\[ 
(2k+3)!!\frac{\partial}{\partial t_{k+1}}Z^{KW}(\hbar,t_0,t_1,t_2,...)=\widehat{\cl}_kZ^{KW}(\hbar,t_0,t_1,t_2,...),\quad k=-1, 0 ,1,...
\]
These are known as Virasoro constraints when we write them instead as
\begin{equation}  \label{bgwvir}
\cl_mZ^{\text{BGW}}(\hbar,t_0,t_1,t_2,...)=0,\quad m=0,1,2,...
\end{equation}
and
\begin{equation}  \label{kwvir} \cl_m'Z^{\text{KW}}(\hbar,t_0,t_1,t_2,...)=0,\quad m=-1,0,1,...
\end{equation}
for
\begin{equation} \label{eq:virasoro} 
\cl_m = -\tfrac12(2m+1)!!\frac{\partial}{\partial t_m} + \tfrac12\widehat{\cl}_m,\qquad \cl_m' = -\tfrac12(2m+3)!!\frac{\partial}{\partial t_{m+1}} + \tfrac12\widehat{\cl}_m.
\end{equation}

The set of operators $\{\cl_0, \cl_1, \cl_2, \ldots\}$ satisfy the Virasoro commutation relations
\[
[\cl_m, \cl_n] = (m-n) \cl_{m+n}, \quad \text{for } m, n \geq 0.
\]
Similarly $\{\cl_{-1}',\cl_0', \cl_1', \ldots\}$ satisfy $[\cl_m', \cl_n'] = (m-n) \cl_{m+n}', \quad \text{for } m, n \geq -1$.

\subsubsection{Intersection numbers}
Kontsevich proved the conjecture of Witten that the KdV tau function $Z^{\text{KW}}$ stores the intersection numbers of $\psi$ classes in the following generating function:
\[ Z^{\text{KW}}(\hbar,t_0,t_1,...)=\exp\sum_{g,n,\vec{k}}\frac{\hbar^{g-1}}{n!}\int_{\overline{\modm}_{g,n}} \prod_{i=1}^n\psi_i^{k_i}t_{k_i}.
\]
Weil-Petersson volumes satisfy the recursion \eqref{mirzvolrec} and arise as intersection numbers over the moduli space of stable curves
\[V_{g,n}^{WP}(L_1,...,L_n)=\int_{\overline{\modm}_{g,n}}\exp\left\{2\pi^2\kappa_1+\frac12\sum_{i=1}^n L_i^2\psi_i\right\}.\]
Together these imply relations among intersection numbers over the moduli space of stable curves equivalent to Kontsevich's theorem which we state here in its Virasoro form.
\begin{thm}[Kontsevich \cite{KonInt}]  \label{KWthm}
\[\cl'_m\left(\exp\sum_{g,n,\vec{k}}\frac{\hbar^{g-1}}{n!}\int_{\overline{\modm}_{g,n}} \prod_{i=1}^n\psi_i^{k_i}t_{k_i}\right)=0,\ m\geq -1.\]
\end{thm}
We only sketch the proof due to Mirzakhani \cite{MirWei} using Weil-Petersson volumes since we will give the similar proof of the analogous result used to prove Theorem~\ref{main} in detail.
\begin{proof}
The top degree terms $\cv_g(\bf{L})$ of $V^{\text{WP}}_{g,n}(\bf{L})$ satisfy the homogeneous recursion:
\begin{align}  \label{toprecm}
&\frac{\partial}{\partial L_1}\left(L_1\cv_g(L_1,\LL_K) \right)=\sum_{j=2}^n L_j\bigg[ \int_0^{L_1-L_j}dx\cdot x(L_1-x) \cv_g(x, \LL_{K\setminus\{j\}})\\
&\hspace{4.8cm} + \frac12\int^{L_1+L_j}_{L_1-L_j}x(L_1+L_j-x) \cv_g(x, \LL_{K\setminus\{j\}}) \bigg] \nonumber \\ 
&\hspace{-1mm}+\frac12\int_0^{L_1}\hspace{-1mm}\int_0^{L_1-x}\hspace{-5mm}dxdy\cdot xy(L_1-x-y)\bigg[
\cv_{g-1}^{\text{WP}}(x,y,\LL_K)+\hspace{-3mm}\mathop{\sum_{g_1+g_2=g}}_{I\sqcup J=K}\hspace{-2mm}\cv_{g_1}^{\text{WP}}(x,\LL_I) \, \cv_{g_2}^{\text{WP}}(y,\LL_J)\bigg]\nonumber
\end{align}
where $K=\{2,...,n\}$.
We skip the proof of this since it is similar to the proof of Proposition~\ref{toprectheta} below.

Write $\displaystyle\langle\prod_{i=1}^n\tau_{k_i}\rangle:=\hbar^{g-1}\int_{\overline{\modm}_{g,n}} \prod_{i=1}^n\psi_i^{k_i}$
where $g$ is intrinsic on the left hand side via $\displaystyle 3g-3+n=\sum_{i=1}^n k_i$.  Then \eqref{toprecm} implies
\begin{align*}
(2k_1+1)!!\langle\prod_{i=1}^n\tau_{k_i}\rangle
&=\tfrac{\hbar}{2} \hspace{-3mm}\mathop{\sum_{i+j=k_1-2}} \hspace{-3mm}(2i+1)!!(2j+1)!!\Big(\langle\tau_i\tau_j\tau_K\rangle+ \hspace{-2mm}\sum_{I\sqcup J=K}\langle\tau_i\tau_I\rangle\langle\tau_j\tau_J\rangle\Big)\\
&\qquad +\sum_{j=2}^n\frac{(2k_1+2k_j-1)!!}{(2k_j-1)!!}\langle\tau_{k_1+k_j-1}\tau_{K\backslash\{j\}}\rangle
\end{align*}
which is equivalent to:
\[
(2k+1)!!\frac{\partial}{\partial t_k}Z(\hbar,t_0,t_1,t_2,...)=\widehat{\cl}_{k-1}Z(\hbar,t_0,t_1,t_2,...),\quad k=0 ,1,2,...
\]
This coincides with the Virasoro contraints satisfied by $Z^{\text{KW}}(\hbar,t_0,t_1,t_2,...)$ and they have the same initial condition $Z(t_0,0,0,...)=t_0^3/3!$ so coincide.
\end{proof}

\subsection{Recursion relations and Virasoro operators}
We now derive Virasoro operators from the top degree terms of \eqref{volrec} analogous to those produced in the proof of Theorem~\ref{KWthm}.  The Virasoro operators derived from \eqref{volrec} coincide with Virasoro operators that annihilate $Z^{\text{BGW}}$.  Following Mirzakhani's method, we express Virasoro constraints in terms of integral recursion relations satisfied by the top degree terms.  This is equivalent to the recursion \eqref{toprec} below which  first appeared in \cite{DNoTop}.

First we need to prove how the linear transformations defined by the kernels $D(x,y,z)$ and $R(x,y,z)$ in \eqref{volrec} act on polynomials analogous to a result of Mirzakhani.
Define
\[F_{2k+1}(t)=\int_0^\infty x^{2k+1}H(x,t)dx\]
where the kernel $H(x,y)$ defined in \eqref{kerH} is used to define $D(x,y,z)$ and $R(x,y,z)$ via \eqref{kerDR}.
\begin{lemma}  \label{polprop}
$F_{2k+1}(t)$ is a degree $2k+1$ monic polynomial in $t$.
\end{lemma}
\begin{proof}
\begin{align*}
F_{2k+1}(t)&=\frac{1}{4\pi}\int_0^\infty x^{2k+1}\left(\frac{1}{\cosh((x-t)/4)}-\frac{1}{\cosh((x+t)/4)}\right)dx\\
&=\frac{1}{4\pi}\int^{\infty}_{-t}\frac{(x+t)^{2k+1}}{\cosh{x/4}}dx-\frac{1}{4\pi}\int^{\infty}_t\frac{(x-t)^{2k+1}}{\cosh{x/4}}dx\\
&=\frac{1}{4\pi}\int^{\infty}_0\frac{(x+t)^{2k+1}-(x-t)^{2k+1}}{\cosh x/4}dx+\frac{1}{4\pi}\int^{0}_{-t}\frac{(x+t)^{2k+1}}{\cosh{x/4}}dx\\
&\qquad+\frac{1}{4\pi}\int^t_0\frac{(x-t)^{2k+1}}{\cosh{x/4}}dx\\
&=\frac{1}{4\pi}\int^{\infty}_0\frac{(x+t)^{2k+1}-(x-t)^{2k+1}}{\cosh x/4}dx\\
&=\frac{1}{2\pi}\sum_{i=0}^kt^{2i+1}\binom{2k+1}{2i+1}\int^{\infty}_0\frac{x^{2k-2i}}{\cosh x/4}dx\\
&=\sum_{i=0}^kt^{2i+1}\binom{2k+1}{2i+1}a_{k-i}\\
&=t^{2k+1}+O(t^{2k})
\end{align*}
where $a_n$ is defined by 
$\displaystyle \frac{1}{\cos(2\pi x)}=\sum_{n=0}^{\infty}a_n\frac{x^{2n}}{(2n)!}.$  In particular $a_0=1$ giving the final equality above.
\end{proof}
Analogous to \eqref{dint}, by the change of coordinates $x=u+v$, $y=u-v$, we have the following identity:
\[\int_0^\infty\int_0^\infty x^{2i+1}y^{2j+1}H(x+y,t)dxdy=\frac{(2i+1)!(2j+1)!}{(2i+2j+3)!}F_{2i+2j+3}(t).
\]
Since $D(x,y,z)=H(y+z,x)$ and $R(x,y,z)=\frac12 H(z,x+y)+\frac12 H(z,x-y)$ we have
\begin{equation}  \label{Dpol} 
\hspace{-1mm}\int_0^\infty\hspace{-2mm}\int_0^\infty\hspace{-2mm} x^{2i+1}y^{2j+1}D(L_1,x,y)dxdy=\hspace{-.5mm}\frac{(2i+1)!(2j+1)!}{(2i+2j+3)!}L_1^{2i+2j+3}+O(L_1^{2i+2j+2})
\end{equation}
and
\begin{equation}  \label{Rpol} 
\int_0^\infty x^{2k+1}R(L_1,L_j,x)dx=\frac12(L_1+L_j)^{2k+1}+\frac12(L_1-L_j)^{2k+1}+O(L^{2k})
\end{equation}
where the right hand sides of \eqref{Dpol} and \eqref{Rpol} are polynomial and $O(L^{2k})$ means the top degree terms are homogeneous of degree $2k$ in $L_1$ and $L_j$.   We see that the recursion \eqref{volrec} (and \eqref{volrecWP}) produces polynomials since the initial condition is a polynomial and it sends polynomials to polynomials.  So, for example,
\[ \int_0^\infty\int_0^\infty yzD(x,y,z)dydz=\frac{x^3}{6}+2\pi^2x
\]
and
\[\int_0^\infty zR(x,y,z)dz=x,\quad\int_0^\infty z^3R(x,y,z)dz=x(x^2+3y^2+12\pi^2).
\]

\begin{proposition}  \label{toprectheta}
The top degree terms $\cv_{g}(\bf{L})$ of any solution to \eqref{volrec} satisfy the homogeneous recursion:
\begin{align}  \label{toprec}
&L_1\cv_g(L_1,\LL_K) =\frac12\sum_{j=2}^n \bigg[ (L_j+L_1) \cv_g(L_j+L_1, \LL_{K\setminus\{j\}})\\
&\hspace{4cm} - (L_j-L_1) \cv_g(L_j-L_1, \LL_{K\setminus\{j\}}) \bigg] \nonumber\\
+\frac12&\int_0^{L_1}dx\cdot x(L_1-x)\bigg[
\cv_{g-1}(x,L_1-x,\LL_K)+\hspace{-3mm}\mathop{\sum_{g_1+g_2=g}}_{I\sqcup J=K}\hspace{-2mm}\cv_{g_1}(x,\LL_I) \, \cv_{g_2}(L_1-x,\LL_J)\bigg]\nonumber
\end{align}
where $K=\{2,...,n\}$.
\end{proposition}
\begin{proof}
From the properties \eqref{Dpol} and \eqref{Rpol}, the top degree terms $\cv_g(L_1,...,L_n)$ of a solution to \eqref{volrec} only depend on the top degree terms $\cv_{g'}(L_1,...,L_{n'})$ of the solution for $2g'-n'<2g-n$.   Moreover, 
\begin{align*}
\int_0^\infty xR(L_1,L_j,x)\cv_{g}&(x,L_{K\backslash\{j\}})dx=\frac12(L_j+L_1) \cv_g(L_j+L_1, \LL_{K\setminus\{j\}})\\
 &- \frac12(L_j-L_1) \cv_g(L_j-L_1, \LL_{K\setminus\{j\}})+\text{ lower order terms.}
\end{align*}
By \eqref{Dpol}, the double integral in \eqref{volrecWP} is a linear operator with input monomials $x^{2i+1}y^{2j+1}$ of $P_{g,n+1}(x,y,L_K)$ and output $\frac{(2i+1)!(2j+1)!}{(2i+2j+3)!}L_1^{2i+2j+3}$.  This linear operator can be realised via the following integral for input $x^my^n$:
\begin{equation}  \label{boxint}
\int_0^Lx^m(L-x)^ndx=\frac{m!n!}{(m+n+1)!}L^{m+n+1}
\end{equation}
which is immediate when $n=0$ and proven by induction for $n>0$ via differentiation of both sides by $L$.  Hence
\begin{align*}
\frac12\int_0^\infty\int_0^\infty \hspace{-2mm}xyD(L_1,x,y)P_{g,n+1}(x,y,L_K)dxdy=&\\
\int_0^{L_1}\hspace{-2mm}dx\cdot x(L_1-x)P_{g,n+1}(x,&L_1-x,\LL_K)
+\text{ lower order terms}
\end{align*}
and the proposition is proven.
\end{proof}
The polynomial $\cv_{g}(\bf{L})$ is homogeneous of degree $g-1$.  Note that \eqref{toprec} indeed produces a degree $g-1$ polynomial inductively starting from the initial condition $\cv_{1}(L_1)=$ constant.

\begin{cor}  \label{kdvtheta}
The recursion \eqref{toprec} satisfied by $\cv_{g}(\bf{L})$ is equivalent to the Virasoro constraints \eqref{bgwvir} applied to the following partition function built out of $\cv_{g}(\bf{L})$
\begin{equation}  \label{partvol}
Z^\cv(\hbar,\{t_m\})=\exp\sum_{g,n}\frac{\hbar^{g-1}}{n!}\cv_{g}(L_1,...,L_n)|_{\{L_i^{2m}=2^mm!t_m\}}.
\end{equation}
The initial condition $\cv_1(L)=\frac18$ implies that $Z(\hbar,\{t_m\})=Z^{\text{BGW}}(\hbar,\{t_m\})$, the Br\'ezin-Gross-Witten tau function of the KdV hierarchy.
\end{cor}
\begin{proof}
Define the coefficient of the monomial $\prod_{i=1}^nL_i^{2m_i}$ in $\cv_g(L_1,...,L_n)$ by
\[c_g(m_1,...,m_n):=\Big[\prod_{i=1}^nL_i^{2m_i}\Big]\cv_g(L_1,...,L_n)\] 
and for a set of positive integers $I=\{ i_1,...,i_k\}$ write $c(m_I):=c(m_{i_1},...,m_{i_k})$.
Since $\cv_g(L_1,...,L_n)$ is a degree $g-1$ symmetric homogeneous polynomial, the coefficient $c_g(m_1,...,m_n)$ is symmetric in the $m_i$  and it vanishes when $\sum_{i=1}^n m_i\neq g-1$.

Take $(2m_1+1)!$ times the coefficient of $\displaystyle L_1\prod_{i=1}^nL_i^{2m_i}$ in \eqref{toprec} to get:
\begin{align} \label{kdvrec}
(2m_1+1)!&c_g(m_1,m_K)
=\sum_{j=2}^n\frac{(2m_1+2m_j+1)!}{(2m_j)!}c_g(m_1+m_j,m_{S\setminus\{j\}})\\
+\tfrac12\hspace{-3mm}\mathop{\sum_{i+j=m_1-1}}\hspace{-3mm}&(2i+1)!(2j+1)!\left(c_{g-1}(i,j,m_K)+\hspace{-2mm}\sum_{I\sqcup J=K}c_{g_1}(i,m_I)c_{g_2}(j,m_J)\right)
\nonumber
\end{align}
where $K=\{2,...,n\}$.
The first term on the right hand side takes the coefficient of $L_1^{2m_1+1}L_j^{2m_j}$ in
\[\tfrac12\Big( (L_j+L_1)^{2k+1}-(L_j-L_1)^{2k+1}\Big)=L_1\sum_m\binom{2k+1}{2m+1}L_1^{2m}L_j^{2(k-m)}
\]
and the second first term on the right hand side uses \eqref{boxint} with $m=2i+1$, $n=2j+1$ and $m+n+1=2m_1+1$.

Define $C_g(m_1,...,m_n):=c_g(m_1,...,m_n)\prod_{i=1}^n2^{m_i}m_i!$ and put
\begin{align*}
F_{g,n}(\{t_m\})&:=\cv_{g}(L_1,...,L_n)|_{\{L_i^{2m}=2^mm!t_m\}}\\
&=\sum_{m\in\bz_+^n}c_g(m_1,...,m_n)\prod_{i=1}^n2^{m_i}m_i!t_{m_i}\\
&=\sum_{m\in\bz_+^n}C_g(m_1,...,m_n)\prod_{i=1}^nt_{m_i}
\end{align*}
so the partition function defined in \eqref{partvol} is $Z^\cv(\hbar,\{t_m\})=\exp\sum_{g,n}\frac{\hbar^{g-1}}{n!}F_{g,n}$ and
\[\frac{\partial^n}{\partial t_{m_1}...\partial t_{m_n}}\log Z^\cv(\hbar,\{t_m\})=\hbar^{g-1}C_g(m_1,...,m_n).
\]
The recursion \eqref{kdvrec} in terms of $C_g(m_1,...,m_n)$ becomes
\begin{align} \label{kdvrec1}
(2m_1+1)!!C_g(m_1,m_K)&
=\sum_{j=2}^n\frac{(2m_1+2m_j+1)!!}{(2m_j-1)!!}C_g(m_1+m_j,m_{S\setminus\{j\}})\\
+\tfrac12\hspace{-4mm}\mathop{\sum_{i+j=m_1-1}}\hspace{-3mm}(2i+1)!!(2j+1&)!!\left(C_{g-1}(i,j,m_K)+\hspace{-2mm}\sum_{I\sqcup J=K}C_{g_1}(i,m_I)C_{g_2}(j,m_J)\right).
\nonumber
\end{align}
 and \eqref{kdvrec1} for $k_1=0 ,1,2,...$ is equivalent to
\[ (2k+1)!!\frac{\partial}{\partial t_k}Z^\cv(\hbar,\{t_m\})=\widehat{\cl}_kZ^\cv(\hbar,\{t_m\}),\quad k=0 ,1,2,...
\]
where $\widehat{\cl}_k$ is defined in \eqref{virop}.  This coincides with the Virasoro constraints satisfied by $Z^{\text{BGW}}(\hbar,\{t_m\})$.  Furthermore, the initial condition $\cv_1(L)=\frac18$ 
is equivalent to the initial condition
\[\log Z^\cv(\hbar,t_0,0,0,...)=-\frac18\log(1-t_0)\] 
via $\cl_0Z^\cv(\hbar,t_0,0,0,...)=0$.  Hence $\partial^2_{t_0}\log Z^\cv(\hbar,t_0,0,0,...)=\frac{1}{8(1-t_0)^2}$ and 
\[Z^\cv(\hbar,t_0,t_1,t_2,...)=Z^{\text{BGW}}(\hbar,t_0,t_1,t_2,...).\]

\end{proof}
\begin{cor} \label{toproven}
Define $\cv_{g}(L_1,...,L_n)$ via the recursion \eqref{toprec} and the initial condition $\cv_{g}(L_1)=\frac18$.  Then
\[V^{\Theta}_{g,n}(L_1,...,L_n)=\cv_{g}(L_1,...,L_n)+\text{ lower order terms}.
\]
Equivalently, the top degree terms of $V^{\Theta}_{g,n}(L_1,...,L_n)$  satisfy the top degree part of the recursion \eqref{volrec}.
\end{cor}
\begin{proof}  
The equality \eqref{taufn}, proven via algebro-geometric methods in \cite{CGGRel}, together with Corollary~\ref{kdvtheta} shows that
\[Z^{\Theta}(\hbar,t_0,t_1,...)=Z^{\text{BGW}}(\hbar,t_0,t_1,...)=Z^\cv(\hbar,t_0,t_1,t_2,...).
\]
The polynomial storing the top degree terms of $V^{\Theta}_{g,n}(L_1,...,L_n)$ is obtained via
\[\int_{\overline{\modm}_{g,n}}\hspace{-2mm}\Theta_{g,n}\exp\left\{\frac12\sum_{i=1}^n L_i^2\psi_i\right\}\]
and the collection of these polynomials produces $Z^{\Theta}(\hbar,t_0,t_1,...)$ via the same construction as \eqref{partvol}.  Hence
\[\cv_{g}(L_1,...,L_n)=\int_{\overline{\modm}_{g,n}}\hspace{-2mm}\Theta_{g,n}\exp\left\{\frac12\sum_{i=1}^n L_i^2\psi_i\right\}.\]
\end{proof}

In the remainder of the paper, we will show that the top degree part of the recursion \eqref{volrec} implies the full recursion.  We will describe here why this is to be expected, via the analogous story in the non-super case.  The Weil-Petersson volumes $V^{WP}_{g,n}(L_1,...,L_n)$ are stored in a partition function, denoted $Z_{\kappa_1}(\hbar,\vec{t},s)$ in \ref{transl}, and the top degree terms of $V^{WP}_{g,n}(L_1,...,L_n)$ correspond to $Z^{\text{KW}}(\hbar,\vec{t})=Z_{\kappa_1}(\hbar,\vec{t},s)|_{s=0}$.  It was proven by Manin and Zograf \cite{MZoInv} that $Z_{\kappa_1}(\hbar,\vec{t},s)$ is a translation via \eqref{KWtran} of $Z^{\text{KW}}(\hbar,\vec{t})$, which satisfies Virasoro constraints, and hence inherits its own Virasoro constraints, which give another way to express Mirzakhani's recursion.  In other words, the top degree part of the recursion implies the full recursion.  

The same structure occurs in the super case---the partition function $Z^\Theta_{\kappa_1}(\hbar,\vec{t},s)$, defined in \eqref{BGWtran} and equivalent to the collection of polynomials $V^\Theta_{g,n}(L_1,...,L_n)$, is obtained by translation of $Z^\Theta(\hbar,\vec{t})$, given in \eqref{BGWtran}, which induces Virasoro constraints satisfied by $Z^\Theta_{\kappa_1}(\hbar,\vec{t},s)$.  This is a special case of Theorem~\ref{higherWPtrans}.  The Virasoro constraints satisfied by $Z^\Theta_{\kappa_1}(\hbar,\vec{t},s)$ are equivalent to recursion relations satisfied by $V^\Theta_{g,n}(L_1,...,L_n)$ and restrict, via $s=0$, to the Virasoro constraints satisfied by $Z^\Theta(\hbar,\vec{t})$.  The implementation of this idea to prove the recursion \eqref{volrec} is achieved via topological recursion in the next section.  

\subsubsection{Translation}  \label{transl}
The partition function 
\[Z_{\kappa_1}(\hbar,\vec{t},s)=\exp\left(\sum_{g,n}\frac{\hbar^{g-1}}{n!}\sum_{\vec{k}\in\bn^n}\int_{\overline{\modm}_{g,n}}\exp(s\kappa_1)\prod_{i=1}^n\psi_i^{k_i}t_{k_i}\right)
\]
is built out of the Weil-Petersson volumes
\[
Z_{\kappa_1}(\hbar,\vec{t},2\pi^2)=\exp\sum_{g,n}\frac{\hbar^{g-1}}{n!}V_{g,n}(L_1,...,L_n)|_{\{L_i^{2k}=2^kk!t_k\}}
\]
and was proven by Manin and Zograf \cite{MZoInv} to be related to the Kontsevich-Witten tau function via translation
\begin{equation}  \label{KWtran}
Z_{\kappa_1}(\hbar,\vec{t},s)=Z^{\text{KW}}(\hbar,t_0,t_1,t_2+s,t_3-s^2/2,...,t_k+(-1)^{k}\frac{s^{k-1}}{(k-1)!},...).
\end{equation}
Similarly, the Weil-Petersson super-volumes build a partition function
\[
Z_{\kappa_1}^\Theta(\hbar,\vec{t},2\pi^2)=\exp\sum_{g,n}\frac{\hbar^{g-1}}{n!}V_{g,n}^\Theta(L_1,...,L_n)|_{\{L_i^{2k}=2^kk!t_k\}}
\]
which is a translation of the Br\'ezin-Gross-Witten tau function.  We have
\begin{align}  \label{BGWtran}
Z^\Theta_{\kappa_1}(\hbar,\vec{t},s)&=\exp\left(\sum_{g,n}\frac{\hbar^{g-1}}{n!}\sum_{\vec{k}\in\bn^n}\int_{\overline{\modm}_{g,n}}\hspace{-2mm}\Theta_{g,n}\exp(s\kappa_1)\prod_{i=1}^n\psi_i^{k_i}t_{k_i}\right)\\
&=Z^{\text{BGW}}(\hbar,t_0,t_1+s,t_2-s^2/2,...,t_k+(-1)^{k+1}\frac{s^{k}}{k!},...).
\nonumber
\end{align} 
which is proven as a special case of a more general result involving all $\kappa$ classes in Theorem~\ref{higherWPtrans} below.  Note that the translation in \eqref{BGWtran} is shifts the indices by one compared to the translation in \eqref{KWtran}.

\subsubsection{Higher Weil-Petersson volumes}  \label{higherWP}
Define the generating function 
\[Z_{\kappa}(\hbar,\vec{t},\vec{s}):=\exp\sum_{g,n,\vec{m}}\frac{\hbar^{g-1}}{n!}\sum_{\vec{k}\in\bn^n}\int_{\overline{\modm}_{g,n}}\prod_{i=1}^n\psi_i^{k_i}t_{k_i}\prod_{j=1}^\infty\kappa_j^{m_j}\frac{s_j^{m_j}}{m_j!}.
\]
for integrals involving all $\kappa$ classes, known as higher Weil-Petersson volumes.
Define the weighted homogeneous polynomials $p_j$ of degree $j$ by
\[ 1-\exp\left(-\sum_{i=1}^\infty s_iz^i\right)=\sum_{j=1}^\infty p_j(s_1,...,s_j)z^j.\]
\begin{thm}[\cite{MZoInv}]  \label{MZ}
\[Z_{\kappa}(\hbar,\vec{t},\vec{s})=Z^{\text{KW}}(\hbar,t_0,t_1,t_2+p_1(\vec{s}),...,t_j+p_{j-1}(\vec{s}),....)
\] 
\end{thm}
The KdV hierarchy is invariant under translations, so an immediate consequence of Theorem~\ref{MZ} is that $Z_\kappa$ is a tau function of the KdV hierarchy in the $t_i$ variables, and the same is true of $Z_{\kappa}^\Theta$ defined analogously by
\[Z_{\kappa}^\Theta(\hbar,\vec{t},\vec{s}):=\exp\sum_{g,n,\vec{m}}\frac{1}{n!}\sum_{\vec{k}\in\bn^n}\int_{\overline{\modm}_{g,n}}\hspace{-2mm}\Theta_{g,n}\cdot\prod_{i=1}^n\psi_i^{k_i}t_{k_i}\prod_{j=1}^\infty\kappa_j^{m_j}\frac{s_j^{m_j}}{m_j!}.
\]
\begin{thm}  \label{higherWPtrans}
\[Z_{\kappa}^\Theta(\hbar,\vec{t},\vec{s})=Z^{\text{BGW}}(\hbar,t_0,t_1+p_1(\vec{s}),...,t_j+p_{j}(\vec{s}),....)\]
\end{thm}
\begin{proof}
When $\vec{s}=(s_1,s_2,...)=(0,0,...)$, the equality of the theorem coincides with  $Z^{\Theta}(\hbar,t_0,t_1,...)=Z^{\text{BGW}}(\hbar,t_0,t_1,...)$ which is proven in \cite{CGGRel}.  The proof of the general $\vec{s}\neq0$ case will follow from showing that it is obtained by translation of the $\vec{s}=0$ case.

The class $\Theta_{g,n}\in H^*(\overline{\modm}_{g,n},\bq)$ pulls back under the forgetful map by
\[\Theta_{g,n+1}=\psi_{n+1}\cdot\pi^*\Theta_{g,n}\]
which gives push-forward relations
\[\pi_*(\Theta_{g,n+1}\psi_{n+1}^{m})=\pi_*(\psi_{n+1}^{m+1}\cdot\pi^*\Theta_{g,n})=\Theta_{g,n}\kappa_m.
\]
This shifts the indices by one compared to the usual pushforward relation $\pi_*(\psi_{n+1}^{m+1})=\kappa_m$.

We will first prove the case $s_i=0$ for $i>1$, which is \eqref{BGWtran}.  The proof in \cite{MZoInv} of \eqref{KWtran} uses the following push-forward relation from \cite{KMZHig} for $\kappa_1^m$ involving a sum over ordered partitions of $m$. 
\begin{equation}  \label{KMZ}
\frac{\kappa_1^{m}}{m!}\prod_{j=1}^n\psi_j^{k_j}=\pi_*\left(\sum_{\mu\vdash m}\frac{(-1)^{m+\ell(\mu)}}{\ell(\mu)!}\prod_{j=n+1}^{n+\ell(\mu)}\frac{\psi_j^{\mu_j+1}}{\mu_j!}\prod_{j=1}^n\psi_j^{k_j}\right)
\end{equation} 
where $\mu\vdash m$ is an ordered partition of $m$ of length $\ell(\mu)$ and $\pi_*:\overline{\modm}_{g,n+\ell(\mu)}\to\overline{\modm}_{g,n}$.  The factor $\displaystyle\prod_{j=1}^n\psi_j^{k_j}$ in \eqref{KMZ} essentially does not participate since it can be replaced by its pull-back in the right hand side of \eqref{KMZ}, using $\displaystyle\psi_{n+1}\cdot\prod_{j=1}^n\psi_j^{k_j}=\psi_{n+1}\cdot\pi^*\prod_{j=1}^n\psi_j^{k_j}$,  and then brought outside of the push-forward.

Integrate \eqref{KMZ} to get
\[
\int_{\overline{\modm}_{g,n}}\frac{\kappa_1^{m}}{m!}\prod_{j=1}^n\psi_j^{k_j}=\sum_{\mu\vdash m}\frac{(-1)^{m+\ell(\mu)}}{\ell(\mu)!}\int_{\overline{\modm}_{g,n+\ell(\mu)}}\prod_{j=n+1}^{n+\ell(\mu)}\frac{\psi_j^{\mu_j+1}}{\mu_j!}\prod_{j=1}^n\psi_j^{k_j}
\]
which is easily seen to be equivalent to the translation \eqref{KWtran} on generating functions.  Notice that  $\mu_j+1\geq 2$ hence the first variable that is translated is $t_2$.

When $\Theta_{g,n}$ is present, there is a shift in the indices by one of the usual pushforward relations, hence $\psi_j^{\mu_j+1}$ in the right hand side of \eqref{KMZ} is replaced by $\psi_j^{\mu_j}$
\[
\Theta_{g,n}\frac{\kappa_1^{m}}{m!}\prod_{j=1}^n\psi_j^{k_j}=\pi_*\left(\Theta_{g,n+\ell(\mu)}\sum_{\mu\vdash m}\frac{(-1)^{m+\ell(\mu)}}{\ell(\mu)!}\prod_{j=n+1}^{n+\ell(\mu)}\frac{\psi_j^{\mu_j}}{\mu_j!}\prod_{j=1}^n\psi_j^{k_j}\right)
\]
which leads to the translation \eqref{BGWtran} on generating functions.  Notice now that  $\mu_j\geq 1$ and the first variable that is translated is $t_1$.  This also explains the shift in the indices by one between the translations \eqref{KWtran} and \eqref{BGWtran}.

We have proven that via translation, one can remove the term $\exp(2\pi^2\kappa_1)$ from $Z_{\kappa_1}^\Theta$, leaving $Z^\Theta$ which coincides with the Br\'ezin-Gross-Witten tau function $Z^{\text{BGW}}$.  Thus $Z_{\kappa_1}^\Theta$ is indeed a translation of $Z^{\text{BGW}}$.\\

The proof of the general case, when all $s_i$ are present, is similar, albeit more technical.  The following relation is proven in \cite{KMZHig}.
\begin{equation}  \label{KMZ1}
\frac{\kappa_1^{m_1}...\kappa_{N}^{m_N}}{m_1!...m_N!}\prod_{j=1}^n\psi_j^{k_j}=\pi_*\left(\sum_{k=1}^{|m|}\frac{(-1)^{|\bf{m}|+k}}{k!}\hspace{0mm}\sum_{ \mu\vdash_k{\bf m}}\prod_{j=n+1}^{n+k}\frac{\psi_j^{|\mu^{(j)}|+1}}{\mu^{(j)}!}\prod_{j=1}^n\psi_j^{k_j}\right)
\end{equation} 
where $\pi_*:\overline{\modm}_{g,n+N}\to\overline{\modm}_{g,n}$, ${\bf m}=(m_1,...,m_N)\in\bz^N$, and $\mu\vdash_k{\bf m}$ is a partition into $k$ parts, i.e. $\mu^{(1)}+...+\mu^{(k)}={\bf m}$, $\mu^{(j)}\neq 0$, $\mu^{(j)}\in \bz^N$, $\displaystyle|\mu^{(j)}|=\sum_i\mu^{(j)}_i$, $\mu^{(j)}!=\displaystyle\prod_i\mu^{(j)}_i!$.
As in the special case above, on the level of generating functions \eqref{KMZ1} leads to the translation in Theorem~\ref{MZ}

Again, when $\Theta_{g,n}$ is present, there is a shift in the indices by one in the pushforward relations, hence $\psi_j^{|\mu^{(j)}|+1}$ in the right hand side of \eqref{KMZ} is replaced by $\psi_j^{|\mu^{(j)}|}$
\[\Theta_{g,n}\frac{\kappa_1^{m_1}...\kappa_{N}^{m_N}}{m_1!...m_N!}\prod_{j=1}^n\psi_j^{k_j}=\pi_*\left(\Theta_{g,n+N}\sum_{k=1}^{|m|}\frac{(-1)^{|\bf{m}|+k}}{k!}\hspace{0mm}\sum_{ \mu\vdash_k{\bf m}}\prod_{j=n+1}^{n+k}\frac{\psi_j^{|\mu^{(j)}|}}{\mu^{(j)}!}\prod_{j=1}^n\psi_j^{k_j}\right)
\]
which has the effect of a shift of the indices by one compared to the translation in Theorem~\ref{MZ}.  By the proof of the case $\vec{s}=0$, we see that $Z_{\kappa}^\Theta(\hbar,\vec{t},\vec{s})$ is translation of the Br\'ezin-Gross-Witten tau function $Z^{\text{BGW}}$ given in the statement of the theorem.
\end{proof}
\begin{corollary}   \label{virec}
The polynomials $V^\Theta_{g,n}(L_1,...,L_n)$ satisfy a recursion that uniquely determines them from $V^\Theta_{1,1}(L)=\tfrac18$.
\end{corollary}
\begin{proof}
The partition function $Z^\Theta_{\kappa_1}(\hbar,\vec{t},s)$ is equivalent to the collection of polynomials $V^\Theta_{g,n}(L_1,...,L_n)$ via $
Z_{\kappa_1}(\hbar,\vec{t},2\pi^2)=\exp\left(\sum\frac{\hbar^{g-1}}{n!}V_{g,n}(L_1,...,L_n)|_{\{L_i^{2k}=2^kk!t_k\}}\right)
$.  Furthermore, $Z^\Theta_{\kappa_1}(\hbar,\vec{t},s)$ satisfies Virasoro constraints induced from the Virasoro constraints \eqref{bgwvir} satisfied by $Z^{\text{BGW}}(\hbar,\vec{t})$ due to their relation via translation \eqref{BGWtran}
proven in Theorem~\ref{higherWPtrans}.  The structure of the Virasoro operators shows that the constraints uniquely determine $Z^\Theta_{\kappa_1}(\hbar,\vec{t},s)$ from $\log Z^\Theta_{\kappa_1}(\hbar,t_0,0,0,...)=-\frac18\log(1-t_0)$.  Hence this induces recursion relations between the polynomials $V^\Theta_{g,n}(L_1,...,L_n)$ that uniquely determines them from $V^\Theta_{1,1}(L)=\tfrac18$.
\end{proof}
The recursion from Corollary~\ref{virec} is not yet explicit, and will turn out to coincide with the recursion \eqref{volrec},  using results from Section~\ref{sec:TR}, but more is needed to show this.   The top degree part of the recursion of Corollary~\ref{virec} uses only the $s=0$ specialisation of \eqref{BGWtran}, which is $Z^{\Theta}(\hbar,t_0,t_1,...)=Z^{\text{BGW}}(\hbar,t_0,t_1,...)$ hence it coincides with the top degree part of the recursion \eqref{volrec} by Corollary~\ref{kdvtheta} which is consistent with Corollary~\ref{toproven}.  A full proof of the recursion \eqref{volrec} and Theorem~\ref{main} will use Theorem~\ref{higherWPtrans} together with an efficient method to encode translation of partition functions, and Virasoro constraints achieved via topological recursion.

\section{Topological recursion}  \label{sec:TR}

Topological recursion produces a collection of correlators $\omega_{g,n}(p_1, \ldots, p_n)$, for $p_i\in C$, from a \emph{spectral curve} $(C,B,x,y)$ consisting of a compact Riemann surface $C$, a symmetric bidifferential $B$ defined on $C\times C$, and meromorphic functions $x,y:C\to\bc$.  It arose out of loop equations satisfied by matrix models and was developed by Chekhov, Eynard and Orantin \cite{CEyHer,EOrInv}.  A technical requirement is that the zeros of $dx$ are simple and disjoint from the zeros of $dy$~\cite{EOrInv}.  In many cases the bidifferential $B$ is taken to be the fundamental normalised differential of the second kind on $C$, \cite{FayThe}, and given by the Cauchy kernel $B= \frac{d z_1 d z_2}{(z_1-z_2)^2}$ when $C$ is rational with global rational parameter $z$.

The correlators $\omega_{g,n}(p_1,...,p_n)$ are a collection of symmetric tensor products of meromorphic 1-forms defined on $C^n$ where $p_i\in C$, for integers $g\geq 0$ and $n \geq 1$. 
They are defined recursively from 
$\omega_{g',n'}(p_1,...,p_{n'})$ for $(g',n')$ satisfying $2g'-2+n'<2g-2+n$.  The recursion can be represented pictorially via different ways of decomposing a genus $g$ surface with $n$ labeled boundary components into a pair of pants containing the first boundary component and simpler surfaces. 

For $2g-2+n>0$ and $L = \{2, \ldots, n\}$, define
\begin{align}  \label{EOrec}
\omega_{g,n}(p_1,p_L)=&\sum_{\alpha}\Res_{p=\alpha}K(p_1,p) \bigg[\omega_{g-1,n+1}(p,\hat{p},p_L)\\
&\hspace{3cm}+ \mathop{\sum_{g_1+g_2=g}}_{I\sqcup J=L}^\circ \omega_{g_1,|I|+1}(p,p_I) \, \omega_{g_2,|J|+1}(\hat{p},p_J) \bigg]
\nonumber
\end{align}
where the outer summation is over the zeros $\alpha$ of $dx$ and the $\circ$ over the inner summation means that we exclude terms that involve $\omega_1^0$.  The point $\hat{p}\in C$ is defined to be the unique point $\hat{p}\neq p$ close to $\alpha$ such that $x(\hat{p})=x(p)$.  It is unique 
since each zero $\alpha$ of $dx$ is assumed to be simple, and \eqref{EOrec} needs only consider $p\in C$ close to $\alpha$.   The recursion takes as input the unstable cases
\[
\omega_{0,1}=-y(p_1)\,dx(p_1) \qquad \text{and} \qquad \omega_{0,2}=B(p_1,p_2).
\]
\begin{center}
\includegraphics[scale=0.4]{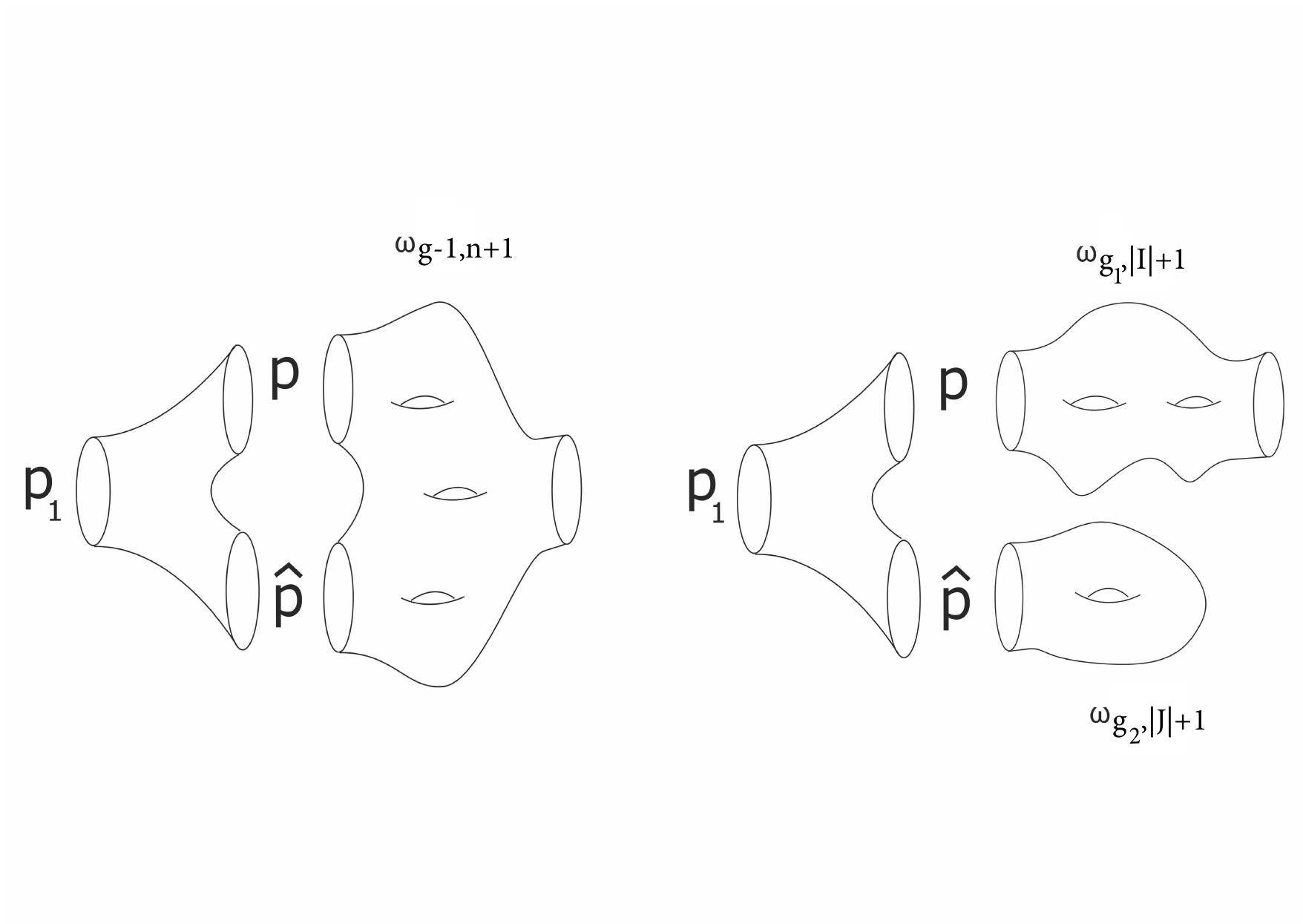}
\end{center}
The kernel $K$ is defined by 
\[
K(p_1,p)=\frac{-\int^p_{\hat{p}}\omega_2^0(p_1,p')}{2[y(p)-y(\hat{p})] \, dx(p)}
\] 
which is well-defined in a neighbourhood of each zero of $dx$. Note that the quotient of a differential by the differential $dx(p)$ is a meromorphic function.  For $2g-2+n>0$, the correlator $\omega_{g,n}$ is symmetric, with poles only at the zeros of $dx$ and vanishing residues.

The poles of the correlator $\omega_{g,n}$ occur at the zeros of $dx$.  A zero $\alpha$ of $dx$ is {\em regular}, respectively {\em irregular}, if $y$ is regular, respectively has a simple pole, at $\alpha$.  A spectral curve is regular if all zeros of $dx$ are regular and irregular otherwise.  The order of the pole in each variable of $\omega_{g,n}$ at a regular, respectively irregular, zero of $dx$ is $6g-4+2n$, respectively $2g$, \cite{DNoTopI,EOrInv}.

Two cases of interest in this paper  
use $x=\frac12z^2$, $B$ is the Cauchy kernel and $y=\frac{\sin(2\pi z)}{2\pi}$, respectively $y=\frac{\cos(2\pi z)}{z}$.  The recursion \eqref{EOrec} allows for functions $y$ that are not algebraic as in these two examples.  Moreover, the recursive definition of $\omega_{g,n}(p_1, \ldots, p_n)$ uses only local information of $x$, $y$ and $B$ around zeros of $dx$.   In particular, $y$ and $B$ need to be only defined in a neighbourhood of the zeros of $dx$ and topological recursion generalises to {\em local} curves in which $C$ is an open subset of a compact Riemann surface \cite{EynInv}. 

\subsubsection{}
In many examples $\omega_{g,n}(p_1,p_2,...,p_n)$ gives the coefficients in the large $N$ expansion of expected values of multiresolvents in a matrix model 
\[\left\langle{\rm Tr}\left(\frac{1}{x(p_1)-A}\right)...{\rm Tr}\left(\frac{1}{x(p_n)-A}\right)\right\rangle_c\] 
where $N$ is the size of the matrix and $g$ indexes the order in the $1/N$ expansion.  The subscript $c$ means cumulant, or the connected part in a graphical expansion.  In such cases, topological recursion follows from the loop equations satisfied by the resolvents.  Saad, Shenker and Stanford \cite{SSSJTG} introduced a matrix model corresponding to the spectral curve $x=\frac12z^2$, $y=\frac{\sin(2\pi z)}{2\pi}$.  Stanford and Witten \cite{SWiJTG} used these ideas to produce the spectral curve $x=\frac12z^2$, $y=\frac{\cos(2\pi z)}{z}$. 

\subsubsection{} Define $\Phi(p)$ up to an additive constant by $d\Phi(p)=y(p)dx(p)$.
For $2g-2+n>0$, the correlators $\omega_{g,n}$ satisfy the dilaton equation~\cite{EOrInv}
\begin{equation} \label{dilaton}
\sum_{\alpha}\Res_{p=\alpha}\Phi(p)\, \omega_{g,n+1}(p,p_1, \ldots ,p_n)=(2-2g-n) \,\omega_{g,n}(p_1, \ldots, p_n),
\end{equation}
where the summation is over the zeros $\alpha$ of $dx$.  The relation \eqref{dilaton} is invariant under $\Phi\mapsto\Phi+c$ where $c$ is a constant, since the poles of $\omega_{g,n+1}(p,p_1, \ldots ,p_n)$ are residueless. The dilaton equation enables the definition of the so-called {\em symplectic invariants}
\[
\omega_{g,0}=\sum_{\alpha}\Res_{p=\alpha}\Phi(p)\,\omega_{g,1}(p).
\]

\subsubsection{}  \label{sec:TRpart}
The correlators $\omega_{g,n}$ are normalised differentials of the second kind in each variable---they have zero $\ca$-periods, and poles only at the zeros $\cp_i$ of $dx$ of zero residue.  Their principal parts are skew-invariant under the local involution $p\mapsto\hat{p}$.   The correlators $\omega_{g,n}$ are polynomials in a basis $V^i_k(p)$ of normalised differentials of the second kind, which have poles only at the zeros of $dx$ with skew-invariant principal part, constructed from $x$ and $B$ as follows.
\begin{definition}\label{evaluationform}
For a Riemann surface equipped with a meromorphic function $(\Sigma,x)$ we define evaluation of any meromorphic differential $\omega$ at a simple zero $\cp$ of $dx$ by
\[
\omega(\cp)^2:=\Res_{p=\cp}\frac{\omega(p)\otimes\omega(p)}{dx(p)}\in\bc
\]
and we choose a square root of $\omega(\cp)^2$ to remove the $\pm1$ ambiguity.
\end{definition}
 
\begin{definition}\label{auxdif}
For a Riemann surface $C$ equipped with a meromorphic function $x:C\to\bc$ and bidifferential $B(p_1,p_2)$ define the auxiliary differentials on $C$ as follows.  For each zero $\cp_i$ of $dx$,  define
\begin{equation}  \label{Vdiff}
\xi^i_0(p)=B(\cp_i,p),\quad \xi^i_{k+1}(p)=-d\left(\frac{\xi^i_k(p)}{dx(p)}\right),\ i=1,...,N,\quad k=0,1,2,...
\end{equation}
where evaluation $B(\cp_i,p)$ at  $\cp_i$ is given in Definition~\ref{evaluationform}.
\end{definition}

From any spectral curve $S$, one can define a partition function $Z^S$ by assembling the polynomials built out of the correlators $\omega_{g,n}$ \cite{DOSSIde,EynInv}.
\begin{definition}  \label{TRpart}
\[Z^S(\hbar,\{u^{\alpha}_k\}):=\left.\exp\sum_{g,n}\frac{\hbar^{g-1}}{n!}\omega^S_{g,n}\right|_{\xi^{\alpha}_k(p_i)=u^{\alpha}_k}.
\]
\end{definition}

\begin{thm} [\cite{DOSSIde}] \label{DOSS}
Given any semisimple CohFT $\Omega$ with flat unit, there exists a local spectral curve $S$ whose topological recursion partition function coincides with the partition function of the CohFT:
\[ Z^S(\hbar,\{u^{\alpha}_k\})=Z_{\Omega}(\hbar,\{t^{\alpha}_k\})
\]
for $\{u^{\alpha}_k\}$ linearly related to $\{t^{\alpha}_k\}$.
\end{thm}
The following converse to Theorem~\ref{DOSS} allows for CohFTs without unit, and in particular a CohFT is not required to have flat unit.
\begin{thm}[\cite{CNoTop}]  \label{CNo}
Consider a spectral curve $S=(\Sigma,B,x,y)$ with possibly irregular zeros of $dx$.  There exist a CohFT $\Omega$, possibly without unit,  such that 
\[ Z^S(\hbar,\{u^{\alpha}_k\})=Z_{\Omega}(\hbar,\{t^{\alpha}_k\}).
\]
\end{thm}
Theorem~\ref{CNo} is a consequence of the following more technical result from \cite{CNoTop}.  Given a spectral curve $S=(\Sigma,B,x,y)$ with $m$ irregular zeros of $dx$ at which $y$ has simple poles, and $D-m$ regular zeros, there exist operators $\hat{R}$, $\hat{T}$ and $\hat{\Delta}$ determined explicitly by $(\Sigma,B,x,y)$ such that the partition function $Z^S$ built from the topological recursion correlators $\omega_{g,n}$ satisfies the following factorisation formula:
\begin{equation}  \label{eq:fact}
Z^S\hspace{-1mm}=\hspace{-1mm}\hat{R}\hat{T}\hat{\Delta}\left[\prod_{j=1}^mZ^{\text{BGW}}(\hbar,\{v^{k,j}\})\hspace{-1mm}\prod_{j=m+1}^D\hspace{-2mm}Z^{\text{KW}}(\hbar,\{v^{k,j}\})\right]
\end{equation}
where  $\{v^{k,j}\}$ are explicit linear combinations of $\{t^{\alpha}_k\}$.
The operators $\hat{R}$, $\hat{T}$ and $\hat{\Delta}$ can be used to construct a CohFT with partition function given by the right hand side of \eqref{eq:fact}. The equality 
\[
Z^{\Theta}(\hbar,t_0,t_1,...)=Z^{\text{BGW}}(\hbar,t_0,t_1,...)
\]
proven in \cite{CGGRel} allows us to replace factors of $Z^{\text{BGW}}$ in \eqref{eq:fact} by factors of $Z^\Theta$.   In particular, this will allow us to produce a spectral curve which stores the polynomials $V_{g,n}^\Theta(L_1,...,L_n)$ in its topological recursion correlators $\omega_{g,n}$.  To explain this, we will first describe the spectral curve which stores the polynomials $V^{WP}_{g,n}(L_1,...,L_n)$.

The CohFT (without flat unit) $\Omega_{g,n}=\exp(2\pi^2\kappa_1)$ has partition function 
\begin{align*}
Z_{\Omega}(\hbar,\{t_k\})&=\exp\sum_{g,n,\vec{k}}\frac{\hbar^{g-1}}{n!}\int_{\overline{\modm}_{g,n}}\exp(2\pi^2\kappa_1)\cdot\prod_{j=1}^n\psi_j^{k_j}\prod t_{k_j}\\
&=\exp\sum_{g,n}\frac{\hbar^{g-1}}{n!}V_{g,n}(L_1,...,L_n)|_{\{L_i^{2k}=2^kk!t_k\}}.
\end{align*}
Its relation to topological recursion, given in the following theorem, was proven by Eynard and Orantin.  It is also a  consequence of Theorem~\ref{CNo}.   
\begin{thm}[\cite{EOrWei}]   \label{eorwei}
Topological recursion applied to the spectral curve
\[ S_{EO}=\left(\bc,x=\frac12z^2,y=\frac{\sin(2\pi z)}{2\pi},B=\frac{dzdz'}{(z-z')^2}\right)
\]
has partition function
\[
Z_{S_{EO}}(\hbar,\{t_k\})=\exp\sum_{g,n}\frac{\hbar^{g-1}}{n!}V^{WP}_{g,n}(L_1,...,L_n)|_{\{L_i^{2k}=2^kk!t_k\}}.
\]
\end{thm}
\begin{remark}  \label{partcor}
The partition function $Z_{S_{EO}}(\hbar,\{t_k\})$ in Theorem~\ref{eorwei} uses $\xi_k=(2k-1)!!\frac{dz}{z^{2k}}$ defined in \eqref{Vdiff} to get
\[Z_{S_{EO}}(\hbar,\{t_k\})=\left.\exp\sum_{g,n}\frac{\hbar^{g-1}}{n!}\omega^S_{g,n}\right|_{\xi_k(z_i)=t_k}\hspace{-7mm}=\exp\sum_{g,n}\frac{\hbar^{g-1}}{n!}V^\Theta_{g,n}(L_1,...,L_n)|_{\{L_i^{2k}=2^kk!t_k\}}.
\]
Hence the expression for $Z_{S_{EO}}(\hbar,\{t_k\})$ in Theorem~\ref{eorwei} is equivalent to the following expression for correlators
\[\omega_{g,n}=\frac{\partial}{\partial z_1}...\frac{\partial}{\partial z_n}\cl\{V^{WP}_{g,n}(L_1,...,L_n)\}dz_1...dz_n.\]
\end{remark}

\subsection{The spectral curve}
In this section we prove Theorems~\ref{main} and \ref{spectral}.
The following theorem is a restatement of Theorem~\ref{spectral} in terms of the partition function $Z_S$ which collects all of the correlators $\omega_{g,n}$.
\begin{thm}  \label{TRTheta}
Topological recursion applied to the spectral curve
\[ S=\left(\bc,x=\frac12z^2,y=\frac{\cos(2\pi z)}{z},B=\frac{dzdz'}{(z-z')^2}\right)\]
has partition function
\[
Z_{S}(\hbar,\{t_k\})=\exp\sum_{g,n}\frac{\hbar^{g-1}}{n!}V^\Theta_{g,n}(L_1,...,L_n)|_{\{L_i^{2k}=2^kk!t_k\}}.
\]
\end{thm}
\begin{proof}
We use the following result from \cite{NorGro}.  Given any regular spectral curve $S=(\Sigma,x,y,B)$ form the irregular spectral curve $S'=(\Sigma,x,dy/dx,B)$.  It is irregular because $dy/dx$ necessarily has poles at the zeros of $dx$.   The factorisation of $Z^S$ given by \eqref{eq:fact}
\[Z^S=\hat{R}\hat{T}\hat{\Delta}Z^{\text{KW}}(\hbar,\{v^{k,m+1}\})... Z^{\text{KW}}(\hbar,\{v^{k,D}\})\]
is related to the factorisation of $Z^{S'}$ by:
\[Z^{S'}=\hat{R}\hat{T}_0\hat{\Delta}Z^{\text{BGW}}(\hbar,\{v^{k,m+1}\})... Z^{\text{BGW}}(\hbar,\{v^{k,D}\})\]
where $T_0(z)=T(z)/z$ is the shift of the indices by one between the translations, explained in Theorem~\ref{higherWPtrans}.  Moreover, due to \eqref{taufn}, if the partition function comes from a CohFT, i.e. $Z^S=Z_\Omega$, then $Z^{S'}=Z_{\Omega^\Theta}$.  This relation is simplified when $dx$ has a single zero, since $R=I$ and it essentially reduces to the shift of the indices by one between the translations, which is clearly visible in \eqref{KWtran} and \eqref{BGWtran}.

Apply this to $S=S_{EO}$ which transforms to $S'$ by 
\[x=\frac12z^2,\ y=\frac{\sin(2\pi z)}{2\pi}\quad\leadsto\quad x=\frac12z^2,\ \frac{dy}{dx}=\frac{\cos(2\pi z)}{z}.
\]
By Theorem~\ref{eorwei}, 
\[Z^{S_{EO}}=\exp\left(\sum_{g,n}\frac{\hbar^{g-1}}{n!}\sum_{\vec{k}\in\bn^n}\int_{\overline{\modm}_{g,n}}\exp(2\pi\kappa_1)\prod_{i=1}^n\psi_i^{k_i}t_{k_i}\right)\]
hence
\begin{align*}
Z^{S'}&=\exp\left(\sum_{g,n}\frac{\hbar^{g-1}}{n!}\sum_{\vec{k}\in\bn^n}\int_{\overline{\modm}_{g,n}}\hspace{-2mm}\Theta_{g,n}\exp(2\pi\kappa_1)\prod_{i=1}^n\psi_i^{k_i}t_{k_i}\right)\\
&=\exp\sum_{g,n}\frac{\hbar^{g-1}}{n!}V^\Theta_{g,n}(L_1,...,L_n)|_{\{L_i^{2k}=2^kk!t_k\}}.
\end{align*}  
\end{proof}
The correlators $\omega_{g,n}$ of the spectral curve $S=S'_{EO}$ are polynomials in the same auxiliary differentials $\xi_k=(2k-1)!!\frac{dz}{z^{2k}}$ as for $S_{EO}$, hence Remark~\ref{partcor} again applies to show that the expression for $Z_{S}(\hbar,\{t_k\})$ in Theorem~\ref{TRTheta} is equivalent to the expression for correlators given in Theorem~\ref{spectral}:
\[\omega_{g,n}=\frac{\partial}{\partial z_1}...\frac{\partial}{\partial z_n}\cl\{V^\Theta_{g,n}(L_1,...,L_n)\}dz_1...dz_n\]
Theorem~\ref{spectral} enables us finally to prove Theorem~\ref{main}, using the recursion between the polynomials $V^\Theta_{g,n}(L_1,...,L_n)$ produced via topological recursion satisfied by $\omega_{g,n}$.

In preparation, we require the following property of the principal part of a rational function.  The principal part of a rational function $r(z)$ at a point $\alpha\in\bc$, denoted by $[r(z)]_{\alpha}$, is the negative part of the Laurent series of $r(z)$ at $\alpha$.  It has the integral expression
\[
[r(z)]_{z=\alpha}=\Res_{w=\alpha}\frac{r(w)dw}{z-w}
\]
since the right hand side is analytic for $z\in\bc\backslash\{\alpha\}$ and 
\[
r(z)=-\Res_{w=z}\frac{r(w)dw}{z-w}=\frac{1}{2\pi i}\int_{\gamma_1-\gamma_2}\frac{r(w)dw}{z-w}=[r(z)]_{z=\alpha}-\frac{1}{2\pi i}\int_{\gamma_2}\frac{r(w)dw}{z-w}
\]
\begin{center} 
\includegraphics[scale=0.2]{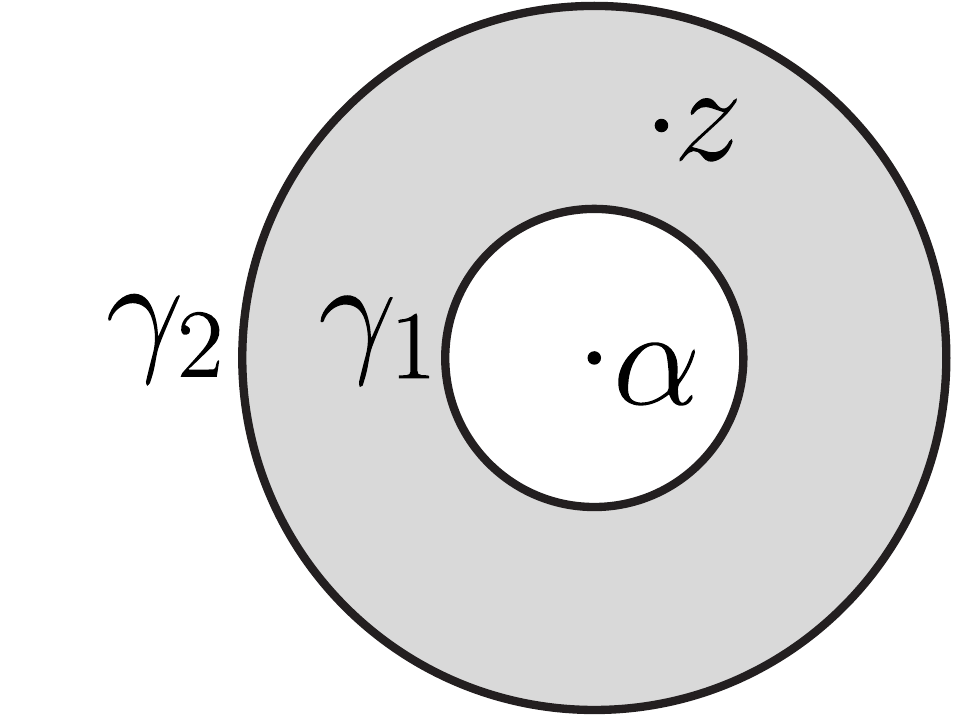}
\end{center}
so that $r(z)-[r(z)]_{z=\alpha}$ is analytic in the region enclosed by $\gamma_2$ in the diagram.   For $\alpha=0$, the even and odd parts of the principal part under $z\mapsto-z$ are denoted by $[r(z)]_{z=0}^+$, respectively $[r(z)]_{z=0}^-$.

In the following theorem, we use $T_{g,n}(L_1,...,L_n)$ to denote symmetric polynomials which will turn out to coincide with $V^\Theta_{g,n}(L_1,...,L_n)$.
\begin{thm}   \label{thetaspec}
The Laplace transform of the recursion \eqref{volrec} satisfied by symmetric polynomials $T_{g,n}(L_1,...,L_n)$ with $T_{1,1}\equiv\frac18$ is equivalent to topological recursion applied to the spectral curve
\[ S=\left(\bc,x=\frac12z^2,y=\frac{\cos(2\pi z)}{z}
,B=\frac{dzdz'}{(z-z')^2}\right)
\]
with correlators
\[\omega_{g,n}=\frac{\partial}{\partial z_1}...\frac{\partial}{\partial z_n}\cl\{T_{g,n}(L_1,...,L_n)\}dz_1...dz_n.\]
\end{thm}
\begin{proof}
The proof is analogous to the proof of Theorem 7.6 by Eynard and Orantin in \cite{EOrWei}.  It is rather technical so we will give the key idea here.  Topological recursion applied to the spectral curve $S$  is related to the recursion \eqref{volrec} by the Laplace transform, and in particular there is a one-to-one correspondence between terms in each of the two recursions.  Lemmas~\ref{lapkerD} and \ref{lapkerR} are the main new ideas in the proof, enabling the calculation of the Laplace transform of the recursion \eqref{volrec}, while the last part of the proof uses techniques which have arisen previously to relate topological recursion to a variety of recursive structures in geometry.

The Laplace transform of a polynomial $P(x_1,...,x_n)$ which is defined by 
\[\cl\{P\}(z_1,...,z_n)=\int_0^\infty...\int_0^\infty e^{-(z_1x_1+...+z_nx_n)}P(x_1,...,x_n)dx_1...dx_n\]
for $Re(z_i)>0$, is a polynomial in $z_i^{-1}$ hence it extends to a meromorphic function on $\bc^n$ with poles along the divisors $z_i=0$. 

The recursion \eqref{volrec} involves the following two linear transformations
\[P(x,y)\mapsto\int_0^\infty\hspace{-2mm}\int_0^\infty\hspace{-1mm}D(z,x,y)P(x,y)dxdy,\quad P(z)\mapsto\int_0^\infty \hspace{-1mm}R(x,y,z)P(z)dz\]
from the spaces of odd (in each variable) polynomials in one and two variables to the spaces of polynomials in two and one variable.  These linear transformations induce linear transformations of the Laplace transforms.   Lemmas~\ref{lapkerD} and \ref{lapkerR} below calculate the Laplace transform of these linear transformations.  
\begin{lemma}  \label{lapkerD}
For $P(x,y)$ an odd polynomial in $x$ and $y$:
\[
\cl\left\{\int_0^\infty\hspace{-2mm}\int_0^\infty dxdy D(L,x,y)P(x,y)\right\}=\left[\frac{1}{\cos(2\pi z)}\cl\{P\}(z,z)\right]_{z=0}
\]
\end{lemma}
\begin{proof}
By linearity we may choose $P=\frac{x^{2i+1}y^{2j+1}}{(2i+1)!(2j+1)!}$ which has Laplace transform $\cl\{P\}(z_1,z_2)=\frac{1}{z_1^{2i+2}z_2^{2j+2}}$.  From Lemma~\ref{polprop} we have
\[
F_{2k+1}(t)=\int_0^\infty x^{2k+1}H(x,t)dx=\sum_{i=0}^kt^{2i+1}\binom{2k+1}{2i+1}a_{k-i}\\
\]
where $a_n$ is defined by
$\displaystyle \frac{1}{\cos(2\pi z)}=\sum_{n=0}^{\infty}a_n\frac{z^{2n}}{(2n)!}.$  
Then $D(x,y,z)=H(y+z,x)$ and a change of coordinates gives:
\begin{align*}
\int_0^\infty\hspace{-2mm}\int_0^\infty\hspace{-2mm} \frac{x^{2i+1}y^{2j+1}}{(2i+1)!(2j+1)!}D(L,x,y)dxdy&=\frac{F_{2i+2j+3}(L)}{(2i+2j+3)!}\\
&=\sum_{m=0}^{i+j+1}\frac{L^{2m+1}}{(2m+1)!}\frac{a_{i+j+1-m}}{(2i+2j+2-2m)!}.
\end{align*}
Hence its Laplace transform is
\[
\cl\left\{\int_0^\infty\hspace{-2mm}\int_0^\infty\hspace{-2mm} \frac{x^{2i+1}y^{2j+1}}{(2i+1)!(2j+1)!}D(L,x,y)dxdy\right\}=\sum_{m=0}^{i+j+1}\frac{1}{z^{2m+2}}\frac{a_{i+j+1-m}}{(2i+2j+2-2m)!}
\]
which coincides with the even principal part of 
\[\frac{1}{\cos(2\pi z)}\cl\{P\}(z,z)\sim \sum_{n=0}^{\infty}a_n\frac{z^{2n}}{(2n)!}\frac{1}{z^{2i+2j+4}}
\] 
where $\sim$ means the Laurent series at $z=0$.  Note that the principal part is even so we can replace $\left[\frac{1}{\cos(2\pi z)}\cl\{P\}(z,z)\right]_{z=0}$ by $\left[\frac{1}{\cos(2\pi z)}\cl\{P\}(z,z)\right]^+_{z=0}$ in the statement.
\end{proof}
\begin{lemma}  \label{lapkerR}
For $P(x)$ an odd polynomial:
\[
\cl\left\{\int_0^\infty\hspace{-2mm} dx R(L_1,L_2,x)P(x)\right\}=\left[\frac{1}{\cos(2\pi z_1)}\frac{\cl\{P\}(z_1)}{(z_2-z_1)}\right]_{z_1=0}^+
\]
\end{lemma}
\begin{proof}
Recall that $R(x,y,z)=\frac12 H(z,x+y)+\frac12 H(z,x-y)$ and choose $P=x^{2k+1}$.  Hence
\begin{align*}
\int_0^\infty dxR(L_1,L_2,x)x^{2k+1}&=\frac12F_{2k+1}(L_1+L_2)+\frac12F_{2k+1}(L_1-L_2)\\
&=\sum_{\epsilon=\pm1}\frac12\sum_{m=0}^k(L_1+\epsilon L_2)^{2m+1}\binom{2k+1}{2m+1}a_{k-m}\\
&=(2k+1)!\sum_{m=0}^k\hspace{-1mm}\mathop{\sum_{j\text{\ even}}}_{i+j=2m+1}\hspace{-2mm}\frac{L_1^iL_2^j}{i!j!}\frac{a_{k-m}}{(2k-2m)!}.
\end{align*}
Hence its Laplace transform is:
\[
\cl\left\{\int_0^\infty \hspace{-2mm}dxR(L_1,L_2,x)x^{2k+1}\right\}=(2k+1)!\sum_{m=0}^k\hspace{-1mm}\mathop{\sum_{j\text{\ even}}}_{i+j=2m+1}\hspace{-1mm}\frac{1}{z_1^{i+1}z_2^{j+1}}\frac{a_{k-m}}{(2k-2m)!}
\]
which coincides with the even principal part in $z_1$ of 
\[\frac{1}{\cos(2\pi z_1)}\frac{\cl\{x^{2k+1}\}(z_1)}{(z_2-z_1)}\sim \sum_{n=0}^{\infty}a_n\frac{z_1^{2n}}{(2n)!}\sum_{j=0}^\infty\frac{z_1^j}{z_2^{j+1}}\frac{(2k+1)!}{z_1^{2k+2}}
\]
where $\sim$ means the Laurent series at $z_1=0$ for fixed $z_2$, hence $|z_1|<|z_2|$.
 
 \end{proof}
\noindent Continuing with the proof of Theorem~\ref{thetaspec}, apply Lemmas~\ref{lapkerD} and \ref{lapkerR} to the recursion \eqref{volrec}.
\begin{align}  \label{laprec} 
\cl\left\{L_1T_{g,n}(L_1,L_K)\right\}=&\frac12\cl\left\{\int_0^\infty\hspace{-2mm}\int_0^\infty \hspace{-2mm}xyD(L_1,x,y)P_{g,n+1}(x,y,L_K)dxdy\right.\\
&+
\sum_{j=2}^n\int_0^\infty \hspace{-2mm}xR(L_1,L_j,x)T_{g,n-1}(x,L_{K\setminus\{j\}})dx\Big\}\nonumber\\
=&\frac12\Big[\frac{1}{\cos(2\pi z_1)}\Big(\cl\{xyT_{g-1,n+1}\}(z_1,z_1,z_K)\nonumber\\
+&\hspace{-2mm}\mathop{\sum_{g_1+g_2=g}}_{I \sqcup J = K}\hspace{-1mm}\cl\{xT_{g_1,|I|+1}\}(z_1,z_I)\cl\{yT_{g_2,|J|+1}\}(z_1,z_J)\Big)\Big]_{z_1=0}^+   \nonumber\\
&+\sum_{j=2}^n\left[\frac{1}{\cos(2\pi z_1)}\frac{\cl\{xT_{g,n-1}\}(z_1,z_{K\setminus\{ j\}})}{z_j-z_1}\right]_{z_1=0}^+  \nonumber.
\end{align}
The principal part of the term involving $D$ coincides with its even principal part, as explained in the note at the end of the proof of Lemma~\ref{lapkerD}, so we have written it as the even part.

Define 
\[\Omega_{g,n}=(-1)^n\frac{\partial}{\partial z_1}...\frac{\partial}{\partial z_n}\cl\{T_{g,n}(L_1,...,L_n)\}dz_1...dz_n.\] 
We will prove that $\Omega_{g,n}$ and the correlators $\omega_{g,n}$ satisfy the same recursion relations and initial values, and in particular conclude that $\Omega_{g,n}=\omega_{g,n}$.)  

Take $(-1)^{n-1}\frac{\partial}{\partial z_2}...\frac{\partial}{\partial z_n}\Big[$\eqref{laprec}$\Big]dz_1...dz_n$, noting that $-\frac{\partial}{\partial z_1}$ is already present since $\cl\{L_1P(L_1)\}=-\frac{\partial}{\partial z_1}\cl\{P\}(z_1)$,
 to get
\begin{align} 
\Omega_{g,n}(z_1,z_K)=&\frac12\left[\frac{1}{\cos(2\pi z_1)dz_1}\Omega_{g-1,n+1}(z_1,z_1,z_K)\right]_{z_1=0}^-\\
+&\frac12\Big[\frac{1}{\cos(2\pi z_1)dz_1}\mathop{\sum_{g_1+g_2=g}}_{I \sqcup J = K}\hspace{-1mm}\Omega_{g_1,|I|+1}(z_1,z_I)\Omega_{g_2,|J|+1}(z_1,z_J)\Big]_{z_1=0}^-
\nonumber\\
&+\sum_{j=2}^n\left[\frac{1}{\cos(2\pi z_1)}\frac{\Omega_{g,n-1}(z_1,z_{K\setminus\{ j\}})}{(z_j-z_1)^2}\right]_{z_1=0}^-.
\nonumber
\end{align}
The even part of the principal part becomes the odd part $[\cdot]^+\to [\cdot]^-$ due to the factor of $dz_1$.  The factors $xy$, $x$ and $y$  on the right hand side of \eqref{laprec} supply derivatives such as $\cl\{xyT_{g-1,n+1}\}(z_1,z_1,z_K)=\frac{\partial^2}{\partial w\partial z}\cl\{T_{g-1,n+1}\}(w\hspace{-1mm}=\hspace{-1mm}z_1,z\hspace{-1mm}=\hspace{-1mm}z_1,z_K)$. 

Topological recursion for the spectral curve $S$ is 
\begin{align*}
\omega_{g,n}(z_1,z_K)&=\Res_{z=0}K(z_1,z) \cf(\{\omega_{g',n'}(z,z_K)\})dzdzdz_K\\
&=-\frac12\Res_{z=0}\left(\frac{dz_1}{z_1-z}-\frac{dz_1}{z_1+z}\right)\frac{1}{2\cos(2\pi z)} \cf(\{\omega_{g',n'}(z,z_K)\})dzdz_K\\
&=-\frac12\left[\frac{1}{\cos(2\pi z_1)} \cf(\{\omega_{g',n'}(z_1,z_K)\})dz_1dz_K\right]_{z_1=0}^-
\end{align*}
where $\cf(z_1,z_K)$ is a rational function given explicitly in \eqref{EOrec} by
\begin{align*}
\cf(z_1,z_K)dz_1^2dz_K=&\omega_{g-1,n+1}(z_1,-z_1,p_L)+ \hspace{-2mm}\mathop{\sum_{g_1+g_2=g}}_{I\sqcup J=L}^{\rm stable} \omega_{g_1,|I|+1}(z_1,z_I) \, \omega_{g_2,|J|+1}(-z_1,z_J) \\
+\sum_{j=2}^n&\big(\omega_{0,2}(z_1,z_j) \, \omega_{g,n-1}(-z_1,z_{K\setminus\{ j\}})+\omega_{0,2}(-z_1,z_j) \, \omega_{g,n-1}(z_1,z_{K\setminus\{ j\}})\big)\\
=&-\omega_{g-1,n+1}(z_1,z_1,p_L)- \hspace{-2mm}\mathop{\sum_{g_1+g_2=g}}_{I\sqcup J=L}^{\rm stable} \omega_{g_1,|I|+1}(z_1,z_I) \, \omega_{g_2,|J|+1}(z_1,z_J) \\
&-\sum_{j=2}^n\big(\omega_{0,2}(z_1,z_j)-\omega_{0,2}(-z_1,z_j)\big)\omega_{g,n-1}(z_1,z_{K\setminus\{ j\}})
\end{align*}
where we have used skew-symmetry of $\omega_{g,n}$ under $z_i\mapsto-z_i$, except for $\omega_{0,2}$.  Hence
\begin{align*} 
\omega_{g,n}(z_1,z_K)=&\frac12\left[\frac{1}{\cos(2\pi z_1)dz_1}\omega_{g-1,n+1}(z_1,z_1,z_K)\right]_{z_1=0}^-\\
&+\frac12\Big[\frac{1}{\cos(2\pi z_1)dz_1}\mathop{\sum_{g_1+g_2=g}}_{I \sqcup J = K}^{\rm stable}\hspace{-1mm}\omega_{g_1,|I|+1}(z_1,z_I)\omega_{g_2,|J|+1}(z_1,z_J)\Big]_{z_1=0}^-\\
&+\sum_{j=2}^n\int_0^\infty\left[\frac{1}{\cos(2\pi z_1)}\frac{\omega_{g,n-1}(z_1,z_{K\setminus\{ j\}})}{(z_j-z_1)^2}\right]_{z_1=0}^-.
\end{align*}
where we have used  $[\omega_{0,2}(-z_1,z_j)\eta(z_1)]_{z_1=0}^-=-[\omega_{0,2}(z_1,z_j)\eta(z_1)]_{z_1=0}^-$ for $\eta(z_1)$ odd.

The rational differentials $\Omega_{g,n}$ and $\omega_{g,n}$ are uniquely determined by their respective recursions and the initial value
\[\Omega_{1,1}(z_1)=-\frac{\partial}{\partial z_1}\cl\{T_{1,1}(L_1)\}dz_1=-\frac{\partial}{\partial z_1}\cl\{\tfrac18\}dz_1=\frac{dz}{8z^2}=\omega_{1,1}(z_1)\]
which both coincide, hence $\Omega_{g,n}=\omega_{g,n}$ as required.
\end{proof}
\begin{corollary}
Theorem~\ref{main} holds, i.e. $V^{\Theta}_{g,n}$ is uniquely determined by $V^{\Theta}_{1,1}(L_1)=\frac18$ and the recursion
\eqref{volrec}.
\end{corollary}
\begin{proof}
The proof is immediate from Theorem~\ref{spectral} and Theorem~\ref{thetaspec}. 
\end{proof}

\begin{remark}
Rewrite the expression for $F^M_{2k+1}(t)=\int_0^\infty x^{2k+1}H^M(x,t)dx$ due to Mirzakhani as:
\[\frac{F^M_{2k+1}(t)}{(2k+1)!}=\sum_{i=0}^{k+1}\zeta(2i)(2^{2i+1}-4)\frac{t^{2k+2-2i}}{(2k+2-2i)!}=\sum_{i=0}^{k+1}b_i\frac{t^{2k+2-2i}}{(2k+2-2i)!}.
\]
where $b_n$ is defined by $\displaystyle\frac{2\pi}{\sin(2\pi z)}=\sum_{n=0}^\infty b_nz^{2n-1}$.
Using this, one can replace $D(x,y,z)$ and $R(x,y,z)$ by $\tfrac{\partial}{\partial x}D^M(x,y,z)$ and $\tfrac{\partial}{\partial x}R^M(x,y,z)$ and replace $\frac{1}{\cos(2\pi z)}$ with $\frac{2\pi}{\sin(2\pi z)}$ in the statements of Lemmas~\ref{lapkerD} and \ref{lapkerR}.  The proofs of these statements appear in the appendix of \cite{EOrWei}, using a different approach.  The viewpoint here shows that the spectral curve $x=\frac12z^2$, $y=\frac{\sin(2\pi z)}{2\pi}$ studied by Eynard and Orantin in \cite{EOrWei} is implicit in Mirzakhani's work.
\end{remark}

Theorem~\ref{thetaspec} and the general property \eqref{dilaton} of topological recursion satisfied by any spectral curve produces another proof of the equation \eqref{dilatonvol} 
\[V^\Theta_{g,n+1}(2\pi i, L_1,...,L_n)=(2g-2+n)V^\Theta_{g,n}(L_1,...,L_n)
\]
which was proven in \ref{cone} using pull-back properties of the cohomology classes $\Theta_{g,n}$.

\subsection{Calculations}  \label{sec:kerprop}

We demonstrate here how to use the recursion \eqref{volrec} and equivalently the recursion \eqref{volrecWP}.  It is clear from its definition \eqref{voltheta} that the function $V^\Theta_{g,n}(L_1,...,L_n)$ is a degree $2g-2$ polynomial in $L_i$ (and degree $g-1$ polynomial in $L_i^2$).  A consequence of Lemma~\ref{polprop} and a change of coordinates shows that this polynomial behaviour also follows from the recursion \eqref{volrec} and elegant properties of the kernels $D(x,y,z)$ and $R(x,y,z)$.

The recursion \eqref{volrec} leads to the following small genus calculations.
The 1-point genus one volume can be calculated using an integral closely related to \eqref{volrec}.
\begin{equation}  \label{vol11} 
2LV^\Theta_{1,1}(L)=\int_0^\infty xD(L,x,x)dx=\int_0^\infty xH(2x,L)dx=\frac14F_1(L)=\frac14 L
\end{equation}
Using \eqref{volrec} we calculate:
\begin{align*}
V^\Theta_{1,n}(L_1,...,L_n)&=\frac{(n-1)!}{8}\\
V^\Theta_{2,n}(L_1,...,L_n)&=\frac{3(n+1)!}{128}\left((n+2)\pi^2+\frac{1}{4}\sum_{i=1}^n L_i^2\right)\\
V^\Theta_{3,n}(L_1,...,L_n)&=\frac{(n+3)!}{2^{16}\cdot 5}\Big(16(n+4)(42n+185)\pi^4+336(n+4)\pi^2\sum_{i=1}^n L_i^2\\
&\hspace{2cm}+25\sum_{i=1}^n L_i^4+84\sum_{i\neq j}^n L_i^2L_j^2\Big).
\end{align*}

\begin{remark}
For a cusped surface corresponding to $L_1=0$, replace the recursion \eqref{volrec} by the limit $L_1\to 0$ of $1/L_1\times$ \eqref{volrec} which replaces the kernels by the limits:
\[ \lim_{x\to 0}\frac{1}{x}D(x,y,z)=\frac{1}{8\pi}\frac{\sinh\frac{y+z}{4}} {\cosh^2\frac{y+z}{4}}
\]
\[ \lim_{x\to 0}\frac{1}{x}R(x,y,z)=\frac{1}{16\pi}\left(-\frac{\sinh\frac{y-z}{4}} {\cosh^2\frac{y-z}{4}}+\frac{\sinh\frac{y+z}{4}} {\cosh^2\frac{y+z}{4}}\right).\]
\end{remark}

\subsubsection{Hyperbolic cone angles}   \label{cone}
One can relax the hyperbolic condition on a representation $\rho:\pi_1\Sigma\to SL(2,\br)$ and allow the image of boundary classes to be elliptic.  The trace of an elliptic element is $\tr h=2\cos(\phi/2)\in(-2,2)$, hence such a boundary class corresponds to a cone of angle $\phi$.  A hyperbolic element with trace $\tr g=2\cosh(L/2))$ corresponds to a closed geodesic of length $L$.  Since $2\cos(\phi/2)=2\cosh(i\phi/2)$, one can interpret a point with cone angle in terms of an imaginary length boundary component, and some formulae generalise by replacing positive real parameters with imaginary parameters.  Explicitly, a cone angle $\phi$ appears by substituting the length $i\phi$ in the volume polynomial.  Mirzakhani's recursion uses a generalised McShane formula \cite{McSRem} on hyperbolic surfaces, which was adapted in \cite{TWZGen} to allow a cone angle $\phi$ that ends up appearing as a length $i\phi$ in such a formula, and hence in the volume polynomial.   The importance of hyperbolic monodromy $g$ is that it gives invertibility of $g-I$ used, for example, in the calculation of the cohomology groups $H^k_{dR}$ of the representation.  Perhaps this condition is required only on the interior and not on the boundary classes.  Regardless of the mechanism of the proofs when cone angles are present, one can evaluate the volume polynomials at imaginary values, and find good behaviour.
\begin{thm}
\begin{equation} \label{eq:dilaton}
V^\Theta_{g,n+1}(2\pi i, L_1,...,L_n)=(2g-2+n)V^\Theta_{g,n}(L_1,...,L_n)
 \end{equation}
\end{thm}
\begin{proof}
Using
\[V^{\Theta}_{g,n}(L_1,...,L_n)=\int_{\overline{\modm}_{g,n}}\hspace{-2mm}\Theta_{g,n}\cdot\exp\left\{2\pi^2\kappa_1+\frac12\sum_{i=1}^n L_i^2\psi_i\right\}\]
the coefficient of $L_1^{2\alpha_1}...L_n^{2\alpha_n}$ in $V^\Theta_{g,n+1}(2\pi i, L_1,...,L_n)$ is
\begin{align*}
\sum_{j=0}^m\frac{(2\pi i)^{2j}2^{-|\alpha|-j}}{\alpha!j!(m-j)!}&\int_{\overline{\modm}_{g,n+1}}\hspace{-5mm}\Theta_{g,n+1}\psi^{\alpha}\psi_{n+1}^j(2\pi^2\kappa_1)^{m-j}\\
&=
\int_{\overline{\modm}_{g,n+1}}\hspace{-5mm}\Theta_{g,n+1}\frac{\psi^{\alpha}}{\alpha!}\frac{2^{-|\alpha|}}{m!}\sum_{j=0}^m\binom{m}{j}(-1)^{j}(2\pi^2\psi_{n+1})^{j}(2\pi^2\kappa_1)^{m-j}\\
&=\int_{\overline{\modm}_{g,n+1}}\hspace{-5mm}\Theta_{g,n+1}\frac{\psi^{\alpha}}{\alpha!}\frac{2^{-|\alpha|}}{m!}(2\pi^2\kappa_1-2\pi^2\psi_{n+1})^{m}\\
&=\int_{\overline{\modm}_{g,n+1}}\hspace{-5mm}\Theta_{g,n+1}\frac{\psi^{\alpha}}{\alpha!}\frac{2^{-|\alpha|}}{m!}(2\pi^2\pi^*\kappa_1)^{m}\\
&=\int_{\overline{\modm}_{g,n+1}}\hspace{-5mm}\psi_{n+1}2^{-|\alpha|}\pi^*\big(\Theta_{g,n}\frac{\psi^{\alpha}}{\alpha!}\frac{(2\pi^2\kappa_1)^{m}}{m!}\big)\\
&=(2g-2+n)2^{-|\alpha|}\int_{\overline{\modm}_{g,n}}\hspace{-2mm}\Theta_{g,n}\frac{\psi^{\alpha}}{\alpha!}\frac{(2\pi^2\kappa_1)^{m}}{m!}
\end{align*}
which is exactly $2g-2+n$ times the coefficient of $L_1^{2\alpha_1}...L_n^{2\alpha_n}$ in $V^\Theta_{g,n}$.

\end{proof}

For $g>1$, the integrals 
\[V^\Theta_{g,0}=\int_{\overline{\modm}_{g}}\Theta_{g}\cdot\exp\left\{2\pi^2\kappa_1\right\}
\]
which give the super volumes
\[ V^{SW}_{g,0}=2^{1-g} V^\Theta_{g,0}
\]
do not arise out of the recursion \eqref{volrec}.  Nevertheless, setting $n=0$ in \eqref{eq:dilaton} allows one to calculate these integrals from $V^\Theta_{g,1}(L)$ which do arise out of the recursion \eqref{volrec}
\[V^\Theta_{g,1}(2\pi i)=(2g-2)V^\Theta_{g,0}.\]

Analogous results were proven in \cite{DNoWei} for the Weil-Petersson volumes.  
\begin{thm}[\cite{DNoWei}] \label{th:string}
For ${\bf L}=(L_1,...,L_n)$
\[
 V^{WP}_{g,n+1}({\bf L},2\pi i)=\sum_{k=1}^n\int_0^{L_k}L_kV^{WP}_{g,n}({\bf L})dL_k
\]
and
\[
\frac{\partial V^{WP}_{g,n+1}}{\partial L_{n+1}}({\bf L},2\pi i)=2\pi i(2g-2+n)V^{WP}_{g,n}({\bf L}).
\]
\end{thm}
It is interesting that \eqref{eq:dilaton} does not require a derivative whereas the analogous result in Theorem~\ref{th:string} involves a derivative.  This feature resembles the relations between the kernels for recursions between super volumes $D(x,y,z)=H(y+z,x)$, and between Weil-Petersson volumes $\frac{\partial}{\partial x}D^M(x,y,z)=H^M(y+z,x)$, and similarly for $R(x,y,z)$ and $R^M(x,y,z)$, where the Weil-Petersson volumes again require a derivative.

\subsubsection{} For a given genus $g$, $V^\Theta_{g,g-1}(L_1, ..., L_{g-1})$ determines all the polynomials $V^\Theta_{g,n}(L_1, ..., L_n)$ as follows.  When $n<g-1$ use \eqref{eq:dilaton} to produce $V^\Theta_{g,n}(L_1, ..., L_n)$ from $V^\Theta_{g,g-1}(L_1, ..., L_{g-1})$.   When $n\geq g$, $V^\Theta_{g,n}(L_1, ..., L_n)$, which is a degree $g - 1$ symmetric polynomial in $L^2_1,...,L^2_n$, is uniquely determined by evaluation at $L_n = 2\pi i$, and this is determined by $V_{g,n-1}(L_1,...,L_{n-1})$ via \eqref{dilaton}. This follows from the elementary fact that a symmetric polynomial $f(x_1,...,x_n)$ of degree less than $n$ is uniquely determined by evaluation of one variable at any $a\in\bc$, $f(x_1,...,x_{n=1},a)$. To see this, suppose otherwise. Any symmetric $g(x_1,...,x_n)$ of degree less than $n$ that evaluates at $a$ as $f$ does, satisfies
\begin{align*}
f(x_1,...,x_{n-1},a)&=g(x_1,...,x_{n-1},a) = (x_n -a)P(x_1,...,x_n) \\
&= Q(x_1,...,x_n)\prod_{j=1}^n(x_j -a)
\end{align*}
but the degree is less than $n$ so the difference is identically 0.

\section{Conclusion}

In this paper, we gave an algebraic-geometric proof of a recursion formula for the volumes of moduli spaces of super hyperbolic surfaces, originally derived using supergeometric methods by Stanford and Witten.  This was achieved by relating the volumes of moduli spaces of super Riemann surfaces to integrals over the Deligne-Mumford moduli space of stable Riemann surfaces $\overline{\mathcal M}_{g,n}$.  We applied a Givental type factorisation of a partition function storing these integrals, via topological recursion, which showed that the recursion between volumes is equivalent to the statement that a generating function for the intersection numbers of a natural family of cohomology classes $\Theta_{g,n}$ with tautological classes on $\overline{\mathcal M}_{g,n}$ is a KdV tau function. This approach was directly analogous to Mirzakhani's proof of the Kontsevich-Witten theorem, which established the KdV property of the generating function for intersection numbers of tautological classes via volumes of moduli spaces of hyperbolic surfaces.

It would be desirable to develop a fully supergeometric proof of these results, filling gaps in the original arguments of Stanford and Witten. Such a proof would help clarify the geometric origin of the recursion and should shed light on the recently observed similar recursive behaviour of the more general volumes that allow Ramond punctures.

\end{document}